\documentclass[12pt, a4paper]{article}

\usepackage{graphicx}
\usepackage[margin=3cm]{geometry}

\usepackage{authblk}

\usepackage{amsmath, amssymb}
\usepackage{bm}
\usepackage{amsthm}
\newtheorem{theorem}{Theorem}[section]
\newtheorem{proposition}[theorem]{Proposition}
\newtheorem{lemma}[theorem]{Lemma}
\newtheorem{corollary}[theorem]{Corollary}
\theoremstyle{definition}
\newtheorem{definition}[theorem]{Definition}
\newtheorem{remark}[theorem]{Remark}

\makeatletter
	
	\@addtoreset{equation}{section}
\makeatother

\renewcommand{\theenumi}{\roman{enumi}}
\renewcommand{\labelenumi}{{\rm (\theenumi)}}

\title{Loewner chains and evolution families \\ on parallel slit half-planes}
\author{Takuya Murayama%
\footnote{The present affiliation is Faculty of Mathematics, Kyushu University, 744 Motooka, Nishi-ku, Fukuoka 819-0395, Japan. %
E-mail address: \texttt{murayama@math.kyushu-u.ac.jp}}}
\affil{\small Department of Mathematics, Kyoto University, Kyoto 606-8502, Japan. \\%
Research Fellow of Japan Society for the Promotion of Science.} 
\date{}

\begin{document}

\maketitle

\begin{abstract}
In this paper,
we define and study Loewner chains and evolution families
on finitely multiply-connected domains in the complex plane.
These chains and families consist of conformal mappings
on parallel slit half-planes and
have one and two ``time'' parameters, respectively.
By analogy with the case of simply connected domains,
we develop a general theory of Loewner chains and evolution families
on multiply connected domains and,
in particular,
prove that they obey the chordal Komatu--Loewner differential equations
driven by measure-valued processes.
Our method involves Brownian motion with darning,
as do some recent studies.


\end{abstract}

\section{Introduction}
\label{sec:intro}

In geometric function theory,
\emph{Loewner chains} and \emph{evolution families}
have been known as useful tools
to study the property of univalent (i.e., holomorphic and injective) functions
on the unit disk $\mathbb{D}$ in the complex plane $\mathbb{C}$.
They are families of univalent functions with ``time'' parameter(s)
and obey the \emph{Loewner differential equations}.
(Some classical references are available on Loewner's method,
for example, Pommerenke~\cite[Chapter~6]{Po75}.)
A famous application of this equation
is the proof of Bieberbach's conjecture given
by de Branges~\cite{dB85} in 1985.
In 2000,
Schramm~\cite{Sc00} employed the Loewner equation
on the upper half-plane
$\mathbb{H} := \{\, z \in \mathbb{C} \mathrel{;} \Im z > 0 \,\}$
to define the \emph{stochastic (Schramm--)Loewner evolution}
(abbreviated as SLE),
which describes phase interfaces
in the scaling limits of random planar lattice models in statistical physics.
After his work,
many applications of Loewner's method have been found in probability theory.

In this paper,
we present natural definitions
of Loewner chains and evolution families
\emph{on finitely multiply-connected domains} in $\mathbb{C}$
and develop a general theory on these families.
In particular, we deduce their evolution equations,
\emph{Komatu--Loewner differential equations}.
Our results will thus extend classical ones
in terms of the connectivity of underlying domains.
Let us first review the relevant studies that motivate our work.

Let $D \subset \mathbb{H}$ be a domain
whose complement $\mathbb{H} \setminus D$ is the union
of $N (\geq 1)$ disjoint line segments parallel to the real axis.
(We assume that none of the segments reduces to a point.)
Such a domain is called a \emph{parallel slit half-plane}.
It is known that any non-degenerate%
\footnote{A multiply connected domain $E$
is said to be non-degenerate
if none of the connected components of $\hat{\mathbb{C}} \setminus E$
reduces to a point.
Here, $\hat{\mathbb{C}}$ is the Riemann sphere.}
$(N+1)$-connected domain can be mapped conformally
onto some parallel slit half-plane with $N$ slits;
see Courant~\cite[Theorem~2.3]{Co50} for example.
For this reason,
we choose parallel slit half-planes
as a canonical form of multiply connected domains in this paper.

The Komatu--Loewner equation was studied
for one-parameter families of \emph{slit-mappings}
in a series of papers~\cite{BF04, BF06, BF08} written by Bauer and Friedrich.
By definition, a slit%
\footnote{This ``slit'' is different from a parallel ``slit'' of $D$.}
of $D$ is a simple curve
$\gamma \colon [0, T) \to \overline{D}$
with $\gamma(0) \in \partial \mathbb{H}$ and $\gamma(0, T) \subset D$.
For each $t \in (0, T)$,
there exists a unique conformal mapping $g_t$
that maps the $(N+1)$-connected domain $D \setminus \gamma(0, t]$
onto a parallel slit half-plane $D_t$
under the hydrodynamic normalization
\begin{equation} \label{eq:intro_hyd_norm}
g_t(z) = z + \frac{a_t}{z} + o\left(z^{-1}\right)
\quad (z \to \infty)
\quad \text{with}\ a_t > 0.
\end{equation}
Such a mapping $g_t$ is called a slit-mapping
(or the mapping-out function of the set $\gamma(0, t]$), and
the coefficient $a_t$ of $z^{-1}$ is called
the \emph{half-plane capacity} of
$\gamma(0,t]$ in $D$.
The function $t \mapsto a_t$ turns out to be increasing and continuous,
and hence, we may assume $a_t = 2t$ by reparametrization.
For such a family $(g_t)_{t \in [0, T)}$,
Bauer and Friedrich derived
the chordal Komatu--Loewner equation%
\footnote{As we adopt the notation
of Chen, Fukushima and Rohde~\cite{CFR16} and of the author~\cite{Mu19spa},
the kernel $\Psi_{D_t}$ in \eqref{eq:intro_1KL} differs
from the original $\Psi_t$ in \cite[Eq.~(18)]{BF08}
by a multiplicative constant $2\pi$.}%
~\cite[Theorem~3.1]{BF08}
\begin{equation} \label{eq:intro_1KL}
\frac{\partial g_t(z)}{\partial t}
= - 2\pi \Psi_{D_t}(g_t(z), \xi(t)),
\quad z \in D.
\end{equation}
Here,
$\xi(t) := \lim_{z \to \gamma(t)} g_t(z) (\in \partial \mathbb{H})$
is the image of the tip $\gamma(t)$
and called the \emph{driving function}.

The kernel $\Psi_{D_t}(z, \xi)$
in the right-hand side of \eqref{eq:intro_1KL}
is given as follows~\cite[Section~2.2]{BF08}:
for a parallel slit half-plane $D$,
let $G_D(z, \zeta)$ be ($\pi^{-1}$ times) the Green function,
$\Phi_D(z) = (\varphi_D^{(1)}(z), \ldots, \varphi_D^{(N)}(z))$
be the harmonic basis, and
$\bm{A}_D = (a_{ij})_{i,j=1, \ldots, N}$ be
the matrix of the periods $a_{ij}$ of $\varphi_D^{(j)}(z)$.
For each fixed $\xi \in \partial \mathbb{H}$, the harmonic function
\begin{equation} \label{eq:def_BMD_Poisson}
z \mapsto K^{\ast}_D(z, \xi)
:= -\frac{1}{2} \frac{\partial}{\partial \bm{n}_{\xi}} G_D(z, \xi)
	- \Phi_D(z) \bm{A}_D^{-1} \frac{\partial}{\partial \bm{n}_{\xi}}
		\Phi_D(\xi)^{\mathrm{tr}}
\end{equation}
is free from periods, and
there exists a unique holomorphic function $\Psi_D(z, \xi)$
with $\Im \Psi_D = K^{\ast}_D$ and
$\lim_{z \to \infty} \Psi_D(z, \xi) = 0$.
Here, $\bm{n}_{\xi}$ is the unit normal of $\partial \mathbb{H}$ at $\xi$
pointing downward.
This is a classical way to construct a conformal mapping
from a given domain onto a parallel slit domain
(cf.\ Section~5, Chapter~6 of Ahlfors~\cite{Ah79}).
Indeed, $\Psi_D(\mathord{\cdot}, \xi)$ is a conformal mapping
from $D$ onto another parallel slit half-plane
that sends $\infty$ to $0$ and has the Laurent expansion of the form
$\Psi_D(z,\xi)=-\pi/(z-\xi)+O(1)$ at the point $\xi$.
Since $K^\ast_D(z,\xi)$ vanishes if $z \in \partial\mathbb{H}\setminus\{\xi\}$, the function $z \mapsto \Psi_D(z, \xi)$ is extended holomorphically across $\partial\mathbb{H}\setminus\{\xi\}$ by the Schwarz reflection principle.

It is important to us
that $K^{\ast}_D$ has an probabilistic interpretation.
Lawler~\cite{La06} identified $K^{\ast}_D$
with the Poisson kernel
of the \emph{excursion reflected Brownian motion}
(ERBM for short) on $D$.
Motivated by his study,
Chen, Fukushima and Rohde~\cite{CFR16} identified
$K^{\ast}_D$ with the Poisson kernel
of the \emph{Brownian motion with darning}%
\footnote{ERBM and BMD are essentially the same; see \cite[Remark~2.2]{CFR16}.}
(BMD for brevity) for $D$.
Let us recall how the latter identification is done.

For a parallel slit half-plane $D$
with parallel slits $C_1, \ldots, C_N$,
let $D^{\ast} := D \cup \{ c^{\ast}_1, \ldots, c^{\ast}_N \}$
be the quotient space of $\mathbb{H}$
with each $C_j$ regarded as a single point $c^{\ast}_j$.
The Lebesgue measure $m_{D^{\ast}}$ is naturally defined on $D^{\ast}$
so that it does not charge $\{ c^{\ast}_1, \ldots, c^{\ast}_N \}$.
The BMD on $D^{\ast}$ is an $m_{D^{\ast}}$-symmetric diffusion process
$Z^{\ast} = ((Z^{\ast}_t)_{t \geq 0}, (\mathbb{P}^{\ast}_z)_{z \in D^{\ast}})$
with the following properties:
\begin{itemize}
\item
the part process of $Z^{\ast}$ in $D$ is
the absorbed Brownian motion on $D$;

\item
$Z^{\ast}$ admits no killings on $\{ c^{\ast}_1, \ldots, c^{\ast}_N \}$.
\end{itemize}
By \cite[Lemma~5.1]{CFR16},
(a Borel-measurable version of) the $0$-order resolvent kernel of $Z^{\ast}$
coincides with the generalized Green function
\begin{equation} \label{eq:BMD_Green}
G^{\ast}_D(z, w)
= G_D(z, w) + 2 \Phi_D(z) \bm{A}_D^{-1} \Phi_D(w)^{\mathrm{tr}}.
\end{equation}
We call it the Green function of $Z^{\ast}$ as well.
We have
\begin{equation} \label{eq:Green_Poisson}
K^{\ast}_D(z, \xi)
= -\frac{1}{2} \frac{\partial}{\partial \bm{n}_{\xi}} G^{\ast}_D(z, \xi)
\end{equation}
by definition,
and the identity
\begin{equation} \label{eq:BMD_Poi_bd_value}
\mathbb{E}^{\ast}_z[g(Z^{\ast}_{\zeta^{\ast}-})]
= \int_{\mathbb{R}} K^{\ast}_D(z, \xi) g(\xi) \, d\xi,
\quad g \in C_b(\mathbb{R}),
\end{equation}
holds by \cite[Lemma~5.2]{CFR16}.
Here, we have identified $\partial \mathbb{H}$ with $\mathbb{R}$.
$\zeta^{\ast}$ denotes the lifetime of $Z^{\ast}$.
For these reasons,
we call $K^{\ast}_D$ the Poisson kernel of $Z^{\ast}$.
Accordingly, we call $\Psi_D$ the complex Poisson kernel of BMD.

Now, we compare the above-mentioned studies
on the Komatu--Loewner equation~\eqref{eq:intro_1KL}
with a similar line of research
on the Loewner differential equation
\begin{equation} \label{eq:intro_Loewner}
\frac{\partial g_t(z)}{\partial t}
= \frac{2}{g_t(z) - \xi(t)},
\quad z \in \mathbb{H},
\end{equation}
for a family $(g_t)_{t \geq 0}$ of slit-mappings of $\mathbb{H}$.
In this simply-connected case,
$g_t$ can be replaced by a more general mapping.
Such replacement leads us to
the concepts of Loewner chains and evolution families.

Roughly speaking, a Loewner chain
on $\mathbb{H}$ is
a family of univalent functions
$f_t \colon \mathbb{H} \hookrightarrow \mathbb{H}$, $t \in [0, T]$,
whose images $f_t(\mathbb{H})$ are expanding ``continuously'' in $t$.
Moreover, each $f_t$ fixes the point at infinity and obeys normalization conditions at $\infty$ slightly more general than \eqref{eq:intro_hyd_norm};
see (H.\ref{ass:hydro}) and (H.\ref{ass:res}) in Section~\ref{sec:assumptions} for the precise conditions.
For such $(f_t)_{t \in [0, T]}$,
the family of the composites
$\phi_{t, s} := f_t^{-1} \circ f_s$, $0 \leq s \leq t \leq T$,
is called an evolution family on $\mathbb{H}$.
For each fixed $s \in [0, T)$, the evolution family obeys the equation
\begin{equation} \label{eq:intro_GB}
\frac{\partial \phi_{t, s}(z)}{\partial t}
= - \dot{a}_t \int_{\mathbb{R}} \frac{\nu_t(d\xi)}{\phi_{t, s}(z) - \xi},
\quad z \in \mathbb{H}\ \text{and a.e.}\ t \in [s, T].
\end{equation}
Here, $\nu_t$ is a Borel probability measure on $\mathbb{R}$.
This equation was derived in some papers;
see I.\ A.\ Aleksandrov, S.\ T.\ Aleksandrov and Sobolev~\cite{AAS83},
Goryainov and Ba~\cite{GB92},
and Bauer~\cite{Ba05}.
From \eqref{eq:intro_GB},
one
can also derive the differential equation for $(f_t)_{t \in [0, T]}$.

\eqref{eq:intro_GB} is related to \eqref{eq:intro_Loewner} as follows:
let $(g_t)_{t \in [0, T]}$ be a family of slit-mappings.
Then $f_t := g_{T-t}^{-1}$, $t \in [0, T]$, form a Loewner chain, and
$\phi_{t, s} := g_{T-t} \circ g_{T-s}^{-1}$, $0 \leq s \leq t \leq T$,
form an evolution family.
\eqref{eq:intro_Loewner} implies that
$(\phi_{t, s})_{0 \leq s \leq t \leq T}$ satisfies \eqref{eq:intro_GB}
with $a_t = 2t$ and $\nu_t = \delta_{\xi(T-t)}$.
Here, $\delta_{\xi(T-t)}$ is the Dirac delta measure supported at $\xi(T-t)$.
Thus, we can say that \eqref{eq:intro_GB}
is a (time-reversed) version of \eqref{eq:intro_Loewner}
with the driving function $t \mapsto \xi(T-t)$
replaced by the \emph{measure-valued} driving process $t \mapsto \nu_t$.

As the Loewner equation~\eqref{eq:intro_Loewner} for slit-mappings
is essential to SLE theory,
the equation~\eqref{eq:intro_GB} driven by a measure-valued process
also has applications to probability theory.
For example,
del Monaco and Schlei{\ss}inger~\cite{dMS16}
studied the $n \to \infty$ limit of
\emph{$n$-multiple SLE}~\cite{BBK05, KL07}.
In their paper,
$\xi_1(t), \ldots, \xi_n(t)$
are the Dyson's non-colliding Brownian motions
(which is well known in random matrix theory),
$\nu_t$ is given as the $n \to \infty$ limit
of the empirical measures
$\nu^{(n)}_t := n^{-1} \sum_{k=1}^n \delta_{\xi_k(t)}$, and
$\nu^{(n)}_t \to \nu_t$ in \eqref{eq:intro_GB}
implies the convergence of the corresponding Loewner chains.
(See also del Monaco, Hotta and Schlei{\ss}inger~\cite{dMHS18} and
Hotta and Katori~\cite{HK18}.)
Moreover, a version of \eqref{eq:intro_GB} on $\mathbb{D}$
has appeared
in the study of \emph{Laplacian growth models}.
See, e.g., Johansson Viklund, Sola and Turner~\cite{JVST12}
and Miller and Sheffield~\cite{MS16}.

In contrast to the applications in the simply-connected case,
no studies have been done
on Loewner chains, evolution families, and measure-valued driving processes
associated with the Komatu--Loewner equation
with regard to conformal mappings on $(N+1)$-connected domains with $N \geq 2$.
(The case $N=1$ is special and has been studied on annuli;
see Remark~\ref{rem:annuli}.)
Our study will thus make a contribution to filling the lack of such a study.

Finally, we make some remarks
on the idea and techniques in this paper.
We owe our basic idea to Goryainov and Ba~\cite{GB92}.
They derived \eqref{eq:intro_GB},
using the integral representation
$f(z) = z - \int_{\mathbb{R}} (z - \xi)^{-1} \, \mu(d\xi)$
for a proper class of holomorphic functions $f$ \cite[Lemma~1]{GB92}.
In our case,
$-(z - \xi)^{-1}$ will be replaced by $\pi \Psi_D(z, \xi)$
(see \eqref{eq:int_rep}).
However,
this replacement is far from straightforward.
The difference comes from the dependence of $\Psi_D(z, \xi)$
on $\xi$ and on $D$.
In order to derive both the integral representation
and desired differential equation,
one has to show
that this dependence is controllable in an appropriate sense.
(For the kernel $(z - \xi)^{-1}$, this is obvious.)
Indeed, such dependence was studied systematically
in the derivation of \eqref{eq:intro_1KL}
by Chen, Fukushima and Rohde~\cite{CFR16}.
Since we treat a measure-valued process $\nu_t(d\xi)$
rather than a continuous function $\xi(t)$,
we need to strengthen their results
so that some estimates about $\Psi_D(z, \xi)$
are uniform with respect to $\xi \in \partial \mathbb{H}$.
We shall do this,
combining methods of geometric function theory,
probability theory, and functional analysis.

The rest of this paper is organized as follows.
In Section~\ref{sec:results},
we state basic assumptions,
the definitions of Loewner chains and of evolution families,
and the main results of this paper.
The main results consist of three parts:
deriving the Komatu--Loewner equations
(Theorem~\ref{thm:result_KL_EF} and Corollary~\ref{cor:result_KL_LC}),
constructing evolution families by solving the Komatu--Loewner equation
(Theorems~\ref{thm:result_unbdd} and \ref{thm:result_bdd}), and
deducing an integral representation formula for conformal mappings
(Theorem~\ref{thm:int_rep}).
The third part is itself a main tool to prove the former two.
The proof of these results are given
in Sections~\ref{sec:Poisson} through \ref{sec:proof3}.
In Section~\ref{sec:Poisson},
we provide some properties of BMD complex Poisson kernel $\Psi_D(z, \xi)$
that are required in various places of this paper.
Section~\ref{sec:proof1} is devoted to the proof of Theorem~\ref{thm:int_rep}.
Section~\ref{sec:continuation} gives a preliminary to
Sections~\ref{sec:proof2} and \ref{sec:proof3}.
In order to describe precisely the behavior of a conformal mapping $f(z)$
as $z$ is near a parallel slit of a slit half-plane,
we perform the analytic continuation of $f(z)$ across the slit
and introduce corresponding symbols.
In Section~\ref{sec:proof2},
we prove Theorem~\ref{thm:result_KL_EF} and Corollary~\ref{cor:result_KL_LC}.
In Section~\ref{sec:proof3},
we prove Theorems~\ref{thm:result_unbdd} and \ref{thm:result_bdd}.
Section~\ref{sec:application} involves two applications of our results.
We deduce the Komatu--Loewner equations
for the family of the mapping-out functions of \emph{hulls with local growth}
in Sections~\ref{sec:local_growth} and \ref{sec:RSLC}
and
for the family of \emph{multiple-slit} mappings
in Section~\ref{sec:multiple_slits}.
Finally, we conclude this paper with some remarks for future works
in Section~\ref{sec:conclusion}.
After the body ends,
we have four appendices.
In Appendix~\ref{sec:KL_slit},
we deduce the evolution equation for the parallel slits
of the family $(D_t)_{t \in [0, T)}$ associated with an evolution family.
(Solutions to this equation are studied in Section~\ref{sec:proof3}.)
Appendices~\ref{sec:at_infinity} and \ref{sec:parameter} collects
some consequences from the assumptions
formulated in Section~\ref{sec:assumptions}.
In Appendix~\ref{sec:top_on_meas},
we recall the definition and basic properties of vague and weak convergences for finite Borel measures.
They will play a role when we treat measure-valued driving processes.

\section{Main results}
\label{sec:results}

\subsection{Assumptions}
\label{sec:assumptions}

We use the symbols
\begin{align*}
\mathbb{H}_{\eta}
&:= \{\, z \in \mathbb{C} \mathrel{;} \Im z > \eta \,\},
&\eta \geq 0, \\
\triangle_{\theta}
&:= \{\, z \in \mathbb{C} \mathrel{;} \theta < \arg z < \pi - \theta \,\},
&\theta \in (0, \pi/2).
\end{align*}
The upper half-plane $\mathbb{H}_0$ is written simply as $\mathbb{H}$.

In what follows,
we state three assumptions~(H.\ref{ass:hydro})--(H.\ref{ass:hull})
on univalent functions
and one assumption~$\mathrm{(Lip)}_F$
on one-parameter families of holomorphic functions.
We shall work under these assumptions throughout this paper.

The assumptions~(H.\ref{ass:hydro}) and (H.\ref{ass:res})
concern the behavior of functions around the point at infinity.
Let $f$ be a univalent function whose domain of definition contains
a half-plane $\mathbb{H}_{\eta_0}$, $\eta_0 \geq 0$.
The function $f$ is said to be
\emph{hydrodynamically normalized at infinity} if
\begin{enumerate}
\renewcommand{\theenumi}{\arabic{enumi}}
\renewcommand{\labelenumi}{{\rm (H.\theenumi)}}

\item \label{ass:hydro}
$\displaystyle
\lim_{\substack{z \to \infty \\ z \in \mathbb{H}_{\eta}}} (f(z) - z) = 0$
for some $\eta \geq \eta_0$.
\end{enumerate}
Whenever we say that $f$ satisfies (H.\ref{ass:hydro}),
we implicitly assume that the domain of $f$ contains some half-plane.
In addition, suppose that
\begin{enumerate}
\setcounter{enumi}{1}
\renewcommand{\theenumi}{\arabic{enumi}}
\renewcommand{\labelenumi}{{\rm (H.\theenumi)}}

\item \label{ass:res}
there exists $c \in \mathbb{C}$ such that
\[
\lim_{\substack{z \to \infty \\ z \in \triangle_{\theta}}} z (f(z) - z) = - c
\quad \text{for any}\ \theta \in (0, \pi/2).
\]
\end{enumerate}
Then we call $c$ the \emph{angular residue} of $f$ at infinity.
These properties are preserved
by taking the inverse or composition of functions;
see Appendix~\ref{sec:at_infinity}.

Let $D_1$ and $D_2$ be parallel slit half-planes
and $f \colon D_1 \to D_2$ be a univalent function.
The following assumption says
that each slit of $D_1$ is ``mapped'' onto a slit of $D_2$ by $f$:
\begin{enumerate}
\setcounter{enumi}{2}
\renewcommand{\theenumi}{\arabic{enumi}}
\renewcommand{\labelenumi}{{\rm (H.\theenumi)}}

\item \label{ass:hull}
the complement $F = D_2 \setminus f(D_1)$ of the image $f(D_1)$ is
an $\mathbb{H}$-hull.
\end{enumerate}
Here, $F$ is called an $\mathbb{H}$-hull
if it is relatively closed in $\mathbb{H}$ and
if $\mathbb{H} \setminus F$ is a simply connected domain.
Although $\mathbb{H}$-hulls are assumed to be bounded in some cases,
we do not assume the boundedness here.
If (H.\ref{ass:hydro}) and (H.\ref{ass:res}) as well as (H.\ref{ass:hull}) hold,
then the angular residue $c$ of $f$ at infinity is non-negative by Theorem~\ref{thm:int_rep} below.
When we focus on the $\mathbb{H}$-hull $F$ rather than the function $f$,
we call the \emph{(BMD) half-plane capacity}%
\footnote{This definition is well-defined because,
for a given $\mathbb{H}$-hull $F \subset D_2$,
a corresponding conformal mapping
$f \colon D_1 \to D_2 \setminus F$
with (H.\ref{ass:hydro}) and (H.\ref{ass:res}) is unique
by Corollary~\ref{cor:uniqueness}.}
of $F$ in $D_2$
and denote it by $\operatorname{hcap}^{D_2}(F)$.
(In this case,
$f^{-1} \colon D_2 \setminus F \to D_1$ is also called
the \emph{mapping-out function} of $F$,
as previously mentioned in Section~\ref{sec:intro}.)

Finally, we state the assumption $\mathrm{(Lip)}_F$
on one-parameter families of holomorphic functions.
Let $I$ be an interval,
$F \colon I \to \mathbb{R}$ be a non-decreasing continuous function,
and $f_t$ be a holomorphic function on a Riemann surface $X$ for each $t \in I$.
The following condition says that
the family $(f_t)_{t \in I}$ is Lipschitz continuous with respect to $F(t)$
locally uniformly in $p \in X$:
\begin{itemize}

\item[$\mathrm{(Lip)}_F$]
For any compact subset $K$ of $X$,
there exists a constant $L_K$ such that
\[
\sup_{p \in K} \lvert f_t(p) - f_s(p) \rvert \leq L_K (F(t) - F(s))
\quad \text{for} \ (s, t) \in I^2 \ \text{with} \ s < t.
\]

\end{itemize}

A series of results related to this condition is presented
in Appendix~\ref{sec:parameter}.
Here we give minimum prerequisites for stating our results precisely in Section~\ref{sec:derivation}.
Let $m_F$ be a unique non-atomic Radon measure on $I$
that satisfies $m_F((s, t]) = F(t) - F(s)$.
Under $\mathrm{(Lip)}_F$,
the function $t \mapsto f_t(p)$
is of finite variation on $I$ for each $p \in X$ and
hence induces a complex measure $\kappa_p$
on every compact subinterval of $I$
that is absolutely continuous with respect to $m_F$.
The limit
\[
\tilde{\partial}^F_t f_t(p)
:= \lim_{\delta \downarrow 0}
\frac{f_{t+\delta}(p) - f_{t-\delta}(p)}{F(t+\delta) - F(t-\delta)}
= \lim_{\delta \downarrow 0}
\frac{\kappa_p((t-\delta, t+\delta))}{m_F((t-\delta, t+\delta))}
\]
exists for $m_F$-a.e.\ $t \in I$ and
is a version of the Radon--Nikodym derivative $d\kappa_p/dm_F$.
Although the set of $t$ for which $\tilde{\partial}^F_t f_t(p)$ does not exist depends on $p \in X$,
Proposition~\ref{prop:ae_diff} ensures that
$\{\, t \in I \mathrel{;}
\tilde{\partial}^F_t f_t(p) \ \text{does not exist for some}\ p \in X \,\}$
is $m_F$-negligible.

\subsection{Komatu--Loewner equations for evolution families and for Loewner chains}
\label{sec:derivation}

From this point,
the symbol $I$ stands for an interval $[0, T)$ or $[0, T]$ 
with $T\in(0,\infty]$.
If $T=\infty$, then $I$ is always understood as the half-line $[0,\infty)$; we do not consider the case $\infty\in I$.
We define
\[
I^2_{\leq} := \{\, (s, t) \in I^2 \mathrel{;} s \leq t \,\},
\quad
I^3_{\leq} := \{\, (s, t, u) \in I^3 \mathrel{;} s \leq t \leq u \,\}
\]
and the same symbols with $\leq$ replaced by $<$ in an obvious manner.

\begin{definition} \label{dfn:evol_family}
Let $D_t$, $t \in I$, be parallel slit half-planes.
We say that a two-parameter family of univalent functions
$\phi_{t, s} \colon D_s \to D_t$, $(s, t) \in I^2_{\leq}$,
with (H.\ref{ass:hydro})--(H.\ref{ass:hull})
is a \emph{(chordal) evolution family over $(D_t)_{t \in I}$}
if the following hold:
\begin{enumerate}
\renewcommand{\theenumi}{\arabic{enumi}}
\renewcommand{\labelenumi}{{\rm (EF.\theenumi)}}

\item \label{cond:EFnorm}
$\phi_{t, t}$ is the identity mapping for each $t \in I$;

\item \label{cond:EFcomp}
$\phi_{u, s} = \phi_{u, t} \circ \phi_{t, s}$ holds on $D_s$ for each $(s, t, u) \in I^3_{\leq}$;

\item \label{cond:EFcont}
the angular residue $\lambda(t)$ of $\phi_{t, 0}$ at infinity is continuous in $t \in I$.
\end{enumerate}
\end{definition}

$\lambda(t)$ is non-decreasing 
by Lemma~\ref{lem:evol_family}~\eqref{lem:EFhcap},
and the one-parameter family $(\phi_{t, t_0})_{t \in I \cap [t_0, T]}$ enjoys
$\mathrm{(Lip)}_{\lambda}$ on $D_{t_0}$
for each fixed $t_0 \in I$
by Lemma~\ref{lem:EFineq}.

We now state the first part of our results.
For a topological space $X$,
let $\mathcal{B}(X)$ and $\mathcal{B}^m(X)$ be
the Borel $\sigma$-algebra of $X$
and its completion with respect to a measure $m$, respectively.
Let $\mathcal{P}(\mathbb{R})$ be the set of Borel probability measures on $\mathbb{R}$.
We can consider the vague or weak topology on $\mathcal{P}(\mathbb{R})$,
but either of them gives rise to the same Borel $\sigma$-algebra by Corollary~\ref{cor:Borel_on_M}.
See Appendix~\ref{sec:top_on_meas} for the definition and properties of vague and weak topologies.

\begin{theorem} \label{thm:result_KL_EF}
Let $(\phi_{t, s})_{(s, t) \in I^2_{\leq}}$ be an evolution family over $(D_t)_{t \in I}$.
There exist an $m_{\lambda}$-null set $N_0$ and
a measurable mapping $t \mapsto \nu_t$
from $(I, \mathcal{B}^{m_{\lambda}}(I))$ to
$(\mathcal{P}(\mathbb{R}), \mathcal{B}(\mathcal{P}(\mathbb{R})))$
such that,
for each fixed $t_0 \in [0, T)$,
the Komatu--Loewner equation
\begin{equation} \label{eq:result_KL_EF}
\tilde{\partial}^{\lambda}_t \phi_{t, t_0}(z)
= \pi \int_{\mathbb{R}} \Psi_{D_t}(\phi_{t, t_0}(z), \xi) \, \nu_t(d\xi),
\quad z \in D_{t_0},
\end{equation}
holds for any $t \in (t_0, T) \setminus N_0$.
Moreover, such a mapping $t \mapsto \nu_t$ is unique on $(0, T) \setminus N_0$.
\end{theorem}

Since $(\phi_{t,t_0})_{t \in I \cap [t_0,T]}$ enjoys $\mathrm{(Lip)}_{\lambda}$ as mentioned above,
we can integrate both sides of \eqref{eq:result_KL_EF} by the measure $m_\lambda$ to get the integral equation equivalent to the differential equation \eqref{eq:result_KL_EF}.
Here, we remind the reader that, also in the usual theory of ordinary differential equations (ODEs for short),
it is important that a solution to an ODE is an absolutely continuous function by the very definition and, due to this absolute continuity, the ODE is equivalent to its integrated form.

The equation~\eqref{eq:result_KL_EF} for evolution families
is easily transferred to the one for Loewner chains,
whose definition is given as follows:

\begin{definition} \label{dfn:LoewnerChain}
Let $D$ and $D_t$, $t \in I$, be parallel slit half-planes.
We say that a family of univalent functions $f_t \colon D_t \to D$, $t \in I$,
with (H.\ref{ass:hydro})--(H.\ref{ass:hull}) is a
\emph{(chordal) Loewner chain over $(D_t)_{t \in I}$ with codomain $D$}
if the following hold:
\begin{enumerate}
\renewcommand{\theenumi}{\arabic{enumi}}
\renewcommand{\labelenumi}{{\rm (LC.\theenumi)}}

\item \label{cond:LCsub}
$f_s(D_s) \subset f_t(D_t)$ holds for each $(s, t) \in I^2_{\leq}$;

\item \label{cond:LCcont}
the angular residue $\ell(t)$ of $f_t$ is continuous in $t \in I$.
\end{enumerate}
\end{definition}

$\phi_{t, s} := f_t^{-1} \circ f_s$, $(s, t) \in I^2_{\leq}$,
is an evolution family, and
the angular residue of $\phi_{t, 0}$ is $\ell(0) - \ell(t)$
(Proposition~\ref{prop:EF_LCrel}~\eqref{prop:LCtoEF}).
Hence, for a fixed $t_0 \in I$,
the family $(f_t^{-1})_{t \in I \cap [t_0, T]}$ of inverse functions
satisfies $\mathrm{(Lip)}_{\ell}$ on $f_{t_0}(D_{t_0})$.
Substituting $\phi_{t, t_0}(w) = (f_t^{-1} \circ f_{t_0})(w)$
with $w = f_{t_0}^{-1}(z)$ into \eqref{eq:result_KL_EF},
we have the following:

\begin{corollary} \label{cor:result_KL_LC}
Let $(f_t)_{t \in I}$ be a Loewner chain over $(D_t)_{t \in I}$
with some codomain.
There exists an $m_{\ell}$-null set $N_0$ and
a $\mathcal{B}^{m_\ell}(I)/\mathcal{B}(\mathcal{P}(\mathbb{R}))$-measurable mapping
$t \mapsto \nu_t$
such that, for each fixed $t_0 \in [0, T)$,
the equation
\begin{equation} \label{eq:result_KL_LC}
\tilde{\partial}^{\ell}_t (f_t^{-1})(z)
= - \pi \int_{\mathbb{R}} \Psi_{D_t}(f_t^{-1}(z), \xi) \, \nu_t(d\xi),
\quad z \in f_{t_0}(D_{t_0}),
\end{equation}
holds for any $t \in [t_0, T) \setminus N_0$.
Moreover, such a mapping $t \mapsto \nu_t$ is unique on $(0, T) \setminus N_0$.
\end{corollary}

\eqref{eq:result_KL_LC} can be further translated into
the partial differential equation for $f_t(z)$
by differentiating the identity $f_t(f_t^{-1}(z)) = z$ in $t$.
We omit the detail.

\subsection{Evolution families generated by Komatu--Loewner equation}
\label{sec:solution}

The second part of our results asserts that the solutions to \eqref{eq:result_KL_EF} form a chordal evolution family.
To apply the usual ODE theory,
our assertion is restricted to the equation~\eqref{eq:result_KL_EF} with $\lambda(t) = 2t$ only,
but this restriction is not essential; with a suitable modification, we can make a similar argument using any manner of parametrization, as is just stated after Theorem~\ref{thm:result_KL_EF} above.
In the following theorem, we denote the one-dimensional Lebesgue measure by $\mathbf{Leb}$.

\begin{theorem} \label{thm:result_unbdd}
Let $(\nu_t)_{t \geq 0}$ be a $\mathcal{B}[0,\infty)^{\mathbf{Leb}}/\mathcal{B}(\mathcal{P}(\mathbb{R}))$-measurable process and
$D$ be a parallel slit half-plane.
Then there exist a family of parallel slit half-planes $(D_t)_{t\ge 0}$ with $D_0=D$
and chordal evolution family
$(\phi_{t, s})_{(s, t) \in [0,\infty)^2_{\leq}}$
over this family such that
\begin{enumerate}
\item \label{item:2nd_result_2}
the associated angular residues satisfy $\lambda(t)=2t$ for $t\ge 0$, and
\item \label{item:2nd_result_3}
for each fixed $t_0\ge 0$, the Komatu-Loewner equation
\begin{equation} \label{eq:result_hcapKL}
\frac{\partial \phi_{t, t_0}(z)}{\partial t}
= 2\pi \int_{\mathbb{R}} \Psi_{D_t}(\phi_{t, t_0}(z), \xi) \, \nu_t(d\xi),
\quad z \in D_{t_0},
\end{equation}
holds for $\mathbf{Leb}$-a.e.\ $t\in[t_0,\infty)$.
\end{enumerate}
This evolution family is determined uniquely once $(\nu_t)_{t \geq 0}$, $D$, and $(D_t)_{t\ge 0}$ are given.
\end{theorem}

We should emphasize that this theorem does not assert the uniqueness of the family $(D_t)_{t\ge 0}$ of parallel slit half-planes for given $(\nu_t)_{t \geq 0}$ and $D$.
The above evolution family $(\phi_{t,s})_{(s,t)\in[0,\infty)^2_{\le}}$ may not be unique for given $(\nu_t)_{t \geq 0}$ and $D$ in this sense.
The next theorem gives a sufficient condition for the uniqueness.

\begin{theorem} \label{thm:result_bdd}
Suppose that in Theorem~\ref{thm:result_unbdd}
there exists $a_t>0$ for every $t>0$ such that
$\bigcup_{0 \leq s \leq t} \operatorname{supp} \nu_s \subset [-a_t, a_t]$.
Then there exist a unique family of parallel slit half-planes $(D_t)_{t \in [0, \infty)}$ with $D_0=D$ and unique evolution family $(\phi_{t, s})_{(s, t) \in [0, \infty)^2_{\leq}}$ over it
that enjoy the conclusions \eqref{item:2nd_result_2} and \eqref{item:2nd_result_3} of Theorem~\ref{thm:result_unbdd}.
\end{theorem}

\subsection{Integral representation of conformal mappings}
\label{sec:representation}

The proof of Theorem~\ref{thm:result_KL_EF}
is based on the following formula,
which is a multiply-connected version of
Lemma~1 of Goryainov and Ba~\cite{GB92}:

\begin{theorem} \label{thm:int_rep}
Let $D_1$ and $D_2$ be parallel slit half-planes
and $f \colon D_1 \to D_2$ be a conformal mapping with {\rm (H.\ref{ass:hull})}.
Then
$f$ satisfies {\rm (H.\ref{ass:hydro})} and {\rm (H.\ref{ass:res})}
if and only if
there exists a finite Borel measure $\mu$ on $\mathbb{R}$ such that
\begin{equation} \label{eq:int_rep}
f(z) = z + \pi \int_{\mathbb{R}} \Psi_{D_1}(z, \xi) \, \mu(d\xi),
\quad z \in D_1.
\end{equation}
If one of these conditions holds,
the limit
$\Im f(x) := \lim_{y \downarrow 0} \Im f(x + iy)$
exists for Lebesgue a.e.\ $x \in \mathbb{R}$,
the measure $\mu$ is uniquely given by
$\mu(d\xi) = \pi^{-1} \Im f(\xi) \, d\xi$, and
the angular residue of $f$ at infinity is $\mu(\mathbb{R})$.
\end{theorem}

We shall denote the measure $\mu(\mathord{\cdot})$ in Theorem~\ref{thm:int_rep}
by $\mu_f(\mathord{\cdot})$ or $\mu(f; \mathord{\cdot})$.

\section{Analysis of BMD complex Poisson kernel}
\label{sec:Poisson}

In this section,
the BMD Poisson kernel $K^{\ast}_D(z, \xi)$
and its complexification $\Psi_D(z, \xi)$
are studied.
In particular, we need to take a close look at the dependence of $\Psi_D$ on $D$,
as already mentioned in Section~\ref{sec:intro}.
We begin this section by introducing the notation employed for this purpose.

\subsection{Notation}
\label{sec:notation}

Let us first confirm a general notation.
For a metric space $(X, d)$,
the open and closed balls with center $a \in X$ and radius $r > 0$
are designated as $B_X(a, r)$ and $\Bar{B}_X(a, r)$, respectively.
The subscript $X$ will be dropped
if it is clear from the context which metric space we are thinking of.
We denote the Euclidean distance by $d^{\text{Eucl}}$.

We define quantities characterizing a fixed parallel slit half-plane $D$.
Set
\begin{align*}
\eta_D
&:=\inf\{\, \Im z \mathrel{;} z\in\mathbb{H}\setminus D \,\}=d^{\text{Eucl}}(\partial\mathbb{H}, \mathbb{H}\setminus D), \\
r^{\text{out}}_D
&:=\sup\{\, \lvert z \rvert \mathrel{;} z\in\mathbb{H} \setminus D \,\}, \\
r^{\text{in}}_D
&:=\inf\{\, \lvert z \rvert \mathrel{;} z\in\mathbb{H} \setminus D \,\}=d^{\text{Eucl}}(0, \mathbb{H}\setminus D).
\end{align*}
Note that $\mathbb{H}\setminus D$ is the union of the slits of $D$.
By definition, $\eta_D\le r^{\text{in}}_D<r^{\text{out}}_D$ and
$\mathbb{H}\setminus D\subset \overline{\mathbb{H}_{\eta_D}}\cap\Bar{\mathbb{A}}(0; r^{\text{in}}_D, r^{\text{out}}_D)$.
Here, $\mathbb{H}_{\eta}:=\{\,z\mathrel{;}\Im z>\eta\,\}$ and
$\mathbb{A}(a;r,R):=\{\,z\mathrel{;}r<\lvert z-a \rvert<R\,\}$.

We next introduce concepts utilized to describe the variation of slit half-planes.
Let $N \geq 1$.
We define an open subset $\mathbf{Slit}$ of $\mathbb{R}^{3N}$
as the totality of vectors
\[
\bm{s}
= (y_1, y_2, \ldots, y_N,
	\ x^{\ell}_1, x^{\ell}_2, \ldots, x^{\ell}_N,
	\ x^r_1, x^r_2 \ldots, x^r_N)
\]
with the following properties:
$y_j > 0$ and $x^{\ell}_j < x^r_j$ hold for every $j = 1, \ldots, N$, and,
if $y_j = y_k$ for some $j \neq k$,
then $x^r_j < x^{\ell}_k$ or $x^r_k < x^{\ell}_j$ holds.
Such a vector represents the endpoints of parallel slits.
To be precise, let
\[
C_j(\bm{s}) := \{\, z = x + i y_j \mathrel{;} x^{\ell}_j \leq x \leq x^r_j \,\},
\quad
D(\bm{s}) := \mathbb{H} \setminus \bigcup_{j=1}^N C_j(\bm{s}).
\]
Then $D(\bm{s})$ is a parallel slit half-plane
with the $N$ slits
whose left and right endpoints are
$z^{\ell}_j := x^{\ell}_j + i y_j$ and $z^r_j := x^r_j + i y_j$,
respectively.
We put $\Psi_{\bm{s}}(z, \xi) := \Psi_{D(\bm{s})}(z, \xi)$.

The set $\mathbf{Slit}$ is endowed with the distance
\[
d_{\mathbf{Slit}}(\bm{s}, \tilde{\bm{s}})
:= \max_{1 \leq j \leq N} \left(
\lvert z^{\ell}_j - \tilde{z}^{\ell}_j \rvert + \lvert z^r_j - \tilde{z}^r_j \rvert
\right).
\]
It is clear that $\eta_{D(\bm{s})}$, $r^{\text{out}}_{D(\bm{s})}$ and $r^{\text{in}}_{D(\bm{s})}$ are continuous functions of $\bm{s}\in\mathbf{Slit}$.

\subsection{Asymptotic behavior as $z \to \infty$}
\label{sec:residue}

We study the asymptotic behavior of the function $z\Psi_D(z,\xi)$ as $z\to\infty$ in this subsection.
The results will be used to handle the angular residue at infinity of a conformal mapping expressed by some integral involving $\Psi_D$.

\begin{proposition} \label{prop:Poi_deriv_infty}
Let $D$ be a parallel slit half-plane.
The identity
\begin{equation} \label{eq:Poi_deriv_infty}
\lim_{z \to \infty} z \Psi_D(z, \xi)
= - \frac{1}{\pi}
\end{equation}
holds for any $\xi \in \partial \mathbb{H}$.
\end{proposition}

\begin{proof}
Recall that, for a fixed $\xi\in\partial\mathbb{H}$, the function $z\mapsto \Psi_D(z,\xi)$ on $D$ is extended holomorphically across $\partial\mathbb{H}\setminus\{\xi\}$ and sends $\infty$ to $0$.
Hence
the holomorphic function $\Psi_D(-z^{-1}, \xi)$ has a zero at $z = 0$,
and the limit
$\lim_{z \to \infty} z \Psi_D(z, \xi)
= - \lim_{z \to 0} z^{-1} \Psi_D(-z^{-1}, \xi)$
exists.
Moreover,
\[
\lim_{z \to \infty} z \Psi_D(z, \xi)
= \lim_{\substack{x \to \infty \\ x \in \mathbb{R}}} x \Psi_D(x, \xi)
\in \mathbb{R}.
\]
Thus, by \cite[Eq.~(A.23)]{CF18} we have
\[
\lim_{z \to \infty} z \Psi_D(z, \xi)
= \lim_{\substack{y \uparrow \infty \\ y \in \mathbb{R}}} \Re (iy \Psi_D(iy, \xi))
= - \lim_{\substack{y \uparrow \infty \\ y \in \mathbb{R}}} y K^{\ast}_D(iy, \xi)
= -\frac{1}{\pi}.
\qedhere
\]
\end{proof}

The rate of convergence in \eqref{eq:Poi_deriv_infty}
may depend on $\xi$ and $D$,
but we have a certain uniform boundedness shown in the next lemma.

\begin{lemma} \label{lem:boundedness}
Let $\Gamma$ be a compact subset of $\mathbf{Slit}$ and $\theta \in (0, \pi/2)$.
Then there exists a constant $L=L(\Gamma,\theta)>0$ such that
\begin{equation} \label{eq:Poi_bound}
\sup\{\,\lvert z \Psi_{\bm{s}}(z,\xi)\rvert \mathrel{;} \bm{s}\in\Gamma,\ z\in\triangle_\theta\setminus B(0,L),\ \xi\in\partial\mathbb{H}\,\}
< \infty.
\end{equation}
\end{lemma}

To prove this lemma, we use the following two facts:

\begin{lemma}[Chen, Fukushima and Rohde~{\cite[Eq.\ (9.26)]{CFR16}}]
\label{lem:Lip_period}
For a parallel slit half-plane $D$,
let $\bm{A}_D$ be the period matrix in \eqref{eq:def_BMD_Poisson}.
Then each component of $\bm{A}_{D(\bm{s})}^{-1}$ is a continous function of $\bm{s}\in\mathbf{Slit}$.
\end{lemma}

\begin{lemma}[Lawler~{\cite[Eq.\ (2.12)]{La05}}]
\label{lem:hit_to_half-circle}
Let
$Z^{\mathbb{H}}=((Z^{\mathbb{H}}_t)_{t\ge 0}, (\mathbb{P}^{\mathbb{H}}_z)_{z\in\mathbb{H}})$
be the absorbed Brownian motion in $\mathbb{H}$.
Then there exists a (measurable) function $\epsilon(z)$
on $\mathbb{H}\setminus\overline{\mathbb{D}}$
with $\epsilon(z)=O(\lvert z\rvert^{-1})$, $z\to\infty$,
such that
\[
\mathbb{P}^{\mathbb{H}}_z(\sigma_{\overline{\mathbb{D}} \cap \mathbb{H}} < \infty)
=\frac{2}{\pi}\frac{\Im z}{\lvert z \rvert^2}(1+\epsilon(z)),
\quad z\in\mathbb{H}\setminus\overline{\mathbb{D}}.
\]
Here, the symbol $\sigma_A$ stands for the first hitting time of $Z^{\mathbb{H}}$ to a Borel set $A$.
\end{lemma}

\begin{proof}[Proof of Lemma~\ref{lem:boundedness}]
Let
\[
K_{\mathbb{H}}(z, \xi) := - \frac{1}{\pi} \Im \left( \frac{1}{z - \xi} \right)
= \frac{1}{\pi} \frac{y}{(x - \xi)^2 + y^2},
\quad z = x + iy \in \mathbb{H},\ \xi \in \mathbb{R}.
\]
This is the Poisson kernel of the absorbed Brownian motion $Z^{\mathbb{H}}$.
Correspondingly, we define the complex Poisson kernel for $\mathbb{H}$
(i.e., the Cauchy kernel)
$\Psi_{\mathbb{H}}(z, \xi) := - \{\pi (z - \xi)\}^{-1}$.
Since
\[
\sup_{\xi \in \partial \mathbb{H}}
\lvert z \Psi_{\mathbb{H}}(z, \xi) \rvert
= \frac{\lvert z \rvert}{\pi \Im z}
< \frac{1}{\pi \sin \theta},
\quad z \in \triangle_{\theta},
\]
it suffices to show that there exists $L>0$ such that
\begin{equation} \label{eq:dis_bound}
\sup\{\,\lvert z(\Psi_{\bm{s}}(z,\xi)-\Psi_{\mathbb{H}}(z,\xi))\rvert\mathrel{;}\bm{s}\in\Gamma,\ z\in\triangle_\theta\setminus B(0,L),\ \xi\in\partial\mathbb{H}\,\}<\infty.
\end{equation}

In what follows, we fix $\bm{s}\in\Gamma$ and $\xi\in\partial\mathbb{H}$ and
put $D=D(\bm{s})$ for notational simplicity.
Let $f(z):=\Psi_D(z,\xi)-\Psi_{\mathbb{H}}(z,\xi)$.
We first consider the imaginary part
\[
v(z):=\Im f(z)
=K^{\ast}_D(z,\xi)-K_{\mathbb{H}}(z,\xi).
\]
By the strong Markov property of $Z^{\mathbb{H}}$,
the Green function of $D$ is written as
\begin{equation} \label{eq:SMP_BM}
G_D(z, w)
=G_{\mathbb{H}}(z,w)-\mathbb{E}^{\mathbb{H}}_z \left[G_{\mathbb{H}}(Z^{\mathbb{H}}_{\sigma_{\mathbb{H} \setminus D}},w) \mathrel{;} \sigma_{\mathbb{H}\setminus D}<\infty \right].
\end{equation}
Using \eqref{eq:BMD_Green}, \eqref{eq:Green_Poisson} and \eqref{eq:SMP_BM} we have
\begin{equation} \label{eq:Poi_discrepancy}
v(z)
= - \mathbb{E}^{\mathbb{H}}_z \left[
	K_{\mathbb{H}}(Z^{\mathbb{H}}_{\sigma_{\mathbb{H} \setminus D}}, \xi)
	\mathrel{;} \sigma_{\mathbb{H} \setminus D}<\infty
	\right]
	+ \Phi_D(z) \bm{A}_D^{-1}
	\frac{\partial}{\partial \bm{n}_{\xi}} \Phi_D(\xi)^{\mathrm{tr}}.
\end{equation}
The expectation in \eqref{eq:Poi_discrepancy} enjoys
\begin{align}
\mathbb{E}^{\mathbb{H}}_z \left[
	K_{\mathbb{H}}(Z^{\mathbb{H}}_{\sigma_{\mathbb{H} \setminus D}}, \xi)
	\mathrel{;} \sigma_{\mathbb{H} \setminus D}<\infty
	\right]
&\leq \mathbb{E}^{\mathbb{H}}_z \left[
	\left( \pi \Im Z^{\mathbb{H}}_{\sigma_{\mathbb{H} \setminus D}} \right)^{-1}
	\mathrel{;} \sigma_{\mathbb{H} \setminus D} < \infty
	\right] \notag \\
& \leq \frac{1}{\pi \eta_D}
	\mathbb{P}^{\mathbb{H}}_z(\sigma_{\mathbb{H} \setminus D}<\infty).
	\label{eq:killing_est}
\end{align}
As for the second term in the right-hand side of \eqref{eq:Poi_discrepancy},
the harmonic basis has a probabilistic expression
\[
\Phi_D(z) = (\varphi_D^{(j)}(z))_{j=1}^N,
\quad
\varphi_D^{(j)}(z) = \mathbb{P}^{\mathbb{H}}_z
	(Z^{\mathbb{H}}_{\sigma_{\mathbb{H} \setminus D}} \in C_j).
\]
Thus, giving suitable estimates on the hitting probabilities of Brownian motion, we can give an upper bound of $v_\xi(z)$.

Here is a simple estimate, called the gambler's ruin estimate.
Let $\eta(\Gamma):=\min_{\bm{s}\in\Gamma}\eta_{D(\bm{s})}$.
For all $z\in\mathbb{H}$ with $0<\Im z<\eta(\Gamma)$, we then have
\begin{equation} \label{eq:hrm_est}
0 \leq \varphi_D^{(j)}(z)
\leq \mathbb{P}^{\mathbb{H}}_z(\sigma_{\mathbb{H} \setminus D} < \infty)
\leq \mathbb{P}^{\mathbb{H}}_z(
\sigma_{\{ w \mathrel{;} \Im w = \eta_D \}} < \infty
)
\leq \frac{\Im z}{\eta(\Gamma)}.
\end{equation}
Hence
\begin{equation} \label{eq:deriv_hrm_est}
0 < - \frac{\partial}{\partial \bm{n}_{\xi}} \varphi^{(j)}(\xi) \leq \frac{1}{\eta(\Gamma)}.
\end{equation}
Lemma~\ref{lem:hit_to_half-circle} gives another estimate. Let $r(\Gamma):=\max_{\bm{s}\in\Gamma}r^{\text{\rm out}}_{D(\bm{s})}$.
For $z\in\mathbb{H}\setminus \bar{B}(0,r(\Gamma))$, we have
\begin{align}
\mathbb{P}^{\mathbb{H}}_z(\sigma_{\mathbb{H} \setminus D}<\infty)
&\leq \mathbb{P}^{\mathbb{H}}_{z/r(\Gamma)}
	(\sigma_{\overline{\mathbb{D}} \cap \mathbb{H}} < \infty)
	\notag \\
&= \frac{2r(\Gamma)}{\pi} \frac{\Im z}{\lvert z \rvert^2}
	\left( 1 + \epsilon\left(\frac{z}{r(\Gamma)}\right) \right).
	\label{eq:hit_large_circle}
\end{align}
\eqref{eq:Poi_discrepancy}--\eqref{eq:hit_large_circle} and Lemma~\ref{lem:Lip_period} yield an upper bound
\begin{equation} \label{eq:im_unif_est}
\lvert v(z) \rvert \leq c_1(\Gamma) \frac{\Im z}{\lvert z \rvert^2} \left( 1 + \epsilon\left(\frac{z}{r(\Gamma)}\right) \right),
\quad z\in\mathbb{H}\setminus \bar{B}(0,r(\Gamma)),
\end{equation}
with some constant $c_1(\Gamma)$ depending only on $\Gamma$, not on each $\bm{s}\in\Gamma$ or $\xi\in\partial\mathbb{H}$.

We now use the Schwarz (or Nevanlinna--Pick) formula to recover $f(z)$ from $v(z)$.
We know that $f(\infty)=0$.
In addition, by \eqref{eq:im_unif_est} there exists a constant $c_2(\Gamma)$ such that
\begin{equation} \label{eq:im_int_on_2r}
\lvert v(u+2ir(\Gamma)) \rvert
\le \frac{c_2(\Gamma)}{u^2+4r(\Gamma)^2},
\quad u\in\mathbb{R}.
\end{equation}
Hence $\int_{\mathbb{R}}\lvert v(u+2ir(\Gamma))/(u-i) \rvert\,du<\infty$.
By these properties, we can apply a version of the Schwarz formula given by del Monaco and Gumenyuk~\cite[Proposition~2.2]{dMG16} to the shifted function $\mathbb{H}\ni z\mapsto f(z+2ir(\Gamma))$,
which yields
\begin{equation} \label{eq:dMG_Schwarz}
f(z)=\frac{1}{\pi}\int_{\mathbb{R}}\frac{v(u+2ir(\Gamma))}{u+2ir(\Gamma)-z}\, du,
\quad z\in\mathbb{H}_{2r(\Gamma)}.
\end{equation}

Let $L:=(2r(\Gamma)+1)/\sin\theta$.
Then we have
\begin{equation} \label{eq:frac_max}
\left\lvert \frac{z}{u+2ir(\Gamma)-z} \right\rvert
\le \frac{\lvert z \rvert}{\lvert z \rvert\sin\theta-2r(\Gamma)}
\le L,
\quad z\in \triangle_\theta\setminus B(0,L).
\end{equation}
Finally \eqref{eq:im_int_on_2r}--\eqref{eq:frac_max} implies that
\[
\lvert zf(z) \rvert
\le \frac{1}{\pi}\int_{\mathbb{R}}\left\lvert \frac{zv(u+2ir(\Gamma))}{u+2ir(\Gamma)-z} \right\rvert\, du
\le \frac{c_2(\Gamma)L}{\pi}\int_{\mathbb{R}}\frac{1}{u^2+4r(\Gamma)^2}\,du
\]
for all $z\in \triangle_\theta\setminus B(0,L)$,
which proves the desired conclusion~\eqref{eq:dis_bound}.
\end{proof}

\subsection{Dependence on domain variation}
\label{sec:dependence}

In this subsection we prove the following:

\begin{proposition} \label{prop:Poi_cpt}
Let $\bm{s}_0 \in \mathbf{Slit}$ and $z_0 \in D(\bm{s}_0)$ be fixed.
Then
\begin{equation} \label{eq:UnifConv}
\lim_{\substack{\bm{s} \to \bm{s}_0 \\ z \to z_0}}
	\sup_{\xi \in \partial \mathbb{H}}
	\lvert \Psi_{\bm{s}}(z, \xi) - \Psi_{\bm{s}_0}(z_0, \xi) \rvert
= 0.
\end{equation}
\end{proposition}

In the proof of Proposition~\ref{prop:Poi_cpt},
we shall use the local Lipschitz continuity of
$\Psi_{\bm{s}}(z, \xi)$ as a function of $\bm{s}$,
which was closely examined by Chen, Fukushima and Rohde~\cite{CFR16}.

\begin{proposition}[Local Lipschitz continuity of $\Psi_{\bm{s}}$~{\cite[Theorem~9.1]{CFR16}}]
\label{prop:locLip}

Given $\bm{s}_0 \in \mathbf{Slit}$,
let $K$ be a compact subset of $D(\bm{s}_0)$ and $J$ be a bounded interval.
There exist constants $\varepsilon_{\bm{s}_0, K}, L_{\bm{s}_0, K, J} > 0$
such that
\[
K \subset D(\bm{s})
\quad \text{and} \quad
\lvert \Psi_{\bm{s}}(z, \xi) - \Psi_{\bm{s}_0}(z, \xi) \rvert
\leq L_{\bm{s}_0, K, J} \, d_{\mathbf{Slit}}(\bm{s}, \bm{s}_0)
\]
hold for any $z \in K$, $\xi \in J$,
and $\bm{s} \in B_{\mathbf{Slit}}(\bm{s}_0, \varepsilon_{\bm{s}_0, K})$.
Moreover,
$\varepsilon_{\bm{s}_0, K}$ depends only on $\bm{s}_0$ and $K$, not on $J$.
\end{proposition}

In addition to Proposition~\ref{prop:locLip},
we shall use the next lemma.
Let $a \wedge b := \min \{a, b\}$.

\begin{lemma} \label{lem:Poi_Koebe}
For a parallel slit half-plane $D$, the inequality
\begin{equation} \label{eq:Poi_Koebe}
\lvert \Psi_D(z, \xi) \rvert
\leq \frac{4}{\pi} \frac{1}{\lvert z - \xi \rvert
	\wedge d^{\mathrm{Eucl}}(\xi, \mathbb{H} \setminus D)},
\quad z \in D,\ \xi \in \partial \mathbb{H},
\end{equation}
holds.
In particular,
the function $\xi \mapsto \Psi_D(z, \xi)$ belongs to
the set $C_{\infty}(\mathbb{R};\mathbb{C})$ of complex-valued continuous functions
vanishing at infinity
for each fixed $z \in D$.
\end{lemma}

\begin{proof}
We recall again that the function $\Psi_D(z,\xi)$ of $z\in D$ is extended holomorphically across $\partial\mathbb{H}\setminus\{\xi\}$ by the Schwarz reflection.
By Chen and Fukushima~\cite[Section~6.1]{CF18},
it has the Laurent expansion
\[
\Psi_D(z, \xi)
= \Psi_{\mathbb{H}}(z, \xi) + \frac{1}{2\pi} b_{\text{BMD}}(\xi; D) + o(1)
\]
around $\xi \in \partial \mathbb{H}$.
($b_{\text{BMD}}(\xi; D)$ is called the \emph{BMD domain constant}.)

We prove \eqref{eq:Poi_Koebe} in the case $\xi = 0$ only.
The general case then follows from the horizontal translation
$\Psi_D(z, \xi) = \Psi_{D - \xi}(z - \xi, 0)$~\cite[Eq.~(3.31)]{CF18}.
Here, $D - \xi = \{\, z \mathrel{;} z + \xi \in D \,\}$.

Let
$T(z) := -1/z$.
For each $r \in (0, r^{\text{in}}_D]$, the function
$h_r(z) := (\pi r)^{-1} (T \circ \Psi_D)(rz, 0)$
is univalent on $\mathbb{D}$
and has the Taylor expansion
\begin{align*}
h_r(z)
&= \frac{1}{\pi r}
\cdot \frac{-1}{-\frac{1}{\pi r z} + \frac{1}{2 \pi} b_{\text{BMD}}(0; D) z + o(z)} \\
&= z + \frac{r b_{\text{BMD}}(0; D)}{2} z^2 + o(z^2)
\end{align*}
around the origin.
Then Koebe's one-quarter theorem implies that
$B(0 ,1/4) \subset h_r(\mathbb{D})$,
which is equivalent to
$\Psi_D(r \mathbb{D}, 0) \supset B(0, 4 (\pi r)^{-1})^c$.
Since $\Psi_D(\mathord{\cdot}, 0)$ is injective on $D$,
we finally obtain
\begin{equation} \label{eq:Poi_inclusion}
\Psi_D(D \setminus (r \mathbb{D}), 0)
\subset \overline{B(0, 4 (\pi r)^{-1})}
\end{equation}
for all $r \in (0, r^{\text{in}}_D]$,
which yields \eqref{eq:Poi_Koebe} with $\xi = 0$.
\end{proof}

\begin{remark}
The author previously~\cite{Mu19jeeq} employed
the same idea as the proof of Lemma~\ref{lem:Poi_Koebe}.
In this opportunity,
we would like to correct minor mistakes in that paper.
Compared with the above-mentioned proof,
the right-hand side of \cite[Eq.~(3.8)]{Mu19jeeq} should be
$4 / (\pi r)$, not $1 / (4 \pi r)$.
Correspondingly,
the inequality in \cite[Theorem~3.1~(ii)]{Mu19jeeq}
should be replaced by $\zeta \geq y_0^2 / 16$.
\end{remark}

We now provide a proof of Proposition~\ref{prop:Poi_cpt}.

\begin{proof}[Proof of Proposition~\ref{prop:Poi_cpt}]
The function
\[
f_{\bm{s}, z}(\xi) := \Psi_{\bm{s}}(z, \xi),
\quad \bm{s} \in \mathbf{Slit},\ z \in D(\bm{s})
\]
of $\xi$ can be regarded
as a continuous function on the one-point compactification
$\mathbb{R} \cup \{\infty\}$
by Lemma~\ref{lem:Poi_Koebe}.
For the proof of \eqref{eq:UnifConv},
it suffices to prove that
there exists $\varepsilon_0 > 0$ such that
\[
\mathcal{F}
:= \{\, f_{\bm{s}, z} \mathrel{;}
\bm{s} \in B_{\mathbf{Slit}}(\bm{s}_0, \varepsilon_0),
\ z \in B_{\mathbb{C}}(z_0, \varepsilon_0) \,\}
\]
is relatively compact in the Banach space $C(\mathbb{R} \cup \{\infty\})$
equipped with the supremum norm.
Indeed, since the pointwise convergence
$f_{\bm{s}, z}(\xi) \to f_{\bm{s}_0, z_0}(\xi)$
holds as $(\bm{s}, z) \to (\bm{s}_0, z_0)$
by Proposition~\ref{prop:locLip},
a limit point of $\mathcal{F}$ as $(\bm{s}, z) \to (\bm{s}_0, z_0)$
is unique.

Let $r > 0$ be such that
$K := \Bar{B}_{\mathbb{C}}(z_0, r) \subset D(\bm{s}_0)$.
The above $\varepsilon_0$ can be taken as
$\varepsilon_0
= 5^{-1} \min \{ r, \eta_{D(\bm{s}_0)}, \varepsilon_{\bm{s}_0, K} \}$.
Here, 
$\varepsilon_{\bm{s}_0, K}$ is given as in Proposition~\ref{prop:locLip}.
To confirm that $\mathcal{F}$ is indeed relatively compact
for this value of $\varepsilon_0$,
we can apply the Arzel\`a--Ascoli theorem.
The uniform boundedness and equicontinuity at $\infty$
of $\mathcal{F}$ are trivial by \eqref{eq:Poi_Koebe}.
Hence it remains to show the equicontinuity at each point $\xi_0 \in \mathbb{R}$.
In what follows, for $\xi \in \mathbb{R}$
let $\widehat{\xi}$ be the vector of $\mathbb{R}^{3N}$
whose first $N$ entries are zero and other $2N$ entries are $\xi$.
For $\xi \in J := (\xi_0 - \varepsilon_0, \xi_0 + \varepsilon_0)$,
we have
\begin{align}
\lvert f_{\bm{s}, z}(\xi) - f_{\bm{s}, z}(\xi_0) \rvert
&= \lvert \Psi_{\bm{s}}(z, \xi) - \Psi_{\bm{s}}(z, \xi_0) \rvert \notag \\
&\leq \lvert \Psi_{\bm{s} - \widehat{\xi} + \widehat{\xi_0}}(z - \xi + \xi_0, \xi_0)
	- \Psi_{\bm{s} - \widehat{\xi} + \widehat{\xi_0}}(z, \xi_0) \rvert \notag \\
&\phantom{=} + \lvert \Psi_{\bm{s} - \widehat{\xi} + \widehat{\xi_0}}(z, \xi_0)
	- \Psi_{\bm{s}}(z, \xi_0) \rvert \notag \\
&\leq \sup_{w \in \Bar{B}(z_0, 2\varepsilon_0)} \lvert \partial_w \Psi_{\bm{s} - \widehat{\xi} + \widehat{\xi_0}}(w, \xi_0) \rvert \lvert \xi - \xi_0 \rvert
\label{eq:equiconti} \\
&\phantom{=} + L_{\bm{s}_0, K, J}
	\, d_{\mathbf{Slit}}(\bm{s}, \bm{s} - \widehat{\xi} + \widehat{\xi_0}).
\notag
\end{align}
Here, $L_{\bm{s}_0, K, J}$ is the Lipschitz constant in Proposition~\ref{prop:locLip}.
Since the family of holomorphic functions
$w \mapsto \Psi_{\tilde{\bm{s}}}(w, \xi_0)$,
$\tilde{\bm{s}} \in B_{\mathbf{Slit}}(\bm{s}, 2\varepsilon_0)$,
is locally bounded on the disk $B_{\mathbb{C}}(z, 3\varepsilon_0)$
by \eqref{eq:Poi_Koebe},
so is $\partial_w \Psi_{\tilde{\bm{s}}}(w, \xi_0)$
by Cauchy's estimate
(Eq.~(25) in Section~2.3, Chapter~4 of Ahlfors~\cite{Ah79}).
In particular,
\[
M
:= \sup_{\substack{w \in \Bar{B}_{\mathbb{C}}(z, 2\varepsilon_0)
\\ \tilde{\bm{s}} \in B_{\mathbf{Slit}}(\bm{s}, 2\varepsilon_0)}}
\lvert \partial_w \Psi_{\tilde{\bm{s}}}(w, \xi) \rvert
< \infty.
\]
Thus, it follows from \eqref{eq:equiconti} that
\[
\lvert f_{\bm{s}, z}(\xi) - f_{\bm{s}, z}(\xi_0) \rvert \leq (L_{\bm{s}_0, K, J} + M) \lvert \xi - \xi_0 \rvert
\]
for any $\bm{s} \in B_{\mathbf{Slit}}(\bm{s}_0, \varepsilon_0)$
and $z \in B_{\mathbb{C}}(z_0, \varepsilon_0)$,
which implies the equicontinuity of $\mathcal{F}$ at $\xi_0$.
\end{proof}

\section{Proof of Theorem~\ref{thm:int_rep}}
\label{sec:proof1}

For the proof of Theorem~\ref{thm:int_rep},
the following lemma is needed:

\begin{lemma} \label{lem:CpBMDInt}
\begin{enumerate}
\item \label{lem:CpBMDhol}
The integral
$\Psi_D[\mu](z)
:= \int_{\mathbb{R}} \Psi_D(z, \xi) \, \mu(d\xi)$
defines a holomorphic function on $D$
for any finite Borel measure $\mu$ on $\mathbb{R}$.

\item \label{lem:cPoi_uniq}
Let $\mu_1$ and $\mu_2$ be two finite Borel measures on $\mathbb{R}$.
If there exists a set $A \subset D$ with an accumulation point in $D$ and such that
\begin{equation} \label{eq:cPoi_uniq}
\int_{\mathbb{R}} \Psi_D(z, \xi) \, \mu_1(d\xi)
= \int_{\mathbb{R}} \Psi_D(z, \xi) \, \mu_2(d\xi)
\end{equation}
for all $z \in A$,
then $\mu_1=\mu_2$.
\end{enumerate}
\end{lemma}

\begin{proof}
\eqref{lem:CpBMDhol} is trivial.
We prove \eqref{lem:cPoi_uniq} only.

Suppose that \eqref{eq:cPoi_uniq} holds for every $z \in A$
with $A$ having an accumulation point in $D$.
By the identity theorem,
\eqref{eq:cPoi_uniq} holds for all $z \in D$.
Taking its imaginary part, we have
\[
\int_{\mathbb{R}} K^{\ast}_D(z, \xi) \, \mu_1(d\xi)
= \int_{\mathbb{R}} K^{\ast}_D(z, \xi) \, \mu_2(d\xi),
\quad z \in D.
\]
Now, the following ``inversion formula'' shows that $\mu_1 = \mu_2$
through a standard measure-theoretic argument:
for any finite Borel measure $\mu$ on $\mathbb{R}$
and $a < b$,
\begin{equation} \label{eq:Poi_inversion}
\lim_{y \downarrow 0}
\int_{\mathbb{R}} \int_a^b K^{\ast}_D(x+iy, \xi) \,dx \,\mu(d\xi)
= \mu((a,b)) + \frac{\mu(\{ a \}) + \mu(\{ b \})}{2}.
\end{equation}

It remains to prove \eqref{eq:Poi_inversion}.
We recall that the identity~\eqref{eq:Poi_inversion}
with $K^{\ast}_D$ replaced by $K_{\mathbb{H}}$
is known as the Stieltjes inversion formula;
see, e.g., Section~4, Chapter~5 of Rosenblum and Rovnyak~\cite{RR94}
or Bondesson~\cite[Theorem 2.4.1]{Bo92}.
Hence, it suffices to show
\begin{equation} \label{eq:Poi_error}
\lim_{y\downarrow 0}
\int_{\mathbb{R}}
\int_a^b
\lvert K^{\ast}_D(x+iy, \xi) - K_{\mathbb{H}}(x+iy, \xi) \rvert
\,dx
\,\mu(d\xi)
= 0.
\end{equation}
In fact,
we have
\[
\lvert K^{\ast}_D(x+iy,\xi)-K_{\mathbb{H}}(x+iy,\xi) \rvert
\leq \left(
\frac{1}{\pi}
+ N \max_{1 \leq i, j \leq N} \left\lvert (\bm{A}_D^{-1})_{i j} \right\rvert
\right)
\frac{y}{(\eta_D)^2}
\]
for all
$0 < y < \eta_D = \min \{\, \Im z \mathrel{;} z \in \mathbb{H} \setminus D \,\}$ 
in the same way as we obtained
\eqref{eq:Poi_discrepancy}--\eqref{eq:deriv_hrm_est}
in the proof of Lemma~\ref{lem:boundedness}.
This yields \eqref{eq:Poi_error}.
\end{proof}

\begin{proof}[Proof of Theorem~\ref{thm:int_rep}]
Throughout this proof,
$D_1$ and $D_2$ are parallel slit half-planes,
and $f \colon D_1 \to D_2$ is a univalent function with (H.\ref{ass:hull}),
as assumed in the theorem.

We first prove that \eqref{eq:int_rep} implies (H.\ref{ass:hydro}) and (H.\ref{ass:res}).
Suppose that \eqref{eq:int_rep} holds for a finite Borel measure $\mu$.
Since \eqref{eq:Poi_Koebe} implies that
$\lvert \Psi_{D_1}(z, \xi) \rvert \leq 4 (\pi \eta_{D_1})^{-1}$
for $z \in \mathbb{H}_{\eta_{D_1}} \cap D_1$ and $\xi \in \partial \mathbb{H}$,
we have
\[
\lim_{\substack{z \to \infty \\ z \in \mathbb{H}_{\eta}}} (f(z) - z)
= \pi \lim_{\substack{z \to \infty \\ z \in \mathbb{H}_{\eta}}}
	\int_{\mathbb{R}} \Psi_{D_1}(z, \xi) \, \mu(d\xi)
= 0
\]
for a sufficiently large $\eta > \eta_{D_1}$ by the dominated convergence theorem.
Similarly,
\eqref{eq:Poi_deriv_infty} and \eqref{eq:Poi_bound} yield
\[
\lim_{\substack{z \to \infty \\ z \in \triangle_{\theta}}} z (f(z) - z)
= \pi \lim_{\substack{z \to \infty \\ z \in \triangle_{\theta}}}
	\int_{\mathbb{R}} z \Psi_{D_1}(z, \xi) \, \mu(d\xi)
	= - \mu(\mathbb{R}).
\]

Next, we prove that (H.\ref{ass:hydro}) and (H.\ref{ass:res}) imply \eqref{eq:int_rep}.
Suppose that $f$ enjoys (H.\ref{ass:hydro}) and (H.\ref{ass:res}).
Then the supremum
$M := \sup_{\xi \in \mathbb{R}} \Im f (\xi + i \eta_0)$
is finite for some $\eta_0 > 0$, and
$f$ maps
$D_1 \setminus \overline{\mathbb{H}_{\eta_0}}$
into $D_2 \setminus \overline{\mathbb{H}_M}$.
In other words,
$\Im f$ is bounded by $M$
on $D_1 \setminus \overline{\mathbb{H}_{\eta_0}}$.
By Fatou's theorem for bounded harmonic functions
(see Garnett and Marshall~\cite[Corollary~2.5]{GM05} for example),
$\Im f(\xi + i \eta)$ converges
as $\eta \downarrow 0$ for a.e.\ $\xi \in \mathbb{R}$.
Writing this limit as $\Im f(\xi)$,
we are going to deduce
\begin{equation} \label{eq:im_int_rep}
\Im f(x + iy) - y = \int_{\mathbb{R}} K^{\ast}_{D_1}(x + iy, \xi) \Im f(\xi) \, d\xi.
\end{equation}
Here, before proving this identity, let us observe that it yields \eqref{eq:int_rep}.
Put $\mu(d\xi) := \pi^{-1} \Im f(\xi) \, d\xi$.
Then we have
\begin{align}
\mu(\mathbb{R})
&= \pi \int_{\mathbb{R}}
\lim_{y \uparrow \infty} y K^{\ast}_{D_1}(iy, \xi)
\, \mu(d\xi)
\notag \\
&\leq \liminf_{y \uparrow \infty}
y \cdot \pi \int_{\mathbb{R}} K^{\ast}_{D_1}(iy, \xi) \Im f(\xi) \, d\xi
= \lim_{y \uparrow \infty} y (\Im f(iy) - y).
\label{eq:tot_mass}
\end{align}
Here,
the first equality follows from \eqref{eq:Poi_deriv_infty},
the second inequality from Fatou's lemma, and
the third equality from \eqref{eq:im_int_rep}.
Since the rightmost side of \eqref{eq:tot_mass} is finite by (H.\ref{ass:res}),
$\mu$ is a finite Borel measure on $\mathbb{R}$.
Since the real part of $f$ is uniquely determined
by the normalization~(H.\ref{ass:hydro}),
we obtain \eqref{eq:int_rep}.

We now deduce \eqref{eq:im_int_rep}.
Let
$Z^{\ast} = ((Z^{\ast}_t)_{t \geq 0}, (\mathbb{P}^{\ast}_z)_{z \in D_1^{\ast}})$
be the BMD on $D_1^{\ast}$.
Given $\eta \in (0, \eta_{D_1})$,
the function $\Im f(x + iy) - y$ is BMD-harmonic%
\footnote{For BMD harmonicity,
see \cite[Section~3]{CFR16}.
In particular,
the imaginary part $v$ of a holomorphic function
on a finitely multiply-connected domain
is BMD-harmonic
if $v$ takes a constant boundary value on each inner boundary component.}
on $D_1 \cap \mathbb{H}_{\eta}$ and
is continuous and vanishing at infinity
on $\partial \mathbb{H}_{\eta}$.
As Chen, Fukushima and Rohde derived the displayed identity below Eq.\ (6.20) in their paper \cite[p.4086]{CFR16} from the Gauss--Green formula, such a BMD-harmonic function is expressed in terms of the BMD Poisson kernel.
From a probabilistic point of view, we can get such an expression from \eqref{eq:BMD_Poi_bd_value} and the maximal value principle for BMD-harmonic functions.
In either case, the result is
\begin{align*}
\Im f(x + iy) - y
&= \mathbb{E}^{\ast}_{x+iy} \left[
	\Im f(Z^{\ast}_{\sigma_{\partial \mathbb{H}_{\eta}}})
	- \Im Z^{\ast}_{\sigma_{\partial \mathbb{H}_{\eta}}}
	\mathrel{;} \sigma_{\partial \mathbb{H}_{\eta}} < \infty
	\right] \\
&= \int_{\mathbb{R}}
	K^{\ast}_{D_1 \cap \mathbb{H}_{\eta}}(x + iy, \xi + i\eta) (\Im f(\xi + i \eta) - \eta)
	\, d\xi
\end{align*}
for $z = x + iy \in D_1 \cap \mathbb{H}_{\eta}$.
Here, $\sigma_{\partial \mathbb{H}_{\eta}}$
is the first hitting time to $\partial \mathbb{H}_{\eta}$ of $Z^{\ast}$.
Letting $\eta \to 0$, we have
\begin{equation} \label{eq:im_BMDPoi}
\Im f(x + iy) - y
= \lim_{\eta \downarrow 0} \int_{\mathbb{R}}
	K^{\ast}_{D_1 \cap \mathbb{H}_{\eta}}(x + iy, \xi + i\eta) \Im f(\xi + i \eta)
	\, d\xi
\end{equation}
for any $z = x + iy \in D_1$.
Changing the order of the limit and integral in \eqref{eq:im_BMDPoi}
will yield \eqref{eq:im_int_rep}.
To confirm that
changing the order is indeed possible,
let
$\tilde{D}_{1, \eta}
:= \{\, z \in \mathbb{H} \mathrel{;} z + i \eta \in D_1 \cap \mathbb{H}_{\eta} \,\}$.
Proposition~\ref{prop:Poi_cpt} implies that
\begin{align*}
\lim_{\eta \downarrow 0} K^{\ast}_{D_1 \cap \mathbb{H}_{\eta}}(x + iy, \xi + i\eta)
&= \lim_{\eta \downarrow 0} K^{\ast}_{\tilde{D}_{1, \eta}}(x + i (y-\eta), \xi) \\
&= K^{\ast}_{D_1}(x + iy, \xi).
\end{align*}
Since we can see that
$\int_{\mathbb{R}} K^{\ast}_{D_1 \cap \mathbb{H}_{\eta}}(z, \xi + i\eta) \, d\xi
= \mathbb{P}^{\ast}_z(Z^{\ast}_{\sigma_{\partial \mathbb{H}_{\eta}}} \in \partial \mathbb{H}_{\eta})
= \mathbb{P}^{\ast}_z(Z^{\ast}_{\zeta^{\ast}-} \in \partial \mathbb{H})
= \int_{\mathbb{R}} K^{\ast}_{D_1}(z, \xi) \, d\xi$,
Scheff\'e's lemma ensures that
\begin{equation} \label{eq:PoiL1conv}
\lim_{\eta \downarrow 0} \int_{\mathbb{R}}
	\left\lvert K^{\ast}_{D_1 \cap \mathbb{H}_{\eta}}(x + iy, \xi + i\eta)
	- K^{\ast}_{D_1}(x + iy, \xi) \right\rvert
	\, d\xi
= 0.
\end{equation}
We now decompose the integral in \eqref{eq:im_BMDPoi} as
\begin{align}
&\int_{\mathbb{R}}
	K^{\ast}_{D_1 \cap \mathbb{H}_{\eta}}(x + iy, \xi + i\eta) \Im f(\xi + i \eta)
	\, d\xi \notag \\
&= \int_{\mathbb{R}}
	\left( K^{\ast}_{D_1 \cap \mathbb{H}_{\eta}}(x + iy, \xi + i\eta)
	- K^{\ast}_{D_1}(x + iy, \xi) \right) \Im f(\xi + i \eta)
	\, d\xi \label{eq:Poi_dcmp} \\
&\phantom{=} + \int_{\mathbb{R}}
	K^{\ast}_{D_1}(x + iy, \xi) \Im f(\xi + i \eta)
	\, d\xi. \notag
\end{align}
Since $\Im f(\xi + i \eta) \leq M$
($M$ is the constant taken just before \eqref{eq:im_int_rep}),
the former integral in the right-hand side of \eqref{eq:Poi_dcmp}
converges to zero as $\eta \downarrow 0$ by \eqref{eq:PoiL1conv}.
To the latter integral
we can apply the dominated convergence theorem.
Summarizing,
we have obtained \eqref{eq:im_int_rep} from \eqref{eq:im_BMDPoi}.

The uniqueness of $\mu$ in \eqref{eq:int_rep}
follows from Lemma~\ref{lem:CpBMDInt}~\eqref{lem:cPoi_uniq}.
The remaining assertions have already been proved
in the above argument.
\end{proof}

\begin{remark}[Limit along BMD paths]
\label{rem:Doob}

We give a rough sketch
of another possible line of proof of Theorem~\ref{thm:int_rep},
which is based on probabilistic potential theory.
Let $\mathbb{P}^{\ast, \xi}_z$ be
the Doob transform of $\mathbb{P}^{\ast}_z$
by $K^{\ast}_D(\mathord{\cdot}, \xi)$,
and suppose that
$\hat{v}(\xi) := \lim_{t \to \zeta^{\ast}} \Im f(Z^{\ast}_t)$
exists $\mathbb{P}^{\ast, \xi}_z$-a.s.\ for a.e.\ $\xi \in \partial \mathbb{H}$.
The martingale convergence theorem yields
\begin{align*}
\Im f(z)
&= \mathbb{E}^{\ast}_z \left[
	\Im f(Z^{\ast}_{\sigma_{\partial \mathbb{H}_{\eta}}})
	\right] \\
&= \mathbb{E}^{\ast}_z \left[
	\hat{v}(Z^{\ast}_{\zeta^{\ast}-})
	\right]
= \int_{\mathbb{R}}
	K^{\ast}_D(z, \xi) \hat{v}(\xi) \, d\xi.
\end{align*}
Here, $\hat{v}(\xi)$ is not the limit along the vertical line
in Theorem~\ref{thm:int_rep}
but the one along BMD paths.
For absorbed Brownian motion on $\mathbb{H}$,
Doob~\cite{Do57} deeply studied
the relationship between
the limit along vertical lines or within sectors
and the one along Brownian paths.
\end{remark}

As a corollary of Theorem~\ref{thm:int_rep},
we can prove the uniqueness of a mapping-out function
of an given $\mathbb{H}$-hull
under the assumptions~(H.\ref{ass:hydro}) and (H.\ref{ass:res}).

\begin{corollary} \label{cor:uniqueness}
Let $D_1$ and $D_2$ be parallel slit half-planes with $N$ slits and
$F$ be an $\mathbb{H}$-hull in $D_2$.
A conformal mapping $f \colon D_1 \to D_2 \setminus F$
with {\rm(H.\ref{ass:hydro})} and {\rm (H.\ref{ass:res})}
is unique if it exists.
\end{corollary}

\begin{proof}
Suppose that two mappings $f$ and $g$ satisfy the assumption.
Then $h := g^{-1} \circ f$ is a conformal automorphism on $D_1$
with (H.\ref{ass:hydro}) and (H.\ref{ass:res})
by Propositions~\ref{prop:inverse_norm} and \ref{prop:composite_norm}.
Theorem~\ref{thm:int_rep} asserts that
$h$ admits the integral representation~\eqref{eq:int_rep}
with $f$ there replaced by $h$.
However, the boundary function $\Im h(\xi)$ is zero
for all $\xi \in \mathbb{R}$
by the boundary correspondence.
Thus, we have $h(z)=z$ and $f=g$.
\end{proof}

\section{Analytic continuation of conformal mappings}
\label{sec:continuation}

In this section,
we consider conformal mappings between parallel slit domains
and their analytic continuation across the slits
and introduce the corresponding notation.
Such
analytic continuation will help us to treat points on the slits
as if they were interior points.
For example,
we conclude the continuity of the (endpoints of) parallel slits of $D_t$, $t \in I$,
over which an evolution family $(\phi_{t, s})_{(s, t) \in I^2_{\leq}}$ is defined,
from the continuity of $(\phi_{t, s})_{(s, t) \in I^2_{\leq}}$ in Lemma~\ref{lem:EFineq}.
We also examine the behavior of
solutions
to the Komatu--Loewner equation 
``around'' the slits in 
Section~\ref{sec:proof3_local}.

Let $E \subset \mathbb{C}$ be a simply connected domain,
$C_j$, $j=1, \ldots, N$, be disjoint horizontal slits in $E$,
and $D:=E \setminus \bigcup_{j=1}^N C_j$.
The left and right endpoints of $C_j$ are designated as
$z^{\ell}_j=x^{\ell}_j + iy_j$ and $z^r_j=x^r_j + iy_j$, respectively.
$C_j$ is a closed line segment, but
its open version is also defined by
$C^{\circ}_j := C_j \setminus \{z^{\ell}_j, z^r_j\}$.
Moreover, we put
\begin{align}
l_D
&:= \frac{1}{2} \min_{1 \leq j \leq N}
	\left( (x^r_j - x^{\ell}_j) \wedge d^{\text{Eucl}}(
	C_j, \partial E \cup {\textstyle \bigcup_{k \neq j} C_k}
	) \right), \label{eq:ell_D_ss5} \\
R_j
&:= \{\, x + iy \mathrel{;} x^{\ell}_j < x < x^r_j,\ \lvert y - y_j \rvert < l_D \,\},
\quad j = 1, \ldots, N. \notag
\end{align}

Let $\Pi_{\eta}$ denote the mirror reflection
with respect to the line $\Im z = \eta$,
i.e., $\Pi_{\eta} z := \overline{z} + 2i \eta$.
For a fixed $j = 1, \ldots, N$,
we take two disjoint sheets
$D \subset \mathbb{C}$
and
$\Pi_{y_j} D \times \{j\} \subset \mathbb{C} \times \{j\}$.
We glue the ``upper edge'' of the slit $C_j$
to the ``lower edge'' of $C_j \times \{j\}$
and the ``lower edge'' of $C_j$
to the ``upper edge'' of $C_j \times \{j\}$
as we glue two copies of $\mathbb{C} \setminus (-\infty, 0]$
along the negative half-line
to make the Riemann surface
on which the algebraic function $\sqrt{z}$ lives.
In other words,
the two ``rectangles''
\begin{align*}
R^{+}_j
&:= \{\, z \in R_j \mathrel{;} \Im z > y_j \,\}
	\cup C^{\circ}_j
	\cup \{\, (z, j) \in R_j \times \{j\} \mathrel{;} \Im z < y_j \,\} \\
R^{-}_j
&:= \{\, z \in R_j \mathrel{;} \Im z < y_j \,\}
	\cup (C^{\circ}_j \times \{j\})
	\cup \{\, (z, j) \in R_j \times \{j\} \mathrel{;} \Im z > y_j \,\}
\end{align*}
are connected sets in the glued sheets.
We call $C^{+}_j := C^{\circ}_j$ the upper edge
and $C^{-}_j := C^{\circ}_j \times \{j\}$ the lower edge
(of the original slit $C_j$).
The union
$C^{\natural}_j := C^{+}_j \cup C^{-}_j \cup \{z^{\ell}_j, z^r_j\}$
in the glued sheets is homeomorphic to a circle.
Repeating such gluing $N$ times,
we define a Riemann surface
\[
D^{\natural}
:= D \cup
\bigcup_{j=1}^N \left( C^{\natural}_j \cup (\Pi_{y_j} D \times \{j\}) \right).
\]
In the rest of this paper,
the superscript ${}^{\natural}$ suffixed to a parallel slit domain
means this gluing operation.

Let
$D_1 = E_1 \setminus \bigcup_{j=1}^N C_{1, j}$
and
$D_2 = E_2 \setminus \bigcup_{j=1}^N C_{2, j}$
be parallel slit domains as above and
$f \colon D_1 \to D_2$
be a conformal mapping
which associates the slit $C_{1, j}$ with $C_{2, j}$
for each $j = 1, \ldots, N$, respectively.
By the Schwarz reflection principle,
$f$ extends to a unique holomorphic function on $D_1^{\natural}$,
which is denoted by $f$ again,
with the relation
\begin{equation} \label{eq:refl_ss5}
f(p) = \Pi_{y_{2, j}} f(\Pi_{y_{1, j}} \operatorname{pr}(p))
\quad \text{for}\ p \in \Pi_{y_{1, j}} D \times \{j\}.
\end{equation}
Here,
the projection $\operatorname{pr}$ is defined by
$\operatorname{pr}(p) := z$
for $p = z \in \mathbb{C}$ and
for $p = (z, j) \in \mathbb{C} \times \{1, \ldots, N\}$.
The reflection principle also implies
that $f$ extends to a unique conformal mapping
$f^{\natural} \colon D_1^{\natural} \to D_2^{\natural}$
that is defined by
$f^{\natural}(z) := f(z)$ for $z \in D_1$
and by
$f^{\natural}((z, j)) := (f((z, j)), j)$ for $z \in \Pi_{y_{1, j}} D_1$.
The image of $p \in C^{\natural}_{1, j}$ by $f^{\natural}$ is then uniquely determined
by taking the limit.
The relation
$\operatorname{pr} \circ f^{\natural} = f$
holds by definition.
We distinguish these two analytic continuations of $f$
by whether the superscript ${}^{\natural}$ is suffixed or not.
A mapping without suffix takes values in the plane $\mathbb{C}$
while one with suffix takes values in the surface $D_2^{\natural}$.

\begin{remark}[Extension of equalities and inequalities]
\label{rem:continuation}
Through the analytic continuation,
equalities involving conformal mappings are also extended.
For example,
let $f$ be a conformal mapping between parallel slit half-planes $D_1$ and $D_2$
with (H.\ref{ass:hydro})--(H.\ref{ass:hull}).
The integral formula~\eqref{eq:int_rep} now reads
\[
f(p)=\operatorname{pr}(p)+\pi\int_{\mathbb{R}} \Psi_{D_1}(p,\xi)\, \mu(f;d\xi),
\quad p \in D_1^{\natural}.
\]
This equality is deduced from the identity theorem in general,
but in this case, we can also derive it by writing down the Schwarz reflection explicitly as in \eqref{eq:refl_ss5}.
Similarly, certain kind of inequalities involving conformal mappings are transformed in a predictable way under the Schwarz reflection,
as we shall see in the proof of Lemma~\ref{lem:EFineq} below.
\end{remark}

Finally, we mention a specific choice of a local coordinate
around $p \in C^{\natural}_j$
for a parallel slit domain $D = E \setminus \bigcup_{j=1}^N C_j$
for later use.
We use the same symbols as above.
If $p \in C^{\pm}_j$, then
$\left. \operatorname{pr} \right\rvert_{R^{\pm}_j} \colon R^{\pm}_j \to R_j$
gives a local coordinate around $p$.
If $p = z^{\ell}_j$,
we use the argument
$\vartheta^{\ell}_j(q)$ of a point $q$ near $p$
determined by the relation
$\exp (i \vartheta^{\ell}_j(q)) = (\operatorname{pr}(q) - z^{\ell}_j)/\lvert \operatorname{pr}(q) - z^{\ell}_j \rvert$
with
$0 < \vartheta^{\ell}_j(q) < 2\pi$ ($q \in D$) 
and $-2\pi < \vartheta^{\ell}_j(q) < 0$ ($q \in \Pi_{y_j} D \times \{j\}$)
and define a square root $\operatorname{sq}^\ell_j(q)=\sqrt{\operatorname{pr}(q) - z^{\ell}_j}$ by
\begin{equation} \label{eq:sq_ss5}
\operatorname{sq}^{\ell}_j(q)
:= \begin{cases}
0
& \text{if}\ q = z^{\ell}_j \\
\exp \left(
	\dfrac{\log \lvert \operatorname{pr}(q) -z^{\ell}_j \rvert + i \vartheta^{\ell}_j(q)}{2}
	\right)
& \text{if}\ q \neq z^{\ell}_j.
\end{cases}
\end{equation}
As one can see,
$\operatorname{sq}^{\ell}_j$ is a homeomorphism
from the ``doubled disk''
$B(z^{\ell}_j, l_D)
\cup \left( (B(z^{\ell}_j, l_D) \setminus \{z^{\ell}_j\}) \times \{j\} \right)$
onto a disk $B(0, \sqrt{l_D})$,
and hence
it is a local coordinate around $p = z^{\ell}_j$.
In more detail,
the map $\operatorname{pr}$ restricted to
$V:=\left(D\cup C_j^\natural\cup (\Pi_{y_j}D\times\{j\})\right)\cap\operatorname{pr}^{-1}(B(z^\ell_j, l_D))$
is a double-cover of $B(z^\ell_j, l_D)$ with a ramification point at $z^\ell_j$,
and therefore $q\mapsto \sqrt{\operatorname{pr}(q)-z^\ell_j}$
has a unique single-valued branch $\operatorname{sq}^\ell_j$ in $V$
mapping $V$ injectively onto $B(0, \sqrt{l_D})$ and
having positive derivative on $C_j^+\cap V \subset \mathbb{C}$.
See, e.g., Example~8.10 in Chapter~1 of Forster~\cite{For81} for such a standard construction of square roots.
A local coordinate $\operatorname{sq}^r_j$ around $p = z^r_j$
is defined similarly.

\section{Proof of Theorem~\ref{thm:result_KL_EF} and Corollary~\ref{cor:result_KL_LC}}
\label{sec:proof2}

As in Section~\ref{sec:derivation},
let $I$ be an interval $[0, T)$ or $[0, T]$.
We fix a chordal evolution family
$(\phi_{t, s})_{(s, t) \in I^2_{\leq}}$
over a family $(D_t)_{t \in I}$ of parallel slit half-planes
with $N$ slits (Definition~\ref{dfn:evol_family}).
The angular residue $\lambda(t)$ of $\phi_{t, 0}$ at infinity
is equal to $\mu(\phi_{t, 0}; \mathbb{R})$
by Theorem~\ref{thm:int_rep}.
We associate vectors
$\bm{s}(t) \in \mathbf{Slit}$, $t \in I$,
of slit endpoints
with $(\phi_{t, s})_{(s, t) \in I^2_{\leq}}$
so that $D_t = D(\bm{s}(t))$ and
$C^{\natural}_j(\bm{s}(t)) = \phi^{\natural}_{t, s}(C^{\natural}_j(\bm{s}(s)))$
hold for every $(s, t) \in I^2_{\leq}$.
The family $(\bm{s}(t))_{t \in I}$ is determined uniquely
by fixing $\bm{s}(0)$.
Although there are $N!$ vectors $\bm{s}$ such that $D(\bm{s}) = D_0$,
the choice of $\bm{s}(0)$ does not affect the subsequent argument.

\subsection{Continuity of evolution families}
\label{sec:proof2_cont}

In the definition of evolution family,
continuity is assumed only for $\lambda(t)$,
but as we shall see below,
this assumption turns out to be sufficient
for the continuity of $\phi_{t,s}$ and $\bm{s}(t)$
with respect to the parameters $s$ and $t$.

We begin with the following lemma,
which is easy but fundamental to the subsequent argument:

\begin{lemma} \label{lem:evol_family}
Let $(s, t) \in I^2_{\leq}$.

\begin{enumerate}
\item \label{lem:EFhcap}
The identity
$\mu(\phi_{t, s}; \mathbb{R}) = \lambda(t) - \lambda(s)$
holds.
In particular, $\lambda(t)$ is non-decreasing on $I$.

\item \label{lem:EFincr}
The inequality
$\Im \phi_{t, s}(z) \geq \Im z$
holds for any $z \in D_s$.
In particular, it follows that $\eta_{D_t} \geq \eta_{D_s}$.
Here,
$\eta_{D_t}
= \min \{\, \Im z \mathrel{;} z \in \mathbb{H} \setminus D_t \,\}
= \min_{1 \leq j \leq N} y_j(t)$.
\end{enumerate}
\end{lemma}

\begin{proof}
\eqref{lem:EFhcap} follows from Proposition~\ref{prop:composite_norm}.
We obtain \eqref{lem:EFincr},
taking the imaginary part of \eqref{eq:int_rep}:
\[
\Im \phi_{t, s}(z)
= \Im z + \pi \int_{\mathbb{R}} K^{\ast}_{D_s}(z, \xi) \, \mu(\phi_{t, s}; d\xi)
\geq \Im z.
\]
Here, as is the case with the classical Poisson kernel \cite[Theorem~II.2.5]{GM05}, $K^\ast_{D_s} \ge 0$ holds.
This follows either from the probabilistic definition mentioned around \eqref{eq:Green_Poisson}, from the formula \eqref{eq:BMD_Poi_bd_value}, or from the fact that $\Psi_{D_s}$ is a holomorphic mapping into $\mathbb{H}$.
\end{proof}

\begin{lemma} \label{lem:EFineq}
Fix $t_0 \in [0, T)$.
For every $\eta \in (0, \eta_{D_{t_0}})$,
$p \in (D_{t_0} \cap \mathbb{H}_{\eta})^{\natural} \subset D_{t_0}^{\natural}$, and
$(s, u) \in (I \cap [t_0, T])^2_{\leq}$,
the inequality
\begin{equation} \label{eq:EFineq}
\lvert \phi_{u, t_0}(p) - \phi_{s, t_0}(p) \rvert \leq \frac{12}{\eta} (\lambda(u) - \lambda(s))
\end{equation}
holds.
In particular,
the one-parameter family $(\phi_{t, t_0})_{t \in I \cap [t_0, T]}$
satisfies $\mathrm{(Lip)}_{\lambda}$ on $D_{t_0}^{\natural}$, and
$\bm{s}(t)$ is continuous.
\end{lemma}

\begin{proof}
Let $(s, u) \in (I \cap [t_0, T])^2_{\leq}$.
By \eqref{eq:int_rep} we have
\[
\phi_{u, t_0}(p) 
= \phi_{u, s}(\phi_{s, t_0}^{\natural}(p))
= \phi_{s, t_0}(p)
	+ \pi \int_{\mathbb{R}}
		\Psi_{D_s}(\phi_{s, t_0}^{\natural}(p), \xi)
		\, \mu(\phi_{u,s}; d\xi).
\]
Then by Lemma~\ref{lem:evol_family}~\eqref{lem:EFhcap},
\begin{equation} \label{eq:EFdifference}
\phi_{u, t_0}(p) - \phi_{s, t_0}(p)
= \pi (\lambda(u)- \lambda(s)) \int_{\mathbb{R}}
\Psi_{D_s}(\phi_{s, t_0}^{\natural}(p), \xi)
\, \frac{\mu(\phi_{u,s}; d\xi)}{\mu(\phi_{u,s}; \mathbb{R})}.
\end{equation}
In this identity,
an upper bound
of $\Psi_{D_s}(\phi_{s, t_0}^{\natural}(p), \xi)$
is given as follows:
let
$p \in (D_{t_0} \cap \mathbb{H}_{\eta})^{\natural} \cap (\mathbb{C} \times \{j\})$
for some $\eta \leq \eta_{D_{t_0}}$ and $j = 1, \ldots, N$.
By definition, we have
\begin{align*}
\Psi_{D_s}(\phi_{s, t_0}^{\natural}(p), \xi)
&= \Pi_{\Im \Psi_{D_s}(z^{\ell}_j(s), \xi)} \Psi_{D_s}(\Pi_{y_j(s)} \phi_{s, t_0}(p), \xi) \\
&= \overline{\Psi_{D_s}(\phi_{s, t_0}(\Pi_{y_j(t_0)} \operatorname{pr}(p)), \xi)}
+ 2i \Im \Psi_{D_s}(z^{\ell}_j(s), \xi)
\end{align*}
and $\Pi_{y_j(t_0)} \operatorname{pr}(p) \in D_{t_0} \cap \mathbb{H}_{\eta}$.
Thus, Lemmas~\ref{lem:Poi_Koebe}
and \ref{lem:evol_family}~\eqref{lem:EFincr} yield
\begin{equation} \label{eq:Poi_Koebe_ref}
\sup_{\xi \in \mathbb{R}}
\left\lvert \Psi_{D_s}(\phi_{s, t_0}^{\natural}(p), \xi) \right\rvert
\leq \frac{12}{\pi \eta}.
\end{equation}
Obviously, this inequality is true
also for $p \in (D_{t_0} \cap \mathbb{H}_{\eta})^{\natural} \cap \mathbb{C}$.
\eqref{eq:EFineq} follows from
\eqref{eq:EFdifference} and \eqref{eq:Poi_Koebe_ref}.
\end{proof}

The preceding lemma shows
the continuity of $\phi_{t, s}$ with respect to just one parameter.
Actually, \eqref{eq:EFineq} implies even its joint continuity,
as in the next proposition.

\begin{proposition} \label{prop:JointCont}
Given $t_0 \in I$,
let $U$ be a bounded open set with $\overline{U} \subset D_{t_0}$ and
$\delta>0$ be
such that $\overline{U} \subset D_t$ for all $t \in \Bar{B}_I(t_0, \delta)$.
The trapezoid
$\{\, (s, t) \mathrel{;} s \in \Bar{B}_I(t_0, \delta),\ t \in I \cap [s, T] \,\}$
is denoted by $\mathcal{T}_{t_0, \delta}$.
Then the mapping
\[
\mathcal{T}_{t_0, \delta} \ni (s, t)
\mapsto \phi_{t, s} \in \mathrm{Hol}(U; \mathbb{C})
\]
is continuous.
Here, $\mathrm{Hol}(U; \mathbb{C})$ is
the set of holomorphic functions on $U$
endowed with the topology of locally uniform convergence.
\end{proposition}

Basically, the proof of Proposition~\ref{prop:JointCont} goes
along the line of Bracci, Contreras and Diaz-Madrigal~\cite[Proposition~3.5]{BCDM12}.
However, since our domain $D_t$ depends on $t$,
we need to modify their proof with the aid of quasi-hyperbolic distance.
Before starting the proof,
let us recall the relation between hyperbolic and quasi-hyperbolic distances briefly.
(For reference, see Sections~5 and 7, Chapter~1 of Ahlfors~\cite{Ah73}.)
Let $D$ be a proper subdomain of $\mathbb{C}$.
Through the universal covering map $\mathbb{D} \twoheadrightarrow D$,
the Poincar\'e metric $2 \lvert dz \rvert / (1 - \lvert z \rvert^2)$
on the unit disk $\mathbb{D}$
induces the \emph{hyperbolic distance}
$d^{\mathrm{Hyp}}_D(z, w)$ on $D$.
It enjoys the \emph{contraction principle}:
for a holomorphic function $f \colon D \to \tilde{D}$,
we have
$d^{\mathrm{Hyp}}_{\tilde{D}}(f(z), f(w)) \leq d^{\mathrm{Hyp}}_D(z, w)$.
In particular, $D \subset \tilde{D}$ implies
$d^{\mathrm{Hyp}}_{\tilde{D}}(z, w) \leq d^{\mathrm{Hyp}}_D(z, w)$
for $z, w \in D$.
The \emph{quasi-hyperbolic distance} on $D$ is defined by
\[
d^{\mathrm{QH}}_{D}(z, w)
:= \inf_{\gamma} \int_{\gamma} \frac{2}{\delta_D(\zeta)} \, \lvert d\zeta \rvert,
\quad \delta_D(\zeta) := d^{\mathrm{Eucl}}(\zeta, \partial D) .
\]
The infimum is taken over all (piecewise) smooth curves $\gamma$
connecting $z$ and $w$.
The quasi-hyperbolic distance dominates the hyperbolic one:
$d^{\mathrm{Hyp}}_D(z, w) \leq d^{\mathrm{QH}}_{D}(z, w)$.
We shall also use the following fact:
if $C$ is a convex subset of $D$ such that
$d^{\mathrm{Eucl}}(C, \partial D) > 0$,
then by definition,
\begin{equation} \label{eq:quasi-hyp_est}
d^{\mathrm{QH}}_{D}(z, w)
\leq \frac{2 \lvert z - w \rvert}{d^{\mathrm{Eucl}}(C, \partial D)},
\quad z, w \in C.
\end{equation}

\begin{proof}[Proof of Proposition~\ref{prop:JointCont}]
Without loss of generality,
we may and do assume that $U$ is a convex set (say, a small disk).
Let $(s_n, t_n)_{n \in \mathbb{N}}$ be
a sequence in $\mathcal{T}_{t_0, \delta}$ converging to $(s, t)$.
The goal is to show that
\begin{equation} \label{eq:UnifConv_EF}
\phi_{t_n, s_n} \to \phi_{t, s} \quad \text{locally uniformly on $U$.}
\end{equation}
We make a reduction of the problem through some steps.

First, we note that the sequence
$(\phi_{t_n, s_n})_{n \in \mathbb{N}}$
is bounded on $U$.
Indeed, putting $\tilde{\eta}_U := \min_{z \in \overline{U}} \Im z > 0$,
we observe from \eqref{eq:EFineq} that
\begin{equation} \label{eq:LocBdd_EF}
\lvert \phi_{t_n, s_n}(z) - z \rvert
= \lvert \phi_{t_n, s_n}(z) - \phi_{s_n, s_n}(z) \rvert
\leq \frac{12 (\lambda(t_n) - \lambda(s_n))}{\tilde{\eta}_U \wedge \eta_{D_0}} 
\end{equation}
for all $z \in U$.
Since
$\lambda(t_n) - \lambda(s_n) \to \lambda(t) - \lambda(s)$
as $n \to \infty$,
the right-hand side of \eqref{eq:LocBdd_EF} is bounded.
Then by Vitali's convergence theorem%
\footnote{See Chapter~7, Section~2 of Rosemblum and Rovnyak~\cite{RR94}
for example.},
\eqref{eq:UnifConv_EF} follows from the pointwise convergence
\begin{equation} \label{eq:PtConv_EF}
\phi_{t_n, s_n}(z) \to \phi_{t, s}(z) \quad \text{for each $z \in U$}.
\end{equation}
Moreover, the convergence in \eqref{eq:PtConv_EF}
can be regarded as the one with respect to $d^{\text{Hyp}}_{\mathbb{H}}$
instead of $d^{\text{Eucl}}$
because they induce the same topology on $\mathbb{H}$.

Next, we note that \eqref{eq:PtConv_EF} holds if and only if,
for each fixed $z \in U$,
any subsequence
$(t^{\prime}_n, s^{\prime}_n)_{n \in \mathbb{N}}$
of $(t_n, s_n)_{n \in \mathbb{N}}$
has a further subsequence
$(t^{\prime \prime}_n, s^{\prime \prime}_n)_{n \in \mathbb{N}}$
such that
$\phi_{t^{\prime \prime}_n, s^{\prime \prime}_n}(z) \to \phi_{t, s}(z)$.
Here, the sequence
$(t^{\prime}_n, s^{\prime}_n)_{n \in \mathbb{N}}$
necessarily has a subsequence
$(t^{\prime \prime}_n, s^{\prime \prime}_n)_{n \in \mathbb{N}}$
with one of the following properties:
(I)
$s^{\prime \prime}_n \leq t^{\prime \prime}_n \leq s$ for all $n$;
(II)
$s \leq s^{\prime \prime}_n$ for all $n$;
(III)
$s^{\prime \prime}_n \leq s \leq t^{\prime \prime}_n$ for all $n$.
Thus, it suffices to show that
$\phi_{t^{\prime \prime}_n, s^{\prime \prime}_n}(z) \to \phi_{t, s}(z)$
in the cases (I)--(III).

In what follows,
we drop the superscript ${}^{\prime \prime}$ for the simplicity of notation
and assume that $(s_n, t_n)_{n \in \mathbb{N}}$ satisfies (I), (II), or (III).

Firstly, assume (I).
Since $t \geq s \geq t_n \to t$,
we have $s = t$ in this case.
Hence by \eqref{eq:LocBdd_EF},
\[
\lvert \phi_{t_n, s_n}(z) - \phi_{t, s}(z) \rvert
= \lvert \phi_{t_n, s_n}(z) - z \rvert
\leq \frac{12 (\lambda(t_n) - \lambda(s_n))}{\pi (\tilde{\eta}_U \wedge \eta_{D_0})}
\to 0
\]
as $n \to \infty$.

Secondly, assume (II).
By the contraction principle, we have
\begin{align}
&d^{\mathrm{Hyp}}_{\mathbb{H}}(\phi_{t_n, s_n}(z), \phi_{t, s}(z)) \notag \\
&\leq d^{\mathrm{Hyp}}_{\mathbb{H}}(\phi_{t_n, s_n}(z), \phi_{t_n, s}(z))
	+ d^{\mathrm{Hyp}}_{\mathbb{H}}(\phi_{t_n, s}(z), \phi_{t, s}(z)) \notag \\
&\leq d^{\mathrm{Hyp}}_{D_{t_n}}(\phi_{t_n, s_n}(z), \phi_{t_n, s_n}(\phi_{s_n, s}(z)))
	+ d^{\mathrm{Hyp}}_{\mathbb{H}}(\phi_{t_n, s}(z), \phi_{t, s}(z)) \notag \\
&\leq d^{\mathrm{Hyp}}_{D_{s_n}}(z, \phi_{s_n, s}(z))
	+ d^{\mathrm{Hyp}}_{\mathbb{H}}(\phi_{t_n, s}(z), \phi_{t, s}(z)). \label{eq:case2_1}
\end{align}
In addition, we apply
$d^{\mathrm{Hyp}} \leq d^{\mathrm{QH}}$ and \eqref{eq:quasi-hyp_est}
to obtain
\begin{align}
d^{\mathrm{Hyp}}_{D_{s_n}}(z, \phi_{s_n, s}(z))
&\leq d^{\mathrm{QH}}_{D_{s_n}}(z, \phi_{s_n, s}(z)) \notag \\
&\leq \frac{2 \lvert z - \phi_{s_n, s}(z) \rvert}%
{d^{\mathrm{Eucl}}(\overline{U},
	\partial \mathbb{H} \cup \bigcup_{j=1}^N C_j(\bm{s}(s_n))}.
\label{eq:case2_2}
\end{align}
(Note that $\phi_{s_n, s}(z) \in U$ if $n$ is large enough.)
Substituting \eqref{eq:case2_2} into \eqref{eq:case2_1}
and then applying \eqref{eq:EFineq},
we have $d^{\mathrm{Hyp}}_{\mathbb{H}}(\phi_{t_n, s_n}(z), \phi_{t, s}(z)) \to 0$
as $n \to \infty$.

Finally, assume (III).
A calculation similar to that in \eqref{eq:case2_1} yields
\begin{gather}
d^{\mathrm{Hyp}}_{\mathbb{H}}(\phi_{t_n, s_n}(z), \phi_{t, s}(z))
\leq d^{\mathrm{Hyp}}_{D_{s}}(\phi_{s, s_n}(z), z)
	+ d^{\mathrm{Hyp}}_{\mathbb{H}}(\phi_{t_n, s}(z), \phi_{t, s}(z)),
	\label{eq:case3_1} \\
\begin{aligned}
d^{\mathrm{Hyp}}_{D_{s}}(\phi_{s, s_n}(z), z)
&\leq d^{\mathrm{QH}}_{D_{s}}(\phi_{s, s_n}(z), z)
\leq \frac{2 \lvert \phi_{s, s_n}(z) - z \rvert}%
{d^{\mathrm{Eucl}}(\overline{U},
	\partial \mathbb{H} \cup \bigcup_{j=1}^N C_j(\bm{s}(s)))} \\
&\leq \frac{24 (\lambda(t_n) - \lambda(s_n))}%
{(\tilde{\eta}_U \wedge \eta_{D_0})
d^{\mathrm{Eucl}}(\overline{U},
	\partial \mathbb{H} \cup \bigcup_{j=1}^N C_j(\bm{s}(s)))}
\to 0.
\end{aligned}
\label{eq:case3_2}
\end{gather}
\eqref{eq:case3_1} and \eqref{eq:case3_2} imply that
$d^{\mathrm{Hyp}}_{\mathbb{H}}(\phi_{t_n, s_n}(z), \phi_{t, s}(z)) \to 0$.
\end{proof}

\subsection{Proof of Theorem~\ref{thm:result_KL_EF}}
\label{sec:proof2_EF}

We move to the argument on the $\lambda(t)$-differentiability.
For $t_0 \in [0, T)$, let
\begin{equation} \label{eq:exception}
N_{t_0}
:= \bigcup_{p \in D_{t_0}^{\natural}}
	\{\, t \in (t_0, T)
	\mathrel{;} \text{$\tilde{\partial}^{\lambda}_t \phi_{t, t_0}(p)$ does not exist} \,\},
\end{equation}
which is an $m_{\lambda}$-null subset of $I$
by Lemma~\ref{lem:EFineq} and Proposition~\ref{prop:ae_diff}.

\begin{lemma} \label{lem:NullConsis}
The identity $N_{t_0} = N_0 \cap (t_0, T)$ holds for every $t_0 \in (0, T)$.
\end{lemma}

\begin{proof}
Let $t_0 \in (0, T)$ and $t \in N_0 \cap (t_0, T)$.
Then for $p \in D^{\natural}_0$ such that
$\tilde{\partial}^{\lambda}_t \phi_{t, 0}(p)$ does not exist,
we have $\phi_{t, 0}(p) = \phi_{t, t_0}(\phi^{\natural}_{t_0, 0}(p))$
by (EF.\ref{cond:EFcomp}).
Hence $t \in N_{t_0}$.

Conversely, assume that $t \in N_{t_0} \setminus N_0$.
Then $\tilde{\partial}^{\lambda}_t \phi_{t, 0}(p)$ exists
for every $p \in D_0^{\natural}$.
Using (EF.\ref{cond:EFcomp}) as above,
we see that $\tilde{\partial}^{\lambda}_t \phi_{t, t_0}(p)$ exists
for $p \in \phi^{\natural}_{t_0, 0}(D_0^{\natural}) \subset D_{t_0}^{\natural}$.
In fact, this derivative exists for any $p \in D_{t_0}^{\natural}$
because we can take any countable subset
of $\phi^{\natural}_{t_0, 0}(D_0^{\natural})$
having an accumulation point in $D_{t_0}^{\natural}$
as the set $A$ in Proposition~\ref{prop:ae_diff}~\eqref{prop:Lip_aeDiff}.
However, this implies $t \notin N_{t_0}$, a contradiction.
\end{proof}

We now derive the Komatu--Loewner differential equation for the evolution family $(\phi_{t,s})_{(s,t)\in I^2_{\le}}$,
following the line of Goryainov and Ba~\cite{GB92}.
Let $\mathcal{M}_{\le 1}(\mathbb{R})$ be the set of Borel sub-probability measures on $\mathbb{R}$, i.e., Borel measures with total mass less than or equal to one.
To state the next theorem, we need the vague topology on $\mathcal{M}_{\le 1}(\mathbb{R})$,
on which we give a summary in Appendix~\ref{sec:top_on_meas}.

\begin{theorem} \label{thm:EFmain}
For every $t \in (0, T) \setminus N_0$,
the normalized measure
$\mu(\phi_{t+\delta, t-\delta}; \mathord{\cdot})\allowbreak
/\mu(\phi_{t+\delta, t-\delta}; \mathbb{R})$
converges vaguely to a measure $\nu_t$ as $\delta \downarrow 0$.
If $\nu_t$ is defined suitably
(say, as the point mass $\delta_0$ at the origin)
on $N_0$,
then $t \mapsto \nu_t$ is a measurable mapping
from $(I, \mathcal{B}^{m_{\lambda}}(I))$ to
$(\mathcal{M}_{\leq 1}(\mathbb{R}), \mathcal{B}(\mathcal{M}_{\leq 1}(\mathbb{R})))$.
Moreover, for each fixed $t_0 \in [0, T)$,
the equation
\begin{equation} \label{eq:EF_KLeq}
\tilde{\partial}^{\lambda}_t \phi_{t, t_0}(p)
= \pi \int_{\mathbb{R}} \Psi_{D_t}(\phi_{t, t_0}^{\natural}(p), \xi) \, \nu_t(d\xi)
\end{equation}
holds for any $t \in (t_0, T) \setminus N_{t_0}$ and $p \in D_{t_0}^{\natural}$.
An $\mathcal{M}(\mathbb{R})_{\leq 1}$-valued process $(\nu_t)_{t \in I}$
that satisfies \eqref{eq:EF_KLeq} is unique on $(0, T) \setminus N_0$.
\end{theorem}

\begin{proof}
We fix $t_0 \in [0, T)$ and $t \in (t_0, T) \setminus N_0$ throughout this proof.
Let $(\delta_n)_{n \in \mathbb{N}}$ be an arbitrary sequence of positive numbers
converging to zero. 
By Proposition~\ref{prop:vague_cpt},
there exists a subsequence $(\delta'_n)_n$
such that the sequence of probability measures
\[
\mu^{\sharp}_n
:= \frac{\mu(\phi_{t+\delta'_n, t-\delta'_n}; \mathord{\cdot})}{\mu(\phi_{t+\delta'_n, t-\delta'_n}; \mathbb{R})}
\]
converges vaguely to $\mu^{\sharp}_{\infty}$ as $n \to \infty$.
We show that
\begin{equation} \label{eq:subseqKL}
\tilde{\partial}^{\lambda}_t \phi_{t, t_0}(p) = \pi \int_{\mathbb{R}} \Psi_{D_t}(\phi_{t, t_0}^{\natural}(p), \xi) \, \mu^{\sharp}_{\infty}(d\xi).
\end{equation}
Let $z \in D_{t_0}$ and $n$ be large enough.
From \eqref{eq:EFdifference} we get
\begin{align}
&\left\lvert
\frac{\phi_{t+\delta'_n, t_0}(z) - \phi_{t-\delta'_n, t_0}(z)}{\pi (\lambda(t+\delta'_n) - \lambda(t-\delta'_n))}
- \int_{\mathbb{R}} \Psi_{D_t}(\phi_{t, t_0}(z), \xi) \, \mu^{\sharp}_{\infty}(d\xi)
\right\rvert \notag \\
&=\left\lvert
\int_{\mathbb{R}} \Psi_{D_{t-\delta'_n}}(\phi_{t-\delta'_n, t_0}(z), \xi) \, \mu^{\sharp}_n(d\xi)
	- \int_{\mathbb{R}} \Psi_{D_t}(\phi_{t, t_0}(z), \xi) \, \mu^{\sharp}_{\infty}(d\xi)
\right\rvert \notag \\
&\leq \int_{\mathbb{R}} \lvert \Psi_{D_{t-\delta'_n}}(\phi_{t-\delta'_n, t_0}(z), \xi) - \Psi_{D_t}(\phi_{t, t_0}(z), \xi) \rvert \, \mu^{\sharp}_n(d\xi)
\label{eq:conv_right} \\
&\phantom{\leq} + \left\lvert
\int_{\mathbb{R}} \Psi_{D_t}(\phi_{t, t_0}(z), \xi) \, \mu^{\sharp}_n(d\xi)
- \int_{\mathbb{R}} \Psi_{D_t}(\phi_{t, t_0}(z), \xi) \, \mu^{\sharp}_{\infty}(d\xi)
\right\rvert.
\notag
\end{align}
In the rightmost side of \eqref{eq:conv_right},
the former integral vanishes as $n \to \infty$
by Proposition~\ref{prop:Poi_cpt} and
the continuity of $\phi_{v, u}$ and $\bm{s}(v)$
with respect to $v$
in Lemma~\ref{lem:EFineq}.
The remaining term in the rightmost side of \eqref{eq:conv_right}
also tends to zero
by Lemma~\ref{lem:Poi_Koebe} and
the vague convergence
of $(\mu^{\sharp}_n)_n$.
Thus, \eqref{eq:subseqKL} holds for $p = z \in D_{t_0}$.
The analytic continuation yields \eqref{eq:subseqKL} for all $p \in D_{t_0}^{\natural}$.

$\mu^{\sharp}_{\infty}$ is, in fact,
independent of the choice of the above subsequence $(\delta'_n)_n$.
To prove this, assume that we have another subsequence
$(\delta''_n)_n$ of $(\delta_n)_n$ such that
\[
\mu^{\flat}_n
:= \frac{\mu(\phi_{t+\delta''_n, t-\delta''_n}; \mathord{\cdot})}{\mu(\phi_{t+\delta''_n, t-\delta''_n}; \mathbb{R})}
\]
converges vaguely to $\mu^{\flat}_{\infty}$ as $n \to \infty$.
Then \eqref{eq:subseqKL}
with $\mu^{\sharp}_{\infty}$ replaced by $\mu^{\flat}_{\infty}$
holds by the same reasoning.
As a result,
\[
\int_{\mathbb{R}} \Psi_{D_t}(\phi_{t, t_0}(z), \xi) \, \mu^{\sharp}_{\infty}(d\xi)
= \int_{\mathbb{R}} \Psi_{D_t}(\phi_{t, t_0}(z), \xi) \, \mu^{\flat}_{\infty}(d\xi),
\quad z \in D_{t_0}.
\]
In particular, we have
\[
\int_{\mathbb{R}} \Psi_{D_t}(z, \xi) \, \mu^{\sharp}_{\infty}(d\xi)
= \int_{\mathbb{R}} \Psi_{D_t}(z, \xi) \, \mu^{\flat}_{\infty}(d\xi)
\]
for all $z \in \phi_{t, t_0}(D_{t_0})$.
Lemma~\ref{lem:CpBMDInt}~\eqref{lem:cPoi_uniq} now yields
$\mu^{\sharp}_{\infty} = \mu^{\flat}_{\infty}$,
which proves both the convergence of
$\mu(\phi_{t+\delta, t-\delta}; \mathord{\cdot})
/\mu(\phi_{t+\delta, t-\delta}; \mathbb{R})$
and the equation~\eqref{eq:EF_KLeq}.
Lemma~\ref{lem:CpBMDInt}~\eqref{lem:cPoi_uniq} also shows
the uniqueness of $\nu_t$.

Recall that
$\mu(\phi_{t-\delta, t+\delta}; d\xi)
= \pi^{-1} \Im \phi_{t+\delta, t-\delta}(\xi) \, d\xi$
holds in Theorem~\ref{thm:int_rep}.
For each $f\in C_c(\mathbb{R})$,
the vague limit
$\nu_t = \lim_{\delta \downarrow 0}
\mu(\phi_{t+\delta, t-\delta}; \mathord{\cdot})
/\mu(\phi_{t+\delta, t-\delta}; \mathbb{R})$
enjoys
\begin{align*}
\int_{\mathbb{R}} f(\xi) \,\nu_t(d\xi)
&= \frac{1}{\pi} \lim_{n \to \infty} \left(
\frac{1}{\lambda(t+1/n) - \lambda(t-1/n)}
\right. \\
&\phantom{= \frac{1}{\pi} \lim_{n \to \infty} \biggl( \biggr.} \left.
\times \lim_{m \to \infty}
\int_{\mathbb{R}} f(\xi) \Im \phi_{t + 1/n,\, t-1/n}\left(\xi + \frac{i}{m}\right) \, d\xi
\right).
\end{align*}
The right-hand side of this equality is a $\mathcal{B}^{m_\lambda}(I)/\mathcal{B}(\mathbb{R})$-measurable function of $t$,
and so is the composite of the two mappings
$t\mapsto\nu_t$ and
$\pi_f\colon\mathcal{M}_{\le 1}(\mathbb{R})\to\mathbb{R}, m\mapsto\int f \,dm$.
Since $\mathcal{B}(\mathcal{M}_{\le 1}(\mathbb{R}))=\sigma(\,\pi_f\mathrel{;} f\in C_c(\mathbb{R})\,)$ by Proposition~\ref{prop:vague_Borel},
the mapping $t\mapsto \nu_t$ is $\mathcal{B}^{m_\lambda}(I)/\mathcal{B}(\mathcal{M}_{\le 1}(\mathbb{R}))$-measurable.
\end{proof}

We notice that
the essence of the preceding proof
can be summarized in the following manner:

\begin{corollary}
\label{cor:EFdiff}
Let $t_0 \in [0, T)$ and $t \in [t_0, T)$.
For a sequence
$(s_n , u_n)_{n \in \mathbb{N}}$
in $[t_0, T)^2_{<}$
with $s_n \leq t \leq u_n$ and
$s_n, u_n \to t$,
the following are equivalent:
\begin{enumerate}
\item \label{item:FuncDiff_seq}
$\displaystyle
\frac{\phi_{u_n, t_0}(p) - \phi_{s_n, t_0}(p)}{\lambda(u_n) - \lambda(s_n)}$
converges as $n \to \infty$
for every $p \in D_{t_0}^{\natural}$;

\item \label{item:MeasDiff_seq}
$\displaystyle
\frac{\mu(\phi_{u_n, s_n}; \mathord{\cdot})}{\mu(\phi_{u_n, s_n}; \mathbb{R})}$
converges vaguely as $n \to \infty$.
\end{enumerate}
If either of the two is true,
then \eqref{eq:EF_KLeq} holds at $t$
with $\tilde{\partial}^{\lambda}_t \phi_{t, t_0}(p)$ and $\nu_t$
replaced by these limits.
\end{corollary}

In Theorem~\ref{thm:EFmain},
it is not proved that $\nu_t$ is a probability measure.
For probabilists
a familiar way to prove such a statement is to show that
the normalized measures $\mu_n^{\sharp}$ in the proof of Theorem~\ref{thm:EFmain} are tight (Definition~\ref{def:tight}),
but the author has no ideas about how to obtain the tightness.
Here we give a much different, simple argument,
following I.\ A.\ Aleksandrov, S.\ T.\ Aleksandrov and Sobolev~\cite{AAS83},
to show that $\nu_t(\mathbb{R})=1$ for $m_\lambda$-a.e.\ $t$.

\begin{theorem} \label{thm:nu_is_prob}
The driving process $\nu_t$ in Theorem~\ref{thm:EFmain} takes value in $\mathcal{P}(\mathbb{R})$,
the set of Borel probability measures on $\mathbb{R}$,
for $m_\lambda$-a.e.\ $t\in I$.
\end{theorem}

\begin{proof}
Let $t\in I$ be fixed.
Integrating \eqref{eq:EF_KLeq} we have
\[
\phi_{t,0}(z)-z
=\pi\int_{[0,t]}\int_{\mathbb{R}}\Psi_{D_u}(\phi_{u,0}(z),\xi)\,\nu_u(d\xi)\,m_\lambda(du),
\quad z\in D_0.
\]
Hence, taking the limit in the angular residue along the imaginary axis, we have
\begin{align}
\lambda(t)&=\lim_{y\uparrow+\infty}y(y-\Im\phi_{t,0}(iy)) \notag \\
&=\pi\lim_{y\uparrow+\infty}\int_{[0,t]}\int_{\mathbb{R}}y\Im\Psi_{D_u}(\phi_{u,0}(iy),\xi)\,\nu_u(d\xi)\,m_\lambda(du) \notag \\
&\le \pi\liminf_{y\uparrow+\infty}\int_{[0,t]}\int_{\mathbb{R}}\Im\phi_{u,0}(iy)\Im\Psi_{D_u}(\phi_{u,0}(iy),\xi)\,\nu_u(d\xi)\,m_\lambda(du) \notag \\
&\le \pi\limsup_{y\uparrow+\infty}\int_{[0,t]}\int_{\mathbb{R}}\left\lvert \phi_{u,0}(iy)\Psi_{D_u}(\phi_{u,0}(iy),\xi) \right\rvert\,\nu_u(d\xi)\,m_\lambda(du). \label{eq:nu_is_prob1}
\end{align}
Here, we have used Lemma~\ref{lem:evol_family}~\eqref{lem:EFincr}.
In the last integrand,
$\phi_{u,0}(iy)$ goes to infinity through some sectorial domain as $y\uparrow+\infty$.
Indeed, \eqref{eq:int_rep} and \eqref{eq:Poi_Koebe} imply that
\[
\lvert \phi_{u,0}(iy)-iy \rvert
\le \pi\int_{\mathbb{R}}\lvert \Psi_{D_u}(iy,\xi) \rvert\, \mu(\phi_{u,0};d\xi)
\le \frac{\pi\lambda(t)}{\eta_{D_0}}
\]
for all $u\in[0,t]$ and $y>\max_{u\in[0,t]}r^{\text{out}}_{D_u}$.
Thus, there exists an angle $\theta\in(0,\pi/2)$ such that $\phi_{u,0}(iy)\in\triangle_\theta$ for all such $u$ and $y$.
We apply Lemma~\ref{lem:boundedness} to a compact set
$\Gamma:=\{\,\bm{s}(u)\mathrel{;}u\in[0,t]\,\}\subset\textbf{Slit}$
and the angle $\theta$.
Then it follows from \eqref{eq:Poi_deriv_infty} and the dominated convergence theorem that
\begin{equation} \label{eq:nu_is_prob2}
\lim_{y\uparrow+\infty}\int_{[0,t]}\int_{\mathbb{R}}\left\lvert \phi_{u,0}(iy)\Psi_{D_u}(\phi_{u,0}(iy),\xi) \right\rvert\,\nu_u(d\xi)\,m_\lambda(du)
=\frac{1}{\pi}\int_{[0,t]}\nu_u(\mathbb{R})\, m_\lambda(du).
\end{equation}
By \eqref{eq:nu_is_prob1} and \eqref{eq:nu_is_prob2}, we have
\[
m_\lambda((0,t])=\lambda(t)\le \int_{[0,t]}\nu_u(\mathbb{R})\, m_\lambda(du).
\]
This is true only if $\nu_u(\mathbb{R})=1$ for $m_\lambda$-a.e.\ $u\in[0,t]$
since $\nu_u(\mathbb{R})\le 1$.
\end{proof}

Theorems~\ref{thm:EFmain} and \ref{thm:nu_is_prob} complete the proof of Theorem~\ref{thm:result_KL_EF}.

For later use, we add a remark to Theorem~\ref{thm:EFmain} (or \ref{thm:result_KL_EF}).
We have considered the differentiation only with respect to $\lambda(t)$,
but it is also reasonable
to consider the differentiation with respect to $t$.
To this end,
a natural manner of thinking is
to assume the absolute continuity of $\lambda(t)$ in $t$ or
to perform time-change.

In the former standpoint,
we suppose that $\lambda(t)$ is absolutely continuous in $t \in I$.
Then by $\mathrm{(Lip)}_{\lambda}$,
the function $t \mapsto \phi_{t, t_0}(p)$ is also absolutely continuous
and hence differentiable in a.e.\ $t$ in the usual sense.
Thus, \eqref{eq:EF_KLeq} reduces to
\begin{equation} \label{eq:KLeq_extrinsic}
\frac{\partial \phi_{t, t_0}(p)}{\partial t}
= \pi \dot{\lambda}(t)
	\int_{\mathbb{R}}
	\Psi_{D_t}(\phi_{t, t_0}^{\natural}(p), \xi)
	\, \nu_t(d\xi)
\end{equation}
for Lebesgue a.e.\ $t \in [t_0, T)$.

In the latter standpoint,
we suppose
that $\lambda(t)$ is (strictly) increasing and
that the condition~\eqref{item:FuncDiff_seq} or \eqref{item:MeasDiff_seq}
in Corollary~\ref{cor:EFdiff} holds
for every $t \in I$
and every choice of $(s_n, u_n)_{n \in \mathbb{N}}$.
For any increasing continuous function
$\theta$ on $I$,
we perform time-change as
$\tilde{\phi}_{t, s} := \phi_{\theta^{-1}(t), \theta^{-1}(s)}$,
$\tilde{D}_t := D_{\theta^{-1}(t)}$,
$\tilde{\lambda}(t) := \lambda(\theta^{-1}(t))$, and
$\tilde{\nu}_t := \nu_{\theta^{-1}(t)}$.
Then
\begin{equation} \label{eq:KLeq_time-change}
\frac{\partial \tilde{\phi}_{t, s}(p)}{\partial \tilde{\lambda}(t)}
:= \lim_{h \to 0}
	\frac{\tilde{\phi}_{t+h, s}(p) - \tilde{\phi}_{t, s}(p)}%
	{\tilde{\lambda}(t+h) - \tilde{\lambda}(t)}
= \pi \int_{\mathbb{R}}
	\Psi_{\tilde{D}_t}(\tilde{\phi}_{t, s}^{\natural}(p), \xi)
	\, \tilde{\nu}_t(d\xi).
\end{equation}
In particular,
choosing $\theta = \lambda/2$ gives
$\tilde{\lambda}(t) = 2t$ and
\begin{equation} \label{eq:KLeq_hcap_parametrized}
\frac{\partial \tilde{\phi}_{t, s}(p)}{\partial t}
= 2 \pi \int_{\mathbb{R}}
	\Psi_{\tilde{D}_t}(\tilde{\phi}_{t, s}^{\natural}(p), \xi)
	\, \tilde{\nu}_t(d\xi).
\end{equation}
In this case,
$(\tilde{\phi}_{t, s})$ is said
to be \emph{parametrized by half-plane capacity}%
\footnote{We have chosen the homeomorphism $\theta = \lambda/2$
so that $\tilde{\lambda}(t) = 2t$,
not $\tilde{\lambda}(t) = t$.
This coefficient two is just a convention in SLE theory and not essential.}
in the SLE context.
\eqref{eq:KLeq_hcap_parametrized}
as well as \eqref{eq:EF_KLeq}
provides a natural way
to regard $\lambda(t)$ as a canonical parameter.

\subsection{Proof of Corollary~\ref{cor:result_KL_LC}}
\label{sec:proof2_LC}

In this subsection,
we observe the correspondence
between evolution families and Loewner chains
to confirm Corollary~\ref{cor:result_KL_LC}.
For making our statement on this correspondence clear,
we give a preliminary result. 
This is an application of Theorem~\ref{thm:result_KL_EF}
(or \ref{thm:EFmain}) and
will be used also in Section~\ref{sec:proof3}.

\begin{proposition} \label{prop:prolong}
Let $0<T<\infty$ and $(\phi_{t, s})_{(s, t) \in [0, T)^2_{\leq}}$ be an evolution family
with $\lambda(t) = \mu(\phi_{t, 0}; \mathbb{R})$.
This family extends to a unique evolution family
$(\tilde{\phi}_{t, s})_{(s, t) \in [0, T]^2_{\leq}}$
in the sense that $\tilde{\phi}_{t, s} = \phi_{t, s}$ for all $(s, t) \in [0, T)^2_{\leq}$
if and only if $\sup_{0 \leq t < T} \lambda(t) < \infty$.
\end{proposition}

\begin{proof}
The ``only if'' part is trivial
because, if there exists $(\tilde{\phi}_{t, s})_{(s, t) \in [0, T]^2_{\leq}}$
with $\tilde{\lambda}(t) = \mu(\tilde{\phi}_{t, 0}; \mathbb{R})$
such that $\tilde{\phi}_{t, s} = \phi_{t, s}$ for all $(s, t) \in [0, T)^2_{\leq}$,
then
$\sup_{0 \leq t < T} \lambda(t)
= \sup_{0 \leq t < T} \tilde{\lambda}(t)
= \tilde{\lambda}(T) < \infty$.

The reminder of this proof is devoted to the ``if'' part.
Since the uniqueness of a desired extension
is trivial by Proposition~\ref{prop:JointCont},
we present its construction below.
Suppose that $\sup_{0 \leq t < T} \lambda(t) < \infty$.
This is equivalent to
$m_{\lambda}([0, T)) < \infty$.
For a spell,
$t_0 \in [0, T)$ is fixed.

$\phi_{t, t_0}(z)$ converges as $t \uparrow T$
for every $z \in D_{t_0}$,
as seen from the integral form of \eqref{eq:result_KL_EF}:
\begin{equation} \label{eq:IntKL}
\phi_{t, t_0}(z)
= z + \pi \int_I \int_{\mathbb{R}}
\Psi_{D_{s}}(\phi_{s, t_0}(z), \xi)
\, \nu_{s}(d\xi) \bm{1}_{[t_0, t)}(s) \, m_{\lambda}(ds).
\end{equation}
Since \eqref{eq:Poi_Koebe} and Lemma~\ref{lem:evol_family}~\eqref{lem:EFincr} yield
\[
\sup_{t_0 \leq s < T} \lvert \Psi_{D_{s}}(\phi_{s, t_0}(z), \xi) \rvert
\leq \frac{4}{\pi} \frac{1}{\Im z \wedge \eta_{D_0}},
\quad z \in D_{t_0},
\]
the dominated convergence theorem applies to \eqref{eq:IntKL}.

Vitali's theorem converts
the pointwise convergence of $(\phi_{t, t_0}(z))_{t \in [t_0, T)}$ above
into the locally uniform convergence
of $(\phi_{t, t_0})_{t \in [t_0, T)}$ on $D_{t_0}^{\natural}$.
Indeed,
the family $(\phi_{t, t_0})_{t \in [t_0, T)}$ is locally bounded on $D_{t_0}^{\natural}$,
for $\mu(\phi_{t, t_0}; \mathbb{R}) \leq m_{\lambda}([0, T))$ holds
in the identity
\begin{equation} \label{eq:EFint_rep}
\phi_{t, t_0}(p)
=\operatorname{pr}(p)+\pi\int_{\mathbb{R}}\Psi_{D_{t_0}}(p,\xi)\, \mu(\phi_{t, t_0}; d\xi),
\quad p\in D_{t_0}^\natural.
\end{equation}
Hurwitz's theorem%
\footnote{See, e.g., Theorem~A in Section~2, Chapter~7
of Rosemblum and Rovnyak~\cite{RR94}.}
guarantees that
$\tilde{\phi}_{T, t_0} := \lim_{t \to T} \phi_{t, t_0}$ is univalent on $D_{t_0}$.
In addition,
using the relative compactness
of $(\mu(\phi_{t, t_0}; \mathord{\cdot}))_{t \in [t_0, T)}$
in the vague topology (Proposition~\ref{prop:vague_cpt}),
we can choose a sequence $(t_n)_{n=1}^{\infty}$ converging to $T$
so that the limit
$\mu_{T, t_0} := \lim_{n \to \infty} \mu(\phi_{t_n, t_0}; \mathord{\cdot})$
exists.
Letting $t \to T$ along this sequence in \eqref{eq:EFint_rep} yields,
in particular for $z\in D_{t_0}$,
\begin{equation} \label{eq:EFlim_int}
\tilde{\phi}_{T, t_0}(z)
=z+\pi\int_{\mathbb{R}}\Psi_{D_{t_0}}(z,\xi)\, \mu_{T, t_0}(d\xi).
\end{equation}

We observe
that $\tilde{\phi}_{T, t_0}$ enjoys the assumption~(H.\ref{ass:hull}) as follows:
by the boundary correspondence
(see, e.g., Courant~\cite[Theorem~2.4]{Co50} and references therein),
the inner boundaries of $\tilde{\phi}_{T, t_0}(D_{t_0})$ are given by
the images $\tilde{\phi}_{T, t_0}(C^{\natural}_j(\bm{s}(t_0)))$, $j = 1, \ldots, N$.
These images must be parallel slits,
possibly degenerating into points,
because $(\phi_{t, t_0})_{t \in [t_0, T)}$ converges uniformly
on the compact set $C^{\natural}_j(\bm{s}(t_0))$.
In fact, these limit slits are non-degenerate
because a conformal mapping preserves
the degree of connectivity and the non-degeneracy%
\footnote{The degree of connectivity is preserved
because it is a topological property.
The non-degeneracy is preserved by the removable singularity theorem.
}.
Thus, writing
$\bm{s}(T) := \lim_{t \to T} \bm{s}(t)$,
we see that $\bm{s}(T) \in \mathbf{Slit}$.
In conclusion,
$\tilde{\phi}_{T, t_0} \colon D_{t_0} \to D_T := D(\bm{s}(T))$
is a univalent function into a parallel slit half-plane
and satisfies (H.\ref{ass:hull}) trivially.
Once (H.\ref{ass:hull}) is proved,
(H.\ref{ass:hydro}) and (H.\ref{ass:res}) are also proved
by Theorem~\ref{thm:int_rep} combined with \eqref{eq:EFlim_int}.

Up to this point,
we have constructed univalent functions
$\tilde{\phi}_{T, t_0} \colon D_{t_0} \to D_T$, $t_0 \in [0, T)$,
with (H.\ref{ass:hydro})--(H.\ref{ass:hull}).
We further define
$\tilde{\phi}_{t, s} := \phi_{t, s}$ for all $(s, t) \in [0, T)^2_{\leq}$ and
$\tilde{\phi}_{T, T} := \mathrm{id}_{D_T}$ (the identity mapping on $D_T$).
The family $(\tilde{\phi}_{t, s})_{(s, t) \in [0, T]^2_{\leq}}$
automatically satisfies (EF.\ref{cond:EFnorm}).
For $0 \leq s \leq t < T$, we have
\[
\tilde{\phi}_{T, s}(z) = \lim_{u \to T} \phi_{u, s}(z) = \lim_{u \to T} \phi_{u, t}(\phi_{t, s}(z)) = \tilde{\phi}_{T, t}(\tilde{\phi}_{t, s}(z)),
\]
which implies (EF.\ref{cond:EFcomp}).
Moreover,
$\tilde{\lambda}(t) := \mu(\tilde{\phi}_{t, 0}; \mathbb{R})$
is non-decreasing by the same proof
as that of Lemma~\ref{lem:evol_family}~\eqref{lem:EFhcap}.
We have
\begin{align*}
&\lim_{t\to T} \lambda(t)
=\lim_{t\to T}\tilde{\lambda}(t)
\le \tilde{\lambda}(T) \\
&=\mu_{T,0}(\mathbb{R})
\le \liminf_n \mu(\phi_{t_n,0};\mathbb{R})
= \lim_{t\to T} \lambda(t).
\end{align*}
Here, $(t_n)_{n=1}^\infty$ is the sequence chosen above,
and the last inequality is due to a basic property \eqref{eq:vague_Fatou} of vague convergence.
It is now clear that $\tilde{\lambda}(T) = \lim_{t \to T} \lambda(t)$,
which implies (EF.\ref{cond:EFcont}).
\end{proof}

The correspondence between Loewner chains and evolution families
is formulated as follows:

\begin{proposition} \label{prop:EF_LCrel}
\begin{enumerate}
\item \label{prop:LCtoEF}
Let $(f_t)_{t \in I}$ be a Loewner chain over $(D_t)_{t \in I}$ with any codomain.
The two-parameter family
\begin{equation} \label{eq:LCtoEF}
\phi_{t, s} := f_t^{-1} \circ f_s, \quad (s, t) \in I^2_{\leq},
\end{equation}
is an evolution family, and
$\lambda(t) = \mu(\phi_{t, 0}; \mathbb{R})$ is bounded on $I$.

\item \label{prop:EFtoLC}
Let $0<T<\infty$,
and suppose that $(\phi_{t,s})_{(s,t)\in [0,T)^2_{\le}}$ is an evolution family over $(D_t)_{t \in [0,T)}$
with $\lambda(t) = \mu(\phi_{t, 0}; \mathbb{R})$ bounded.
Its prolongation to $[0, T]^2_{\leq}$,
which is guaranteed by Proposition~\ref{prop:prolong},
is designated by the same symbol.
Then the family
\begin{equation} \label{eq:EFtoLC}
f_t := \phi_{T, t}, \quad t \in [0, T],
\end{equation}
is a Loewner chain over $(D_t)_{t \in [0, T]}$ with codomain $D_T$.
\end{enumerate}
\end{proposition}

\begin{proof}
\eqref{prop:LCtoEF}
The univalent functions
$\phi_{t, s} = f_t^{-1} \circ f_s \colon D_s \to D_t$
enjoy (H.\ref{ass:hydro})--(H.\ref{ass:hull})
by Propositions~\ref{prop:inverse_norm} and \ref{prop:composite_norm}.
The latter proposition also implies that $\lambda(t) = - \ell(t) + \ell(0)$,
which proves (EF.\ref{cond:EFcont}).
The conditions~(EF.\ref{cond:EFnorm}) and (EF.\ref{cond:EFcomp}) are trivial.
Finally, we have
\[
\lambda(t) \leq \ell(t) + \lambda(t) = \ell(0) < \infty.
\]

\noindent
\eqref{prop:EFtoLC}
For $(s, t) \in [0, T]^2_{\leq}$, we have
\[
f_s(z) = \phi_{T, s}(z) = \phi_{T, t}(\phi_{t, s}(z)) = f_t(\phi_{t, s}(z)).
\]
Hence $f_s(D_s) = f_t(\phi_{t, s}(D_s)) \subset f_t(D_t)$,
which implies (LC.\ref{cond:LCsub}).
Moreover, Proposition~\ref{prop:composite_norm} yields
\[
\mu(f_s; \mathbb{R}) = \mu(f_t; \mathbb{R}) + \lambda(t) - \lambda(s),
\]
which implies (LC.\ref{cond:LCcont}).
\end{proof}

Thanks to Proposition~\ref{prop:EF_LCrel}~\eqref{prop:LCtoEF},
Theorem~\ref{thm:result_KL_EF} (or \ref{thm:EFmain})
applies to a Loewner chains $(f_t)_{t \in I}$,
which yields Corollary~\ref{cor:result_KL_LC}.

\begin{remark}[Terminal conditions on Loewner chains]
\label{rem:terminal}
Given an evolution family $(\phi_{t,s})_{(s,t)\in [0,T]^2_{\le}}$,
the family \eqref{eq:EFtoLC} is not a unique Loewner chain that satisfies \eqref{eq:LCtoEF}.
Indeed, for any univalent function $h$ on $D_T$ with (H.\ref{ass:hydro})--(H.\ref{ass:hull}),
the family
\begin{equation} \label{eq:terminal}
h\circ \phi_{T,t},\quad t\in [0,T],
\end{equation}
is also a Loewner chain that satisfies \eqref{eq:LCtoEF}.
One way to formulate the uniqueness is to take the terminal condition into account.
Namely,
\emph{given an evolution family $(\phi_{t,s})_{(s,t)\in [0,T]^2_{\le}}$ over $(D_t)_{t\in [0,T]}$ and
a univalent function $h$ on $D_T$ with {\rm (H.\ref{ass:hydro})--(H.\ref{ass:hull})},
a Loewner chain $(f_t)_{t\in [0,T]}$ over $(D_t)_{t\in [0,T]}$ which enjoys \eqref{eq:LCtoEF} and $f_T=h$ is unique}.
To see this,
let $(\tilde{f}_t)_{t\in [0,T]}$ be another chain satisfying $\tilde{f}_t\circ \phi_{t,s}=\tilde{f}_s$ and $\tilde{f}_T=h$.
We consider the mapping $\psi_t:=\tilde{f}_t\circ f_t^{-1}$ from $f_t(D_t)$ onto $\tilde{f}_t(D_t)$ for each $t$.
By the terminal conditions we have $\psi_T=h\circ h^{-1}=\mathrm{id}$.
For a fixed $t\in[0,T]$ and $w\in f_t(D_t)$,
there exists $z\in D_t$ with $w=f_t(z)$, and
\begin{align*}
w&=\psi_T(w)=\tilde{f}_T\circ f_T^{-1}(f_t(z))
=\tilde{f}_T\circ f_T^{-1}(f_T\circ \phi_{T,t}(z)) \\
&=\tilde{f}_T(\phi_{T,t}(z))=\tilde{f}_t(z)
=\tilde{f}_t(f_t^{-1}(w))=\psi_t(w).
\end{align*}
This implies $\tilde{f}_t=f_t$.
In view of \eqref{eq:terminal} and this uniqueness statement,
the Loewner chain \eqref{eq:EFtoLC} may be regarded as a \emph{standard} one;
for comparison, see Definition~3.5 and Theorem~3.6 of Contreras, Diaz-Madrigal and Gumenyuk~\cite{CDMG10},
in which a general discussion is presented on the uniqueness of Loewner chains in the unit disk $\mathbb{D}$.
\end{remark}

\section{Proof of Theorems~\ref{thm:result_unbdd} and \ref{thm:result_bdd}}
\label{sec:proof3}

In this section,
we study the solutions
to the Komatu--Loewner equation~\eqref{eq:result_hcapKL}
with half-plane capacity parametrization
and show that they form an evolution family.

\subsection{Local solutions to Komatu--Loewner equation}
\label{sec:proof3_local}

When we consider solutions to the Komatu--Loewner equation,
we need a special care with the fact
that the right-hand side of \eqref{eq:result_hcapKL} depends on $D_t$.
Since $D_t$ is the codomain of $\phi_{t, t_0}$,
the dependence of the right-hand side of \eqref{eq:result_hcapKL}
on the unknown variable $\phi_{t, t_0}(z)$
is quite complicated.
Even the existence of a solution to \eqref{eq:result_hcapKL}
is unclear from the usual ODE theory.

Bauer and Friedrich~\cite{BF08} presented a way
to resolve the above-mentioned problem with regard to \eqref{eq:intro_1KL}.
(See also Chen and Fukushima~\cite{CF18}.)
They first derived the Komatu--Loewner equation
for the slit endpoints $z^{\ell}_j(t)$ and $z^r_j(t)$ from \eqref{eq:intro_1KL}.
Since the dependence of $D_t$ on these endpoints is simple,
one can prove the existence of a solution $z^{\ell}_j(t)$ and $z^r_j(t)$.
Once $z^{\ell}_j(t)$ and $z^r_j(t)$ (and thus $D_t$) are determined,
it is possible to prove the existence of a solution to \eqref{eq:intro_1KL}.

We follow the idea of Bauer and Friedrich.
In our case,
the Komatu--Loewner equation for slits is formulated as follows
(see Appendix~\ref{sec:KL_slit} for its derivation):
for $N \geq 1$, $\bm{s} \in \mathbf{Slit}$,
and $\nu \in \mathcal{P}(\mathbb{R})$,
let
\begin{align*}
\bm{b}(\nu, \bm{s})
&:= (b_k(\nu, \bm{s}))_{k=1}^{3N}, \\
b_k(\nu, \bm{s})
&:= \begin{cases}
\pi \int_{\mathbb{R}} \Im \Psi_{\bm{s}}(z^{\ell}_k, \xi) \, \nu(d\xi),
& 1 \leq k \leq N \\
\pi \int_{\mathbb{R}} \Re \Psi_{\bm{s}}(z^{\ell}_{k - N}, \xi) \, \nu(d\xi),
& N + 1 \leq k \leq 2N \\
\pi \int_{\mathbb{R}} \Re \Psi_{\bm{s}}(z^r_{k - 2N}, \xi) \, \nu(d\xi),
& 2N + 1 \leq k \leq 3N.
\end{cases}
\end{align*}
We fix a $\mathcal{P}(\mathbb{R})$-valued
Lebesgue-measurable process $(\nu_t)_{t \geq 0}$. 
The
Komatu--Loewner equation for the vector $\bm{s}(t)$ of slit endpoints
is written as
\begin{equation} \label{eq:KLforSlit_vec}
\frac{d \bm{s}(t)}{d t} = 2 \bm{b}(\nu_t, \bm{s}(t)).
\end{equation}
This is an ODE driven by $\nu_t$ with phase space $\mathbf{Slit}$.

We begin our argument with the existence of a local solution to \eqref{eq:KLforSlit_vec}.
Here, we mean by a solution to an ODE an absolutely continuous function which satisfies the ODE in the a.e.\ sense.
By \eqref{eq:Poi_Koebe}, the function
$[0,\infty)\times \mathbf{Slit}\ni (t, \bm{s})\mapsto \bm{b}(\nu_t, \bm{s})$
satisfies the \emph{Carath\'eodory condition};
see Theorem~1.1 in Section~1, Chapter~2 of Coddington and Levinson~\cite{CL55}
or Theorem~5.1 in Section~I.5 of Hale~\cite{Ha69}.
By the cited theorems the following holds:

\begin{proposition} \label{prop:LocSolSlit}
Let $t_0 \in [0, \infty)$.
For every initial condition
$\bm{s}(t_0) = \bm{s}_0 \in \mathbf{Slit}$,
the ODE~\eqref{eq:KLforSlit_vec} has a local solution $\bm{s}(t)$.
\end{proposition}

Let $D$ be a parallel slit half-plane with $N$ slits
and $\bm{s}_0 \in \mathbf{Slit}$ be such that $D(\bm{s}_0) = D$.
Note that Proposition~\ref{prop:LocSolSlit} says nothing
about the uniqueness of a solution.
We choose a solution $\bm{s}(t)$ to \eqref{eq:KLforSlit_vec}
with $\bm{s}(0) = \bm{s}_0$ freely,
denote its maximal interval of existence by $[0, T)$, and
write $D_t := D(\bm{s}(t))$.
For this solution $\bm{s}(t)$, we consider the ODE
\begin{equation} \label{eq:KLforEF}
\frac{d z(t)}{d t}
= 2 \pi \int_{\mathbb{R}} \Psi_{\bm{s}(t)}(z(t), \xi) \, \nu_t(d\xi),
\end{equation}
which is the same as \eqref{eq:result_hcapKL}
except that some symbols are replaced.

\begin{proposition} \label{prop:KL_local_Lipschitz}
The function
$H(t, z) := \int_{\mathbb{R}} \Psi_{\bm{s}(t)}(z, \xi) \, \nu_t(d\xi)$
on the domain
$\bigcup_{t \in [0, T)} (\{t\} \times D_t) \subset [0, T) \times \mathbb{C}$
enjoys the local Lipschitz condition in $z$.
In particular,
\eqref{eq:KLforEF} has a unique local solution
$z(t) = z(t; s, z_0)$
with $z(s; s, z_0) = z_0 \in D_s$
for every $s \in [0, T)$ and $z_0 \in D_s$.
\end{proposition}

\begin{proof}
Let $s \in [0, T)$ and $r>0$ be such that $\bar{B}(z_0, 2r) \subset D_s$.
There exists $\delta>0$ such that
$\bar{B}(z_0, 2r) \subset D_t$
for all $(s - \delta)^{+} := \max \{(s - \delta), 0\} \leq t \leq s + \delta$.
It suffices to show that $H(t, z)$ is Lipschitz in $z$
on the set $[(s - \delta)^{+}, s + \delta] \times \bar{B}(z_0, r)$.

For any $t \in J := [(s - \delta)^{+}, s + \delta]$
and $z_1, z_2 \in \bar{B}(z_0, r)$,
\begin{align*}
H(t, z_2) - H(t, z_1)
&= \int_{\mathbb{R}}
	\left( \Psi_{\bm{s}(t)}(z_2, \xi) - \Psi_{\bm{s}(t)}(z_1, \xi) \right)
	\, \nu_t(d\xi) \\
&= \int_{\mathbb{R}}
	\int_{z_1}^{z_2} \frac{\partial}{\partial z} \Psi_{\bm{s}(t)}(z, \xi) \, dz
	\, \nu_t(d\xi).
\end{align*}
Since
$\bar{B}(z, r) \subset \bar{B}(z_0, 2r) \subset \bigcap_{u \in J} D_u$
for every $z \in \Bar{B}(z_0, r)$,
Cauchy's estimate yields
\[
\lvert H(t, z_2) - H(t, z_1) \rvert
\leq r^{-1} M \lvert z_2 - z_1 \rvert,
\quad M := \sup_{u \in J, z \in \bar{B}(z_0, r)} \lvert \Psi_{\bm{s}(u)}(z, \xi) \rvert.
\]
Finally, we see from \eqref{eq:Poi_Koebe} that $M$ is finite,
noting that
$\eta_{D_u} = \min_{1 \leq j \leq N} y_j(u)$
is non-decreasing in $u$
because the first $N$ entries of $\bm{b}(\nu_t, \bm{s}(t))$
in \eqref{eq:KLforSlit_vec}
are positive.
\end{proof}

For each $s \in [0, T)$ and $z_0 \in D_s$,
let $\tau_{s, z_0}$ be the right endpoint of the maximal interval
of existence of the solution $z(t; s, z_0)$ to \eqref{eq:KLforEF}.

\begin{proposition} \label{prop:MaxSol}
$\tau_{s, z_0} = T$ holds
for any $s \in [0, T)$ and $z_0 \in D_s$.
\end{proposition}

We shall prove Proposition~\ref{prop:MaxSol}
along the line of Chen and Fukushima~\cite[Section~5.1]{CF18}.
Assume that $\tau_{s, z_0} < T$.
Since $\Im z(t; s, z_0)$ is non-decreasing in $t$,
$\Psi_{\bm{s}(t)}(z(t; s, z_0), \xi)$ is bounded by \eqref{eq:Poi_Koebe}.
The following limit exists by the dominated convergence theorem:
\begin{align*}
\tilde{z}
&:= \lim_{t \to \tau_{s, z_0}} z(t; s, z_0) \\
&= z_0 + 2 \pi \int_s^{\tau_{s, z_0}} \!\!\! \int_{\mathbb{R}}
 \Psi_{\bm{s}(u)}(z(u; s, z_0), \xi)
\, \nu_u(d\xi) \, du.
\end{align*}
We have $\tilde{z} \notin D_{\tau_{s, z_0}}$ because otherwise the solution $t \mapsto z(t; s, z_0)$ would extend beyond $\tau_{s, z_0}$.
Moreover, it does not lie on $\partial \mathbb{H}$
because $\Im z(t; s, z_0)$ is non-decreasing in $t$.
Thus, $\tilde{z} \in C_j(\tau_{s, z_0})$ for some $j$.
We shall deduce a contradiction from this,
using the uniqueness of a solution
to the Komatu--Loewner equation ``around $C_j^{\natural}(\bm{s}(t))$.''
The precise uniqueness statement will be divided into four lemmas below.

\begin{lemma} \label{lem:SolAboveSlit}
Let $t_0 \in [0, T)$ and $p_0 \in C_j^{+}(t_0)$.
Suppose that
an open neighborhood $U_{p_0}$ of $p_0$ in $D_{t_0}^{\natural}$
and $\delta > 0$
are such that
$V_{p_0} := \left. \operatorname{pr} \right\rvert_{U_{p_0}}(U_{p_0})$
satisfies
$V_{p_0} \subset R_j(t)$,
$\emptyset \neq V_{p_0} \cap C_j(t) \subset C_j^{\circ}(t)$,
and
$V_{p_0} \cap \bigcup_{k \neq j} C_k(t) = \emptyset$
for all $(t_0 - \delta)^{+} \leq t \leq t_0 + \delta$.
(By the continuity of $\bm{s}(t)$, such $U_{p_0}$ and $\delta$ exist.)
The ODE
\begin{equation} \label{eq:KLaboveSlit}
\frac{d \tilde{z}(t)}{dt}
= 2 \pi \int_{\mathbb{R}}
	\Psi_{\bm{s}(t)}(
	(\left. \operatorname{pr} \right\rvert_{R_j^{+}(t)})^{-1}(\tilde{z}(t)), \xi
	)
	\, \nu_t(d\xi)
\end{equation}
has a unique local solution
$\tilde{z}(t) = \tilde{z}(t; t_0, p_0) \in V_{p_0}$
with $\tilde{z}(t_0; t_0, p_0) = \operatorname{pr}(p_0)$.
Moreover,
$\tilde{z}(t; t_0, p_0) \in C_j^{\circ}(t)$
holds for every $t$ in some neighborhood of $t_0$.
\end{lemma}

\begin{proof}
The unique existence of a local solution to \eqref{eq:KLaboveSlit}
is proved in the same way
as in
the proof of Proposition~\ref{prop:KL_local_Lipschitz}.
We construct this local solution in the following way:
the ODE
\begin{equation} \label{eq:KLslit_x-direction}
\frac{d\hat{x}(t)}{dt}
= 2 \pi \int_{\mathbb{R}} \Re \Psi_{\bm{s}(t)}(\hat{x}(t) + iy_j(t), \xi) \, \nu_t(d\xi).
\end{equation}
with $\hat{x}(t_0) = \Re \operatorname{pr}(p_0)$
has a unique local solution $\hat{x}(t)$.
We may assume that $x^{\ell}_j(t) < \hat{x}(t) < x^r_j(t)$.
Then $\hat{y}(t) := y_j(t)$ is a unique solution to the ODE
\begin{equation} \label{eq:KLslit_y-direction}
\frac{d\hat{y}(t)}{dt}
= 2 \pi \int_{\mathbb{R}} \Im \Psi_{\bm{s}(t)}(\hat{x}(t) + i \hat{y}(t), \xi) \, \nu_t(d\xi)
\end{equation}
with $\hat{y}(t_0) = \Im \operatorname{pr}(p_0) = y_j(t_0)$
because
$\Im \Psi_{\bm{s}(t)}(x + iy_j(t), \xi) = \Im \Psi_{\bm{s}(t)}(z^{\ell}_j(t), \xi)$
for any $x^{\ell}_j(t) < x < x^r_j(t)$.
By \eqref{eq:KLslit_x-direction} and \eqref{eq:KLslit_y-direction},
$\hat{z}(t) := \hat{x}(t) + iy_j(t)$
is a local solution to \eqref{eq:KLaboveSlit}
with initial condition $\hat{z}(t_0) = \operatorname{pr}(p_0)$,
and $\hat{z}(t) \in C_j^{\circ}(t)$
on some neighborhood of $t_0$.
\end{proof}

Since the proof of the next lemma is quite similar
to that of Lemma~\ref{lem:SolAboveSlit},
we omit its proof:

\begin{lemma} \label{lem:SolBelowSlit}
Let $t_0$, $\delta$, $U_{p_0}$, and $V_{p_0}$ be defined
as in Lemma~\ref{lem:SolAboveSlit}
with $p_0 \in C^{-}_j(t_0)$ instead of $p_0 \in C^{+}_j(t_0)$.
The ODE
\begin{equation} \label{eq:KLbelowSlit}
\frac{d \tilde{z}(t)}{dt}
= 2 \pi \int_{\mathbb{R}}
	\Psi_{\bm{s}(t)}(
	(\left. \operatorname{pr} \right\rvert_{R_j^{-}(t)})^{-1}(\tilde{z}(t)), \xi
	)
	\, \nu_t(d\xi)
\end{equation}
has a unique local solution
$\tilde{z}(t) = \tilde{z}(t; t_0, p_0)$
with
$\tilde{z}(t_0; t_0, p_0) = \operatorname{pr}(p_0)$.
Moreover,
$\tilde{z}(t; t_0, p_0) \in C_j^{\circ}(t)$
holds for every $t$ in some neighborhood of $t_0$.
\end{lemma}

Let us move to the third lemma of the four.
We shall slightly modify the notation for the local coordinate around $z^\ell_j(t)$ in $D_t^\natural$,
introduced in \eqref{eq:sq_ss5};
namely, in order to emphasize the dependence on $t$,
we will write $\operatorname{sq}^\ell_{j,t}$ in place of $\operatorname{sq}^\ell$.
Let
\begin{equation} \label{eq:Psi_norm}
\Psi^{\ell, j}_{D_t}[\nu_t](z)
:= \int_{\mathbb{R}} \left(
	\Psi_{D_t}((\operatorname{sq}^{\ell}_{j, t})^{-1}(z), \xi) - \Psi_{D_t}(z^{\ell}_j(t), \xi)
	\right) \, \nu_t(d\xi).
\end{equation}
Since $\Psi^{\ell, j}_{D_t}[\nu_t](z)$ has a zero at $z = 0$,
there exists a holomorphic function $H^{\ell, j}(t, z)$
such that
\begin{equation} \label{eq:1stZero_Psi}
\Psi^{\ell, j}_{D_t}[\nu_t](z) = z H^{\ell, j}(t, z).
\end{equation}
By Cauchy's integral formula we have,
for a fixed $0 < r < l_{D_t}$,
\begin{equation} \label{eq:Psi_deriv}
H^{\ell, j}(t, z)
= \frac{1}{2\pi i} \int_{\partial B(0, r)}
	\frac{\Psi^{\ell, j}_{D_t}[\nu_t](\zeta)}{\zeta (\zeta - z)}
	\, d\zeta,
\quad z \in B(0, r);
\end{equation}
see Eq.~(29) in Section~3.1, Chapter~4 of Ahlfors~\cite{Ah79}.
Here, $l_{D_t}$ is defined by \eqref{eq:ell_D_ss5}.

\begin{lemma} \label{lem:SolNearEnd}
\begin{enumerate}
\item \label{lem:rewrite_KL_end}
Suppose that $z(t)$ is a solution to \eqref{eq:KLforEF} and
that $z(t) \in B(z^{\ell}_j(t), l_{D_t}) \setminus C_j(t)$ on some interval.
Then $\tilde{z}(t) := \operatorname{sq}^{\ell}_{j, t}(z(t))$ enjoys
\begin{equation} \label{eq:KLnearEnd}
\frac{d\tilde{z}(t)}{dt}  = \pi H^{\ell, j}(t, \tilde{z}(t))
\end{equation}
a.e.\ on that interval.

\item \label{lem:LocSol_KLnearEnd}
Let $t_0 \in [0, T)$.
Suppose that $\delta, r > 0$ are such that
$l_{D_t} > 2r$
for $(t_0 - \delta)^{+} \leq t \leq t_0 + \delta$
(such $\delta$ and $r$ exist by the continuity of $\bm{s}(t)$).
For each fixed
$p_0 \in (B(z^{\ell}_j(t_0), r) \setminus C^{\circ}_j(t_0))
	\cup ((B(z^{\ell}_j(t_0), r) \setminus C_j(t_0)) \times \{j\})$,
the ODE~\eqref{eq:KLnearEnd}
has a unique local solution
$\tilde{z}(t) = \tilde{z}(t; t_0, p_0)$
with
$\tilde{z}(t_0; t_0, p_0) = \operatorname{sq}^{\ell}_{j, t_0}(p_0)$.
Moreover,
$\Im \tilde{z}(t; t_0, z^{\ell}_j(t_0)) = 0$
for every $t$ in some neighborhood of $t_0$.
\end{enumerate}
\end{lemma}

\begin{proof}
\eqref{lem:rewrite_KL_end}
Since
$\tilde{z}(t)
= \operatorname{sq}^{\ell}_{j, t}(z(t))
= \sqrt{z(t) - z^{\ell}_j(t)}$,
we have
\begin{align}
&\frac{d}{dt} \tilde{z}(t)
= \frac{1}{2 \tilde{z}(t)} \left(
\frac{dz(t)}{dt} - \frac{dz^{\ell}_j(t)}{dt}
\right) \notag \\
&= \frac{\pi}{\tilde{z}(t)}
\int_{\mathbb{R}} \left(
\Psi_{D_t}((\operatorname{sq}^{\ell}_{j, t})^{-1}(\tilde{z}(t)), \xi)
- \Psi_{D_t}(z^{\ell}_j(t), \xi)
\right) \, \nu_t(d\xi)
\label{eq:rewrite_KL_end0}
\end{align}
by \eqref{eq:KLforEF}.
Substituting \eqref{eq:1stZero_Psi} into \eqref{eq:rewrite_KL_end0},
we obtain \eqref{eq:KLnearEnd}.

\noindent
\eqref{lem:LocSol_KLnearEnd}
As in the proof of Proposition~\ref{prop:KL_local_Lipschitz},
we can conclude the unique existence of a local solution
from the local boundedness of $H^{\ell, j}(t, z)$ on $B(0, r)$.
This boundedness easily follows
from the expressions~\eqref{eq:Psi_norm} and \eqref{eq:Psi_deriv}
combined with Cauchy's estimate.

By the definition of $\Psi^{\ell, j}_{D_t}[\nu_t]$ and $\operatorname{sq}^{\ell}_{j, t}$,
we have
\[
\Im \Psi^{\ell, j}_{D_t}[\nu_t]((\operatorname{sq}^{\ell}_{j, t})^{-1}(x)) = 0,
\quad x \in B(0, r) \cap \mathbb{R},
\]
and thus $\Im H^{\ell, j}(x, t) = 0$ for
$x \in B(0, r) \cap \mathbb{R}$.
Hence, the solution to \eqref{eq:KLnearEnd}
belongs to $\mathbb{R}$
if the initial value $\operatorname{sq}^{\ell}_{j, t_0}(p_0)$ is real.
In particular,
since $\operatorname{sq}^{\ell}_{j, t_0}(z^{\ell}_j(t_0)) = 0$,
we have
$\Im \tilde{z}(t; t_0, z^{\ell}_j(t_0)) = 0$.
\end{proof}

By the same proof as that of Lemma~\ref{lem:SolNearEnd},
we have the fourth lemma:

\begin{lemma} \label{lem:SolNearEnd2}
All the results in Lemma~\ref{lem:SolNearEnd}
with the superscript $\ell$ replaced by $r$ hold true.
\end{lemma}

Using the four lemmas above,
we now give a proof of Proposition~\ref{prop:MaxSol}.

\begin{proof}[Proof of Proposition~\ref{prop:MaxSol}]
Assuming that $\tau_{s, z_0} < T$,
we have already shown that
$\tilde{z} := \lim_{t \to \tau_{s, z_0}} z(t; s, z_0)
\in C_j(\tau_{s, z_0})$ for some $j$.

Suppose that $\tilde{z} = z^{\ell}_j(\tau_{s, z_0})$.
By Lemma~\ref{lem:SolNearEnd}~\eqref{lem:rewrite_KL_end},
$\tilde{z}(t) := \operatorname{sq}^{\ell}_{j, t}(z(t; s, z_0))$
is a local solution to \eqref{eq:KLnearEnd}
with $\tilde{z}(\tau_{s, z_0}) = z^{\ell}_j(\tau_{s, z_0})$.
However,
Lemma~\ref{lem:SolNearEnd}~\eqref{lem:LocSol_KLnearEnd} implies
that $z(t; s, z_0) \in C_j(t)$ for some $t < \tau_{s, z_0}$,
which contradicts the definition of $\tau_{s, z_0}$.
Thus, $\tilde{z} \neq z^{\ell}_j(\tau_{s, z_0})$.
In the same way, Lemma~\ref{lem:SolNearEnd2} yields
$\tilde{z}(\tau_{s, z_0}) \neq z^r_j(\tau_{s, z_0})$.

The remaining case is that
$\tilde{z} \in C^{\circ}_j(\tau_{s, z_0})$.
Since $\tilde{z}$ does not coincide with the endpoints of $C_j(\tau_{s, z_0})$,
we can take $\delta > 0$ so that
$\Im z(t; s, z_0) - y_j(t)$ takes a constant sign
on $[\tau_{s, z_0} - \delta, \tau_{s, z_0})$.
Suppose that $\Im z(t; s, z_0) > y_j(t)$ on this interval.
Then
$\tilde{z}(t) = z(t; s, z_0)$
is a local solution to \eqref{eq:KLaboveSlit}
with $\tilde{z}(\tau_{s, z_0}) \in C^{+}_j(\tau_{s, z_0})$.
However, this contradicts
Lemma~\ref{lem:SolAboveSlit}
because $z(t; s, z_0) \notin C_j(t)$ for $t < \tau_{s, z_0}$.
Similarly, the case $\Im z(t; s, z_0) < y_j(t)$ does not occur
by Lemma~\ref{lem:SolBelowSlit}.

We have seen
that the assumption $\tau_{s, z_0} < T$ leads to a contradiction,
and hence $\tau_{s, z_0} = T$.
\end{proof}

\subsection{Evolution family formed by the solutions}
\label{sec:proof3_global}

By Proposition~\ref{prop:MaxSol},
we can define a mapping
$\phi_{t, s} \colon D_s \ni z_0 \mapsto z(t; s, z_0) \in D_t$
for any $(s, t) \in [0, T)^2_{\leq}$.
The backward equation of \eqref{eq:KLforEF} in the next lemma
will be used to show that $\phi_{t, s}$ enjoys (H.\ref{ass:hull}).

\begin{lemma} \label{lem:EFback}
\begin{enumerate}
\item \label{lem:EFback_sol}
For a fixed $t_0 \in (0, T)$ and $z_0 \in D_{t_0}$,
the backward equation
\begin{equation} \label{eq:EFback}
\frac{dw(t)}{dt}
= - 2 \pi \int_{\mathbb{R}} \Psi_{\bm{s}(t_0 - t)}(w(t), \xi) \, \nu_{t_0 - t}(d\xi)
\end{equation}
has a unique local solution $w(t) = w(t; t_0, z_0)$
with $w(0) = z_0$.

\item \label{lem:EFback_time}
Let $\tilde{\tau}_{t_0, z_0}$ be the right endpoint of
the maximal interval of existence
of the solution $w(t; t_0, z_0)$.
If $\tilde{\tau}_{t_0, z_0} < t_0$, then
$\lim_{t \to \tilde{\tau}_{t_0, z_0}} \Im w(t; t_0, z_0) = 0$.

\item \label{lem:EFback_inf}
There exists a constant $\delta_0 > 0$ such that
\[
\tilde{\tau}_{t_0, z_0} \geq (2\delta_0) \wedge t_0
\]
for any $t_0 \in (0, T)$ and $z_0 \in D_{t_0} \cap \mathbb{H}_{\eta_{D_0}/2}$.
Here,
we define
$\eta_{D_0} := \min \{\, \Im z \mathrel{;} z \in \mathbb{H} \setminus D_0 \,\}$ 
as in Section~\ref{sec:notation}.
\end{enumerate}
\end{lemma}

\begin{proof}
\eqref{lem:EFback_sol} and \eqref{lem:EFback_time}
are proved in the same ways
as in Propositions~\ref{prop:KL_local_Lipschitz} and \ref{prop:MaxSol},
respectively.
\eqref{lem:EFback_inf} follows from the fact
that $\sup_{\xi \in \mathbb{R}} \Im \Psi_{\bm{s}(t_0 - t)}(w, \xi)$
is bounded by a constant
for any $t \in [0, t_0]$ and $w \in D_{t_0 - t} \cap \mathbb{H}_{\eta_{D_0}/4}$
by \eqref{eq:Poi_Koebe}.
\end{proof}

We now show that the solutions to \eqref{eq:KLforEF} form an evolution family.

\begin{theorem} \label{thm:SolEF}
Let $z(t) = z(t; s, z)$ be the solution to \eqref{eq:KLforEF}
with $z(s) = z$
for each $s \in [0, T)$ and $z \in D_s$.
Then the family of the mappings
\[
\phi_{t, s} \colon D_s \to D_t,\ z \mapsto z(t; s, z)
\]
parametrized by $(s, t) \in [0, T)^2_{\leq}$
is a chordal evolution family over $(D_t)_{t\in[0,T)}$
with angular residues $\lambda(t)=2t$.
\end{theorem}

\begin{proof}
$\phi_{t, s}$ is holomorphic
by the general theory of ODEs;
see, e.g., Theorem~8.4 in Chapter~1 of Coddington and Levinson~\cite{CL55}
or Exercise~3.3 in Chapter~I of Hale~\cite{Ha69}.
Although they treat vector fields which are jointly continuous
with respect to time and spatial variables,
suitable modifications are possible.
We can also rely on the results on weak holomorphic vector fields
given by Contreras, Diaz-Madrigal and Gumenyuk~\cite[Theorem~2.3]{CDMG13} instead.
The uniqueness of a solution to \eqref{eq:KLforEF}
implies the univalence and property~(EF.\ref{cond:EFcomp}) of $\phi_{t, s}$.
(EF.\ref{cond:EFnorm}) is trivial.

Let $\delta_0$ be the constant in Lemma~\ref{lem:EFback}~\eqref{lem:EFback_inf}.
From the relation between the solutions
to \eqref{eq:KLforEF} and to \eqref{eq:EFback},
we see that,
if $(s, t) \in [0, T)^2_{\leq}$ satisfies $t - s \leq \delta_0$,
then $\phi_{s, t}(D_s)$ contains
$D_s \cap \mathbb{H}_{\eta_{D_0}/2}$.
Since conformal mappings preserve
the degree of connectivity of domains,
this inclusion implies that
$D_t \setminus \phi_{t, s}(D_s)$
is an $\mathbb{H}$-hull.
In fact, the restriction $t - s \leq \delta_0$ is unnecessary
because, for any $(s, t) \in [0, T)^2_{\leq}$,
there exists a finite sequence
$s = t_0 \leq t_1 \leq \cdots \leq t_n = t$
such that $t_k - t_{k-1} \leq \delta_0$ for $k=1, \ldots, n$.
From the decomposition
\[
\phi_{t, s} = \phi_{t_n, t_{n-1}} \circ \cdots \circ \phi_{t_2, t_1} \circ \phi_{t_1, t_0},
\]
we can conclude that
$D_t \setminus \phi_{t, s}(D_s)$
is an $\mathbb{H}$-hull,
i.e., $\phi_{t,s}$ satisfies (H.\ref{ass:hull}) in the general case.

It remains to prove (H.\ref{ass:hydro}) and (H.\ref{ass:res}) for $\phi_{t,s}$.
Integrating \eqref{eq:KLforEF} and replacing $z(t)$ with $\phi_{t,s}(z)$, we have
\begin{equation} \label{eq:Int_KLforEF}
\phi_{t, s}(z)-z
=2\pi\int_s^t \int_{\mathbb{R}}\Psi_{\bm{s}(u)}(\phi_{u,s}(z), \xi)\, \nu_u(d\xi) \, du
\end{equation}
for all $(s,t)\in [0,T)^2_{\le}$ and $z \in D_s$.
Since $\Im \phi_{u,s}(z)$ is non-decreasing in $u$, the integrand satisfies
$\lvert \Psi_{\bm{s}(u)}(\phi_{u,s}(z), \xi) \rvert \leq 4(\pi \eta_{D_0})^{-1}$
for $z\in D_s\cap\mathbb{H}_{\eta_{D_0}}$ and $\xi\in\mathbb{R}$ by \eqref{eq:Poi_Koebe}.
Thus we have
\begin{equation} \label{eq:bdd_from_KLeq}
\lvert \phi_{t,s}(z)-z \rvert \le \frac{2t}{\eta_{D_0}},
\qquad (s,t)\in [0,T)^2_{\le},\ z\in D_s\cap\mathbb{H}_{\eta_{D_0}}.
\end{equation}
In particular,
$\phi_{t, s}(z) \to \infty$ as $z$ goes to $\infty$ through $D_s\cap\mathbb{H}_{\eta_{D_0}}$.
Hence it follows from \eqref{eq:Int_KLforEF} and the dominated convergence theorem that
\begin{align*}
&\lim_{\substack{z \to \infty \\ \Im z > \eta_0}} (\phi_{t, s}(z) - z) \\
&= 2 \pi \int_s^t \int_{\mathbb{R}}
\lim_{\substack{z \to \infty \\ \Im z > \eta_{D_0}}}
\Psi_{\bm{s}(u)}(\phi_{u, s}(z), \xi)
\, \nu_u(d\xi) \, du =0,
\end{align*}
which proves (H.\ref{ass:hydro}) for $\phi_{t, s}$.

To establish (H.\ref{ass:res}), fix $(s,t)\in [0,T)^2_{<}$ and $\theta\in (0,\pi/2)$.
Let $L>(2t/\eta_{D_0})+\max_{u\in [0,t]}r^{\text{out}}_{D_u}$.
By \eqref{eq:bdd_from_KLeq} there exists an angle $\theta^\prime\in (\theta, \pi/2)$ such that,
if $z\in \triangle_\theta \setminus \Bar{B}(0,L)$ and $s\le u\le t$,
then $\phi_{u,s}(z)\in \triangle_{\theta^\prime} \setminus \Bar{B}(0, L+2t/\eta_{D_0})$.
We now apply Lemma~\ref{lem:boundedness}
to a compact set $\Gamma:=\{\, \bm{s}(u)\mathrel{;} s\le u \le t \,\}\subset \mathbf{Slit}$ and the angle $\theta^\prime$.
Then by \eqref{eq:Int_KLforEF}, \eqref{eq:bdd_from_KLeq}, \eqref{eq:Poi_deriv_infty} and the dominated convergence theorem,
we have
\begin{align*}
&\lim_{\substack{z\to\infty \\ z\in\triangle_\theta}} z(\phi_{t,s}(z)-z)
= 2\pi \lim_{\substack{z\to\infty \\ z\in\triangle_\theta}} \int_s^t \int_{\mathbb{R}} z\Psi_{\bm{s}(u)}(\phi_{u,s}(z), \xi)\, \nu_u(d\xi) \, du \\
&=2\pi \lim_{\substack{z\to\infty \\ z\in\triangle_\theta}} \int_s^t \int_{\mathbb{R}}\bigl\{\phi_{u,s}(z)\Psi_{\bm{s}(u)}(\phi_{u,s}(z), \xi) \\
&\phantom{=2\pi \lim_{\substack{z\to\infty \\ z\in\triangle_\theta}} \int_s^t \int_{\mathbb{R}}\{ } - (\phi_{u,s}(z)-z)\Psi_{\bm{s}(u)}(\phi_{u,s}(z), \xi)\bigr\}\, \nu_u(d\xi) \, du \\
&=2\pi(t-s).
\end{align*}
This proves (H.\ref{ass:res}) with $\lambda(t)=2t$.
\end{proof}

Once $\phi_{t,s}$ above turns out to be an evolution family,
we can easily show that the Komatu--Loewner equations~\eqref{eq:KLforSlit_vec} and \eqref{eq:KLforEF} always have a global solution.
Namely the following holds:

\begin{corollary} \label{cor:global_solution}
For every initial value $\bm{s}(0)=\bm{s}_0\in\mathbf{Slit}$,
a local solution $\bm{s}(t)$ to the ODE~\eqref{eq:KLforSlit_vec} can be continued to the whole time interval $[0,\infty)$.
\end{corollary}

\begin{proof}
Assume that the lifetime $T$ of $\bm{s}(t)$ is finite.
By Theorem~\ref{thm:SolEF} there exists a unique evolution family $(\phi_{t,s})_{(s,t)\in [0,T)^2_{\le}}$ over $(D(\bm{s}(t)))_{t\in [0,T)}$ with $\lambda(t)=2t$.
Since $\sup_{0\le t<T}\lambda(t)=2T<\infty$,
it extends to an evolution family whose index set is $[0,T]^2_{\le}$ by Proposition~\ref{prop:prolong}.
Hence the limit $\bm{s}(T)=\lim_{t\to T}\bm{s}(t)$ exists in $\mathbf{Slit}$,
which contradicts the definition of $T$.
\end{proof}

By Theorem~\ref{thm:SolEF} and Corollary~\ref{cor:global_solution} we have arrived at Theorem~\ref{thm:result_unbdd}.
Theorem~\ref{thm:result_bdd} is now just a consequence of the following uniqueness of solutions:

\begin{proposition} \label{prop:UniSolSlit}
Let $J \subset [0, \infty)$ be a bounded interval and
$(\nu_t)_{t \in J}$ be a $\mathcal{P}(\mathbb{R})$-valued Lebesgue-measurable process.
If $\bigcup_{t \in J} \operatorname{supp} \nu_t$ is bounded,
then for each $t_0 \in J$ and $\bm{s}_0 \in \mathbf{Slit}$,
a solution $\bm{s}(t)$ to \eqref{eq:KLforSlit_vec} with $\bm{s}(t_0) = \bm{s}_0$ is unique on $J$.
\end{proposition}

\begin{proof}
The Lipschitz condition of the vector field $(t,\bm{s}) \mapsto \bm{b}(\nu_t, \bm{s})$
follows in the same way as in Chen and Fukushima~\cite[Lemma~4.1]{CF18}
if $\bigcup_{t \in J} \operatorname{supp} \nu_t$ is bounded.
\end{proof}

\section{Application}
\label{sec:application}

In this section, we consider two applications,
which will illustrate how a driving process $(\nu_t)_{t \in I}$
in the Komatu--Loewner equation~\eqref{eq:result_KL_LC} reflects the ``geometry'' and ``continuity'' of the corresponding Loewner chain.

\subsection{Bounded $\mathbb{H}$-hulls with local growth}
\label{sec:local_growth}

In the fourth paragraph of Section~\ref{sec:intro},
we have mentioned the Komatu--Loewner equation \eqref{eq:intro_1KL} corresponding to slit-mappings.
There, the driving process is reduced to a family of point masses $\delta_{\xi(t)}$
with $\xi(t)$ being a real-valued continuous function of $t$.
In this and the next subsections,
we study necessary and sufficient conditions for such reduction.

Below we consider the following setting:
let
$(F_t)_{t \in [0, T]}$
be a family of \emph{bounded} $\mathbb{H}$-hulls
growing in a parallel slit half-plane
$D = \mathbb{H} \setminus \bigcup_{j=1}^N C_j$.
Here, ``growing'' means strictly increasing:
$s < t$ implies $F_s \subsetneq F_t$.
A slit $\gamma$ of $D$
is an example of such a family of hulls.
For each $t \in [0, T]$, let
$g_t \colon D \setminus F_t \to D_t$
be the mapping-out function of $F_t$,
that is, a unique conformal mapping
onto a parallel slit half-plane $D_t$
with $\lim_{z \to \infty} (g_t(z) - z) = 0$%
\footnote{Since $F_t$ is bounded,
$g_t(z)$ has the Laurent expansion around $\infty$
after the Schwarz reflection across $\partial\mathbb{H}\setminus \overline{F_t}$.
We then see that
(H.\ref{ass:hydro}) and (H.\ref{ass:res}) can be combined into the simple normalization
$\lim_{z \to \infty} (g_t(z) - z) = 0$.}.
The BMD half-plane capacity
$\ell(t) := \operatorname{hcap}^D(F_t)$
is strictly increasing in $t$
as in \cite[Eq.~(8)]{BF08}, \cite[Lemma~5.15 (iii)]{CF18}.
Thus, except that $\ell(t)$ is not continuous,
$(g^{-1}_t)_{t \in [0, T]}$ is a \emph{reversed} Loewner chain, i.e.,
$(g^{-1}_{T-t})_{t \in [0, T]}$ is a Loewner chain.

Provided that $\ell(t)$ is continuous,
the Komatu--Loewner equation \eqref{eq:result_KL_LC}
with $f_t=g_{T-t}^{-1}$ holds in the setting above.
Our question is under what condition on $(F_t)_{t \in [0, T]}$ the measure-valued driving process $\nu_t$ is reduced to a real-valued continuous function.
One condition is given by the next definition.

\begin{definition}
\label{def:LocalGrowth}
\begin{enumerate}
\item \label{item:def_cross-cut}
A \emph{cross-cut} $C$ of a domain $E\subset \mathbb{C}$ is the trace of a simple curve
$c\colon [0, 1] \to \overline{E}$
with $c(0), c(1) \in \partial E$ and $c(0, 1) \subset E$.

\item \label{item:def_LGP}
A family of growing $\mathbb{H}$-hulls $(F_t)_{t \in [0, T]}$ in $D$ is said to have the \emph{local growth property}
if, for $\varepsilon > 0$, there exists a constant $\delta \in (0, T)$ with the following property:
for each $t \in [0, T - \delta]$,
some cross-cut $C$ of $D \setminus F_t$
with $\operatorname{diam}(C) < \varepsilon$
separates the increment
$F_{t + \delta} \setminus F_t$
from the point at infinity
in $D \setminus F_t$.
\end{enumerate}
\end{definition}

As B\"ohm~\cite[p.95]{Bo15} pointed out,
we can always assume in Definition~\ref{def:LocalGrowth} \eqref{item:def_LGP} that the endpoints of the cross-cut $C$ lie on the outer boundary $\partial (\mathbb{H} \setminus F_t)$ of $D \setminus F_t$.
In other words, $C$ is a cross-cut of $\mathbb{H}\setminus F_t$.
In this case, the cross-cut $C$ separates the domain $\mathbb{H}\setminus F_t$ into exactly two components:
one is bounded and the other is unbounded.
We shall designate the former component as $\operatorname{ins}C$ below.

In Definition~\ref{def:LocalGrowth},
$(F_t)_{t \in [0, T]}$ has the ``uniform continuity''
in terms of the diameter of cross-cuts.
This will be clearer
if we rephrase the local growth property as follows:
\emph{for any $\varepsilon > 0$,
there exists $\delta > 0$ such that,
if $0 \leq t - s \leq \delta$,
then some cross-cut $C$ of $D \setminus F_s$
with $\operatorname{diam}(C) < \varepsilon$
separates $F_t \setminus F_s$ from $\infty$ in $D \setminus F_s$.}
Indeed, even if $s > T - \delta$,
the difference $F_t \setminus F_s (\subset F_T \setminus F_{T - \delta})$
is separated from $\infty$ in $D \setminus F_{T - \delta}$
by a cross-cut $C$ of $D \setminus F_{T - \delta}$.
By definition, $C$ does not intersect
$F_s \setminus F_{T - \delta} (\subset F_T \setminus F_{T - \delta})$
except at its endpoints.
Thus, it is also a cross-cut of
$D \setminus F_s
= (D \setminus F_{T - \delta}) \setminus (F_s \setminus F_{T - \delta})$.

For the radial Loewner equation in $\mathbb{D}$,
Pommerenke~\cite{Po66} proved that
the local growth property holds if and only if
the driving process is reduced to a real-valued continuous function.
In the SLE context,
Lawler, Schramm and Werner~\cite[Theorem~2.6]{LSW01}
proved this equivalence for the chordal Loewner equation~\eqref{eq:intro_Loewner},
and Zhan~\cite[Proposition~2.1]{Zh04} mentioned the annulus case.
B\"ohm~\cite[Theorem~5.1]{Bo15} proved this fact for the radial Komatu--Loewner equation in circularly slit disks.

The next theorem shows one direction of the above-mentioned equivalence for the chordal Komatu--Loewner equation.

\begin{theorem} \label{thm:bdd_H-hull}
In the above-mentioned setting,
if
\begin{enumerate}
\renewcommand{\theenumi}{\arabic{enumi}}
\renewcommand{\labelenumi}{{\rm (P.\theenumi)}}

\item \label{item:LGP}
$(F_t)_{t \in [0, T]}$ has the local growth property,
\end{enumerate}
then
\begin{enumerate}
\setcounter{enumi}{1}
\renewcommand{\theenumi}{\arabic{enumi}}
\renewcommand{\labelenumi}{{\rm (P.\theenumi)}}

\item \label{item:sinKL}
$\ell(t)$ is continuous, and
there exists a continuous function $\xi(t)$ on $[0, T]$ such that
\begin{equation}
\label{eq:app_sinKLeq}
\frac{\partial g_t(z)}{\partial \ell(t)}
:= \lim_{h \to 0}
	\frac{g_{t+h}(z) - g_t(z)}{\ell(t+h) - \ell(t)}
= -\pi \Psi_{D_t}(g_t(z), \xi(t))
\end{equation}
for every $z \in D \setminus F_t$ and $t \in [0, T]$.
\end{enumerate}
\end{theorem}

We shall establish the converse
$\mathrm{(P.\ref{item:sinKL})} \Rightarrow \mathrm{(P.\ref{item:LGP})}$
in the next subsection,
adding one more condition equivalent to (P.\ref{item:sinKL}) from another viewpoint.
In the rest of this subsection,
we focus on the proof of Theorem~\ref{thm:bdd_H-hull}.

Our proof goes along the line of Lawler, Schramm and Werner~\cite[Theorem~2.6]{LSW01} with suitable modification.
(See also Pommerenke's original argument~\cite{Po66}.)
Before proof,
we recall the definition of extremal length.
Let $\Gamma$ be a path family in a planar domain,
that is,
a set consisting of rectifiable paths%
\footnote{
In general, the definition of path family allows an element
$\gamma \in \Gamma$
to be a countable union of curves,
but in what follows, we treat only connected paths.}.
The \emph{extremal length} of $\Gamma$ is defined by
\[
\operatorname{EL}(\Gamma)
:= \sup_{\rho}
	\frac{\left(\inf_{\gamma \in \Gamma} \int_{\gamma} \rho(z) \, \lvert dz \rvert\right)^2 }%
	{ \int_U \rho(z)^2 \,dx \, dy },
\quad z = x + iy.
\]
Here, we fix some domain $U$ containing all paths of $\Gamma$,
and the supremum is taken over all non-negative Borel measurable functions $\rho$ on $U$
with $0 < \int_U \rho^2 \,dx \, dy < \infty$.
$\operatorname{EL}(\Gamma)$ is independent of the choice of $U$
and moreover conformally invariant.
We refer the reader to
Chapter~4 of Ahlfors~\cite{Ah73}
or
Chapter~IV of Garnett and Marshall~\cite{GM05}
for the properties of extremal length.

We now begin the proof of Theorem~\ref{thm:bdd_H-hull}.
Until the end of this subsection,
suppose that $(F_t)_{t\in[0, T]}$ has the local growth property.
Let
$L:=\sup\{\, \lvert z \rvert \mathrel{;} z\in F_T\cup (\mathbb{H}\setminus D) \,\}$
and
$\varepsilon_0:=d^{\text{Eucl}}(F_T, \mathbb{H}\setminus D)^2/4$.
For each $(s, t) \in [0, T]^2_{<}$,
we denote by $\Gamma_{t,s}$ the set of rectifiable cross-cuts of $\mathbb{H}\setminus F_s$
which separate $F_t \setminus F_s$ from $B(0, L+2)^c$ in $D \setminus F_s$.
We note that $\Gamma_{u,s}\subset \Gamma_{u,t}\cap \Gamma_{t,s}$ holds for $s<t<u$.

\begin{lemma} \label{lem:s8_diam_bound}
For each $\varepsilon\in (0,\varepsilon_0)$
there exists a constant $\delta=\delta(\varepsilon)>0$ such that,
if $0\le s\le t\le u\le T$ with $u-s<\delta$ then
\begin{equation} \label{eq:s8_diam_bound}
\inf_{\gamma\in g_t(\Gamma_{u,s})} \operatorname{diam}(\operatorname{ins}\gamma)
\le r(\varepsilon):=\frac{4\sqrt{2}\pi(L+2)}{\sqrt{\log(1/\varepsilon)}}.
\end{equation}
\end{lemma}

\begin{proof}
Let $\varepsilon\in (0,\varepsilon_0)$.
By the local growth property
there exists $\delta$ such that,
for $(s,u)\in[0,T]^2_{<}$ with $u-s<\delta$,
some cross-cut $C$ with $\operatorname{diam}(C)<\varepsilon$ separates $F_u\setminus F_s$ from $\infty$ in $D\setminus F_s$.
For a fixed $z_0 \in C$,
let $\Gamma^{\prime}$ be the set of rectifiable paths
separating the inner and outer boundaries of the annulus
$\mathbb{A}(z_0; \varepsilon, \sqrt{\varepsilon})
= \{\, z \in \mathbb{C} \mathrel{;} \varepsilon < \lvert z - z_0 \rvert < \sqrt{\varepsilon} \,\}$.
Since any $\gamma^{\prime} \in \Gamma^{\prime}$
contains some path $\gamma \in \Gamma_{u,s}$
in the sense that $\gamma \subset \gamma^{\prime}$,
the extension rule
(\cite[Theorem~4.1]{Ah73}, \cite[Eq.~(3.2)]{GM05})
implies that
$\operatorname{EL}(\Gamma_{u,s})
\leq \operatorname{EL}(\Gamma^{\prime})
= 4\pi/\log (1/\varepsilon)$.
Here, see Section~1, Chapter~IV of \cite{GM05}
for getting the value of $\operatorname{EL}(\Gamma^{\prime})$.
Moreover, by the conformal invariance of extremal length,
\begin{equation} \label{eq:EL_bound}
\operatorname{EL}(g_t(\Gamma_{u,s}))
= \operatorname{EL}(\Gamma_{u,s})
\leq \frac{4 \pi}{\log (1/\varepsilon)},
\quad s\le t\le u.
\end{equation}

$g_t(B(0, L+2)\cap (D \setminus F_t))$ is bounded.
Indeed, a consequence \cite[Lemma~3.9]{Mu19spa} from the hydrodynamic normalization implies that
\[
\operatorname{diam} (g_t(B(0, L+2) \cap (D \setminus F_t))) \leq 4 (L + 2).
\]
Thus, by \eqref{eq:EL_bound} and the definition of extremal length,
\[
\inf_{\gamma \in g_t(\Gamma_{u,s})}
\left( \int_{\gamma} \, \lvert dz \rvert \right)^2
\leq \frac{4\pi}{\log (1/\varepsilon)}
\int_{g_t(B(0, L+2) \cap (D \setminus F_t))} \, dx \, dy
\leq \frac{32 \pi^2 (L + 2)^2}{\log (1/\varepsilon)}.
\]
Hence the conclusion \eqref{eq:s8_diam_bound} follows.
\end{proof}

Since $g_s(F_t\setminus F_s)\subset \operatorname{ins}\gamma$
for each $\gamma\in g_s(\Gamma_{t,s})$
and $\lim_{\varepsilon\to 0}r(\varepsilon)=0$,
we have the following by \eqref{eq:s8_diam_bound} with $u=t$:

\begin{corollary} \label{cor:s8_driver}
For each $t\in [0,T)$, there exists a point $\xi(t)\in \partial\mathbb{H}$ such that
\begin{equation} \label{eq:shrink}
\bigcap_{\Delta t>0}\overline{g_t(F_{t+\Delta t}\setminus F_t)}=\{\xi(t)\},
\quad t\in[0, T).
\end{equation}
\end{corollary}

\begin{proposition} \label{prop:s8_hcap_conti}
The function
$\ell(t):=\operatorname{hcap}^D(F_t)$
is continuous on $[0,T]$.
\end{proposition}

\begin{proof}
Let $\varepsilon\in (0,\varepsilon_0)$,
$(s,t)\in [0,T]^2_{<}$ with $t-s<\delta(\varepsilon)$,
and $g_{t,s}:=g_s\circ g_t^{-1}$.
The function $g_{t,s}$ maps $D_t=g_t(D\setminus F_t)$ conformally onto $D_s\setminus g_s(F_t\setminus F_s)$.
On account of the boundary correspondence,
for any $\gamma\in g_t(\Gamma_{t,s})$
\begin{itemize}
\item $\Im g_{t,s}(\xi)=0$
for every $\xi\in \partial\mathbb{H}\setminus \overline{\operatorname{ins}\gamma}$;
\item $\Im g_{t,s}(\xi)\le \sup_{z\in g_s(F_t\setminus F_s)}\Im z \le r(\varepsilon)$
for $\mathbf{Leb}$-a.e.\ $\xi\in \overline{\operatorname{ins}\gamma}\cap \partial\mathbb{H}$.
\end{itemize}
In particular,
as $\mu(g_{t,s};d\xi)=\pi^{-1}\Im g_{t,s}(\xi)\, d\xi$,
we have
\begin{equation} \label{eq:s8_supp_mu}
\operatorname{supp}[\mu(g_{t,s}; \mathord{\cdot})]
\subset \overline{\operatorname{ins}\gamma}\cap \partial\mathbb{H},
\quad \gamma\in g_t(\Gamma_{t,s}).
\end{equation}
Since
\[
\operatorname{hcap}^{D_s}(g_s(F_t \setminus F_s))
= \frac{1}{\pi}\int_{\operatorname{supp}[\mu(g_{t,s}; \mathord{\cdot})]}\Im g_{t,s}(\xi)\, d\xi,
\]
it follows from \eqref{eq:s8_diam_bound} with $u=t$ that
\begin{equation} \label{eq:s8_ell_uc}
\ell(t) - \ell(s)
= \operatorname{hcap}^{D_s}(g_s(F_t \setminus F_s))
\leq r(\varepsilon)^2/\pi.
\end{equation}
This proves the continuity of $\ell$.
\end{proof}

\begin{proposition} \label{prop:s8_xi_uc}
The points $\xi(t)$, $t\in [0,T)$, in Corollary~\ref{cor:s8_driver} form a uniformly continuous function on $[0,T)$.
\end{proposition}

\begin{proof}
Suppose that $\varepsilon\in (0,\varepsilon_0)$, $(s,t)\in [0,T]^2_{<}$ with $t-s<\delta(\varepsilon)$, and $u\in (t,T)$ with $u-s<\delta(\varepsilon)$.
Let $r^\prime\in (0,\eta_{D_T})$.
By Lemma~\ref{lem:s8_diam_bound} we can take a disk $U$ centered at a point of $\partial\mathbb{H}$ with radius less than $2r(\varepsilon)+r^\prime$
so that, for some $\gamma\in g_t(\Gamma_{u,s})$,
it holds that
$g_t(F_u\setminus F_t)\subset \operatorname{ins}\gamma\subset U$.
Now by \eqref{eq:int_rep}
\begin{equation} \label{eq:int_rep_LGP}
g_{t,s}(z)
= z+\pi \int_{\operatorname{supp}[\mu(g_{t,s}; \mathord{\cdot})]}\Psi_{D_t}(z,\xi)\cdot \pi^{-1}\Im g_{t,s}(\xi)\, d\xi,
\end{equation}
and thus it follows from \eqref{eq:Poi_Koebe} and \eqref{eq:s8_ell_uc} that
\begin{equation} \label{eq:s8_g_conti}
\lvert g_{t,s}(z)-z \rvert
\leq \frac{4}{r^{\prime}}(\ell(t)-\ell(s))
\le \frac{4r(\varepsilon)^2}{\pi r^\prime},
\quad z \in D_t\setminus \overline{U}.
\end{equation}
Since a cross-cut of $g_t(\Gamma_{u,s})$, which separates $g_t(F_u\setminus F_t)$ from $\infty$ in $D_t$, is mapped to a cross-cut of $g_s(\Gamma_{u,s})$ by the univalent mapping $g_{t,s}$,
a geometric consideration combined with \eqref{eq:s8_g_conti} shows
\[
\sup_{\substack{z\in g_s(F_u\setminus F_s) \\ w\in g_t(F_u\setminus F_t)}}\lvert z-w \rvert
\le 2(2r(\varepsilon)+r^\prime)+\frac{4r(\varepsilon)^2}{\pi r^\prime}.
\]
Hence, in view of Corollary~\ref{cor:s8_driver} we have
\[
\lvert \xi(s)-\xi(t) \rvert
\leq 2(2r(\varepsilon)+r^{\prime})+\frac{4r(\varepsilon)^2}{\pi r^\prime}.
\]
Letting $\delta\to 0$ and then $r^\prime\to 0$ yields
\[
\limsup_{\delta\to 0}\sup_{\substack{(s,t)\in [0,T)^2_{<} \\ 0<t-s<\delta}} \lvert \xi(s)-\xi(t) \rvert = 0,
\]
as desired.
\end{proof}

By Proposition~\ref{prop:s8_xi_uc},
$\xi(t)$ is extended to a continuous function on $[0,T]$.
Finally we give a proof of Theorem~\ref{thm:bdd_H-hull}.

\begin{proof}[Proof of Theorem~\ref{thm:bdd_H-hull}]
Only in this proof, we use a convention $t^\prime:=T-t$ for each $t\in [0,T]$.
To apply our general theory, we put
\begin{align*}
\phi_{u,s}&:=g_{s^\prime,u^\prime}=g_{u^\prime}\circ g_{s^\prime}^{-1},
&(s,u) \in [0,T]^2_{\le}, \\
\lambda(t)&:=\ell(T)-\ell(t^\prime),
&t \in [0,T].
\end{align*}

Fix $0\le t_0< t\le T$.
Let $\varepsilon\in (0,\varepsilon_0)$ and
$t_0<s\le t\le u\le T$ with $u-s<\delta(\varepsilon)$.
We can take $\alpha\in (s^\prime,T)$ so that $\alpha-u^\prime<\delta(\varepsilon)$.
Then by \eqref{eq:shrink} and \eqref{eq:s8_supp_mu},
\begin{equation} \label{eq:s8_xi_and_supp}
\{\xi(s^\prime)\}\cup \operatorname{supp}[\mu(g_{s^\prime,u^\prime}; \mathord{\cdot})]\subset \overline{\operatorname{ins}\gamma}\cap \partial\mathbb{H}
\quad \text{for}\ \gamma\in g_{s^\prime}(\Gamma_{\alpha,u^\prime}).
\end{equation}
Hence by \eqref{eq:s8_diam_bound} we have
\[
\operatorname{supp}[\mu(\phi_{u,s}; \mathord{\cdot})]\subset [\xi(s^\prime)-r(\varepsilon), \xi(s^\prime)+r(\varepsilon)].
\]
Now it is not difficult to see that
the normalized measure
$\mu(\phi_{u,s}; \mathord{\cdot})/\mu(\phi_{u,s}; \mathbb{R})$
converges weakly to
$\delta_{\xi(t^\prime)}(\mathord{\cdot})$
as $s,u\to t$.
By Corollary~\ref{cor:EFdiff},
we have
\begin{equation} \label{eq:s8_rev_EF}
\begin{split}
\frac{\partial \phi_{t,t_0}(z)}{\partial \lambda(t)}
&= \pi \int_{\mathbb{R}}
\Psi_{D_{t^\prime}}(\phi_{t,t_0}(z), \tilde{\xi}) \, \delta_{\xi(t^\prime)}(d\tilde{\xi}) \\
&=\pi \Psi_{D_{t^\prime}}(\phi_{t,t_0}(z), \xi(t^\prime)).
\end{split}
\end{equation}
Substituting $g_{t_0^\prime}(z)$ into the $z$ in this equation and taking time-reversal,
we obtain \eqref{eq:app_sinKLeq} for $t\in [0,T)$.

In the above, the case $t=T$ in \eqref{eq:app_sinKLeq} is excluded since we have taken an extra parameter $\alpha\in (s^\prime,T)$ in \eqref{eq:s8_xi_and_supp}.
Nevertheless, \eqref{eq:app_sinKLeq} holds also for $t=T$,
since $g_t(z)$ and the right-hand side of \eqref{eq:app_sinKLeq} are continuous at $t=T$.
This is just a basic calculus once we apply such a reparametrization as in \eqref{eq:KLeq_hcap_parametrized},
which is possible because $\ell(t)$ is strictly increasing and continuous.
\end{proof}

\subsection{Continuity of $\mathbb{H}$-hulls in Carath\'eodory's sense}
\label{sec:RSLC}

In this subsection,
we introduce another condition~(P.\ref{item:RSLC})
and prove that (P.\ref{item:LGP})--(P.\ref{item:RSLC})
are mutually equivalent.

We first define the left continuity of $(F_t)_{t \in [0, T]}$.
We use a classical concept,
Carath\'eodory's kernel convergence of domains,
following Section~5, Chapter~V of Goluzin~\cite{Go69}.

\begin{definition}
\label{def:kernel_convergence}

\begin{enumerate}

\item \label{item:kernel}
Let $G_n$, $n \in \mathbb{N}$, be domains in $\mathbb{C}$
and $a \in \mathbb{C}$.
The \emph{kernel $\ker_a (G_n)_{n \in \mathbb{N}}$ with respect to $a$}
is defined as
the connected component of the set
$\{\, z \in \mathbb{C} \mathrel{;}
B(z; r) \subset \bigcap_{n \geq N} G_n
\ \text{for some} \ r > 0 \ \text{and some} \ N \in \mathbb{N} \,\}$
containing $a$.

\item \label{item:kernel_convergence}
Let $I$ be an interval, $t_0 \in I$, and $a \in \mathbb{C}$.
Let $G_t$, $t \in I$, be domains in $\mathbb{C}$.
We say that
$G_t$ \emph{converges} to $G_{t_0}$ as $t \to t_0$
\emph{in the sense of kernel} (or \emph{in Carath\'eodory's sense})
with respect to $a$
if $\ker_a (G_{s_n})_{n \in \mathbb{N}} = G_{t_0}$ for every sequence
$(s_n)_{n \in \mathbb{N}}$ of $I$ with $s_n \to t_0$.

\item \label{item:left_continuity}

$(F_t)_{t \in [0, T]}$ is said to be \emph{left continuous} at $t_0 \in (0, T]$
(in the sense of kernel convergence or in Carath\'eodory's sense)
if the domain $D \setminus F_t$ converges to $D \setminus F_{t_0}$
as $t$ increases to $t_0$
in the sense of kernel with respect to some
(indeed, any) $a \in D \setminus F_T$.

\end{enumerate}

\end{definition}

Since $(F_t)_{t \in [0, T]}$ is increasing,
it automatically holds that
$D \setminus F_{t_0} \subset \ker_a (D \setminus F_{s_n})_{n \in \mathbb{N}}$
for any sequence $(s_n)_{n \in \mathbb{N}}$
with $s_n \uparrow t_0$.
Our left continuity requires
that this inclusion should be equality.
The right continuity of $(F_t)_{t \in [0, T]}$ is defined in the same manner.
This right continuity follows from the property~\eqref{eq:shrink},
which is also called the ``right continuity with limit $\xi(t)$''
in Section~1, Chapter~4 of Lawler~\cite{La05}.

\begin{lemma}[{Murayama~\cite[Lemma~4.4]{Mu19spa}}]
\label{lem:RCLC}
If $(F_t)_{t \in [0, T]}$ is continuous
(in the sense of kernel convergence),
then $\ell(t) = \operatorname{hcap}^D(F_t)$ is continuous.
\end{lemma}

The author proved in the previous paper~\cite{Mu19spa}
that the property~\eqref{eq:shrink} and left continuity hold
if and only if the chordal Komatu--Loewner equation~\eqref{eq:intro_1KL} holds.
Although some care is required
on the difference between the settings
in the previous and present papers,
we can prove the following theorem:

\begin{theorem} \label{thm:LGP_RSLC}
In Theorem~\ref{thm:bdd_H-hull},
{\rm (P.\ref{item:LGP})}, {\rm (P.\ref{item:sinKL})}, and
the following {\rm (P.\ref{item:RSLC})} are mutually equivalent:
\begin{enumerate}
\setcounter{enumi}{2}
\renewcommand{\theenumi}{\arabic{enumi}}
\renewcommand{\labelenumi}{{\rm (P.\theenumi)}}

\item \label{item:RSLC}
the property~\eqref{eq:shrink} holds
for some continuous function $\xi(t)$ on $[0, T]$,
and $(F_t)_{t \in [0, T]}$ is left continuous on $(0, T]$
in the sense of Definition~\ref{def:kernel_convergence}~\eqref{item:left_continuity}.
\end{enumerate}
\end{theorem}

\begin{proof}
By Proposition~\ref{prop:s8_hcap_conti} and Lemma~\ref{lem:RCLC},
the function $\ell(t)$ is increasing and continuous
if one of (P.\ref{item:LGP})--(P.\ref{item:RSLC}) holds.
In particular,
if we take any increasing and continuous function
$\theta(t)$ on $[0, T]$
and perform time-change as
$\tilde{F}_t := F_{\theta^{-1}(t)}$,
$\tilde{g}_t := g_{\theta^{-1}(t)}$,
$\tilde{D}_t := D_{\theta^{-1}(t)}$, and
$\tilde{\xi}(t) := \xi(\theta^{-1}(t))$,
then the conditions~(P.\ref{item:sinKL}) and (P.\ref{item:RSLC})
on $(F_t)_{t \in [0, T]}$
are equivalent to
those on $(\tilde{F}_t)_{t \in [0, \theta(T)]}$,
respectively.
(See also \eqref{eq:KLeq_time-change}.)
We can easily prove,
using the uniform continuity of $\theta$ on $[0, T]$,
that this invariance under reparametrization is the case also for (P.\ref{item:LGP}).
Therefore,
we may reparametrize $(F_t)_{t \in [0, T]}$
whenever it is necessary
to make our setting consistent to those in relevant studies.

In addition to the invariance under reparametrization,
we note that the conditions~(P.\ref{item:LGP}) and (P.\ref{item:RSLC})
are independent of whether parallel slits exist or not.
To be precise,
if $(F_t)_{t \in [0, T]}$ enjoys (P.\ref{item:LGP}) or (P.\ref{item:RSLC})
as a family of $\mathbb{H}$-hulls \emph{in $D$},
then so does it, respectively,
as a family of $\mathbb{H}$-hulls \emph{in $\mathbb{H}$},
and vice versa.
This is clear from definition (see \cite[Proposition~4.7]{Mu19spa} for example).

Lawler, Schramm and Werner~\cite[Theorem~2.6]{LSW01} showed
that (P.\ref{item:LGP}) is equivalent to (P.\ref{item:sinKL})
as long as $(F_t)_{t \in [0, T]}$ is regarded
as a family of hulls in $\mathbb{H}$.
We also know
from Murayama~\cite[Theorem~4.6]{Mu19spa}
that (P.\ref{item:sinKL}) and (P.\ref{item:RSLC})
are equivalent both in $\mathbb{H}$ and in $D$,
but this result applies to a right-open interval $[0, T)$ only.
It remains to extend it to $t = T$.

The implication
$\mathrm{(P.\ref{item:sinKL})} \Rightarrow \mathrm{(P.\ref{item:RSLC})}$
at $t = T$ is trivial,
because Lemma~\ref{lem:EFineq} shows
the continuity of the corresponding vector $\bm{s}(t)$ of slit endpoints.
We observe
$\mathrm{(P.\ref{item:RSLC})} \Rightarrow \mathrm{(P.\ref{item:sinKL})}$
at $t = T$.
The (left) continuity of $\ell(t)$ at $t = T$
follows from the left continuity of $(F_t)_{t \in [0, T]}$
\cite[Lemma~4.4~(i)]{Mu19spa}.
Also the left continuity of $(F_t)_{t \in [0, T]}$
and the kernel theorem~\cite[Theorem~3.8]{Mu19spa}
imply that
$\bm{s}(t) \to \bm{s}(T)$
and
$g_t(z) \to g_T(z)$, $z \in D \setminus F_T$, as $t \uparrow T$.
Now \eqref{eq:app_sinKLeq} at $t=T$ holds in the same way as in the last paragraph of the proof of Theorem~\ref{thm:bdd_H-hull}.
\end{proof}

\subsection{Multiple slits from outer boundary}
\label{sec:multiple_slits}

Let $D$ be a parallel slit half-plane and
$\gamma_k \colon [0, T] \to \overline{D}$, $k = 1, \ldots, n$,
be $n$ disjoint simple curves with
$\gamma_k(0) \in \partial \mathbb{H}$
and
$\gamma_k(0, T] \subset D$.
We put
$F_t := \bigcup_{k=1}^n \gamma_k(0, t]$
and consider the mapping-out function
$g_t \colon D \setminus F_t \to D_t$.
For each $k$ and $t$,
there exists a unique point $\xi_k(t) \in \partial \mathbb{H}$ such that
$\lim_{z \to \xi_k(t)} g_t(z) = \gamma_k(t)$
by the boundary correspondence.

\begin{proposition} \label{prop:app_multiple-path}

\begin{enumerate}

\item \label{prop:contMultiDriver}
$\ell(t) := \operatorname{hcap}^D(F_t)$ and
$\xi_k(t)$, $k=1, \ldots, n$,
are continuous in $t$.

\item \label{prop:app_multiKLeq}
There exist an $m_{\ell}$-null set $N \subset [0, T]$ and
weights $c_1(t), \ldots, c_n(t) \geq 0$ with $\sum_{k=1}^n c_k(t) = 1$ such that
\begin{equation} \label{eq:app_multiKLeq}
\tilde{\partial}^{\ell}_t g_t(z)
= -\pi \sum_{k=1}^n c_k(t) \Psi_{D_t}(g_t(z), \xi_k(t))
\end{equation}
holds for every $t \in [0, T] \setminus N$ and $z \in D \setminus F_t$.
Here, each weight $c_k(t)$ can be chosen to be $\mathcal{B}[0,T]^{m_\ell}/\mathcal{B}[0,1]$-measurable and is unique for $m_\ell$-a.e.\ $t$.
\end{enumerate}

\end{proposition}

\begin{proof}
We omit the detail because the proof of this proposition is similar to that of Theorem~\ref{thm:bdd_H-hull}.
We just make some comments.
Firstly, even if the support
$\operatorname{supp}[\mu(\phi_{u, s}; \mathord{\cdot})]$
for $\phi_{u, s} = g_{T-u} \circ g_{T-s}$
shrinks to the $n$-point set
$\{ \xi_1(T-t), \ldots, \xi_n(T-t) \}$
as $s, u \to t$,
the normalized measure
$\mu(\phi_{u, s}; \mathord{\cdot})/\mu(\phi_{u, s}; \mathbb{R})$
does not necessarily converge weakly.
The mass on a neighborhood of each $\xi_k(T-t)$ may oscillate.
For this reason, Theorem~\ref{thm:EFmain} should be employed in place of Corollary~\ref{cor:EFdiff} in the present case.
The second comment is that, apart from our method, we can also obtain \eqref{prop:contMultiDriver} from Lemmas~2.38 and 2.43 of B\"ohm~\cite{Bo15}.
Lastly, the measurability and uniqueness of $c_k(t)$ follows from the $\mathcal{B}[0,T]^{m_\ell}/\mathcal{B}(\mathcal{P}(\mathbb{R}))$-measurability and uniqueness of $\nu_t=\sum_k c_k(t)\delta_{\xi_k(t)}$.
Indeed, the measurability of $\nu_t$ implies that, for each $f\in C_c(\mathbb{R})$, the function
$t\mapsto \sum_k c_k(t)f(\xi_k(t))$ is $\mathcal{B}[0,T]^{m_\ell}/\mathcal{B}(\mathbb{R})$-measurable.
A suitable choice of $f$ yields the measurability of $c_k(t)$.
\end{proof}

We make some remarks on \eqref{eq:app_multiKLeq}.
In contrast to the equation~\eqref{eq:app_sinKLeq} for ``single-slit mappings,''
one has not formulated so far
any explicit condition on $(F_t)_{t \in [0, T]}$
that is equivalent to
the equation~\eqref{eq:app_multiKLeq} for ``multiple-slit mappings.''
For example,
we can replace disjoint paths in Proposition~\ref{prop:app_multiple-path}
by disjoint hulls of local growth.
See Starnes~\cite{St19}.
However,
this replacement does not give
a necessary condition for \eqref{eq:app_multiKLeq}.
We also have to consider the more complicated situation
in which one path or hull touches other one.
In fact,
B\"ohm and Schlei{\ss}inger~\cite{BS16} studied
the $t$-differentiability of the mapping-out function $g_t(z)$
for the union of two paths $\gamma_1(0, t]$ and $\gamma_2(0, t]$
such that $\gamma_1(0) = \gamma_2(0)$.
They gave a condition on $\gamma_1$ and $\gamma_2$
sufficient for $g_t(z)$ to be (right-)differentiable
at $t=0$ \cite[Theorem~1.5]{BS16},
while constructing an example of a pair $(\gamma_1, \gamma_2)$
for which $t \mapsto g_t(z)$ is not differentiable at $t=0$.

In the above case of B\"ohm and Schle{\ss}inger,
the two paths $\gamma_1$ and $\gamma_2$ touch each other at a single time $t=0$.
A touch at distinct times, i.e., $\gamma_1(s)=\gamma_2(t)$ for some $s\neq t$ should also be taken into account.
For example, even if we assume that the driving functions enjoy $\xi_1(t)\neq \xi_2(t)$ for all $t$,
the boundary correspondence does not rule out the possibility of $\gamma_2(t)$ lying on $\gamma_1(0,t)$.
See Schleissinger~\cite[Theorem~1.2]{Sc12} to find a sufficient condition (on $\xi_k$'s) for the family of $\mathbb{H}$-hulls generated by \eqref{eq:app_multiKLeq} to split into $n$ disjoint slits
in the case of chordal Loewner equation in $\mathbb{H}$.

\section{Concluding remarks}
\label{sec:conclusion}

In this section,
we conclude this paper
with some remarks on the relation of our study to previous and future works.

\subsection{Remaining problems}
\label{sec:problems}

We recall from Section~\ref{sec:proof3}
that a solution to the Komatu--Loewner equation for slits~\eqref{eq:KLforSlit_vec}
is not proved to be unique in Proposition~\ref{prop:LocSolSlit}.
However,
as we believe that a driving process $\nu_t$ has
all the information of the corresponding evolution family,
\emph{the uniqueness of slit motion is plausible}.
For proof of this uniqueness,
a closer study will be required
on the local Lipschitz continuity of the function
$\mathbf{Slit} \ni \bm{s} \mapsto \Psi_{\bm{s}}(z, \xi)$~\cite[Theorem~9.1]{CFR16}.
If \emph{the Lipschitz constant turns out to be independent of $\xi \in \mathbb{R}$},
then we can drop the assumption
that $\bigcup_{t \in J} \operatorname{supp} \nu_t$ is bounded
in Proposition~\ref{prop:UniSolSlit}.
A possible way to show this independence is to improve
the interior variation method developed in Section~12
of Chen, Fukushima and Rohde~\cite{CFR16}.

In this paper, we have considered the chordal case on parallel slit half-planes.
It is a natural problem
to construct analogous theories of Loewner chains and evolution families
on circularly slit disk (radial case) and
on circularly slit annuli (bilateral case).
See Komatu~\cite{Ko50},
Bauer and Friedrich~\cite{BF04, BF06, BF08},
Fukushima and Kaneko~\cite{FK14},
B\"ohm and Lauf~\cite{BL14},
and B\"ohm~\cite{Bo15}
for relevant studies.

\subsection{Reversed Loewner chains and SLE}
\label{sec:SLE}

In Sections~\ref{sec:intro} and \ref{sec:application},
we have considered reversed evolution families and reversed Loewner chains.
As is illustrated in Section~\ref{sec:application},
we can \emph{derive} a differential equation for reversed families
without any additional effort.
However, matters are different in regard to \emph{solving} the equation.
We can find this difference
in the corresponding Komatu--Loewner equation for slits
$\dot{\bm{s}}(t) = - 2\bm{b}(\nu_t, \bm{s}(t))$.
The $y$-coordinates $y_j(t)$ ($1 \leq j \leq N$) are decreasing in $t$
for this backward equation
whereas they are increasing
for the forward equation~\eqref{eq:KLforSlit_vec}.
Thus,
even if $\operatorname{supp} \nu_t \subset [-a, a]$ for some $a > 0$,
the vector $\bm{s}(t)$ of slit endpoints may not be defined globally.
If it is not defined globally,
then there exists some $\zeta \in (0, \infty)$ such that
$\lim_{t \uparrow \zeta} \min_{1 \leq j \leq N} y_j(t) = 0$.
In this case,
the slits of $D_t$ may be absorbed
by the outer boundary $\partial \mathbb{H}$ at $t = \zeta$.
For a reversed Loewner chain
$(f_t)_{t \in [0, T]}$ with $f_0(D_0) = D$,
such absorption is closely related to the phenomenon
that the hulls $F_t = D \setminus f_t(D_t)$ ``swallow''
some part of the slits of $D$.

The author studied the case $\zeta < \infty$
with regard to \eqref{eq:intro_1KL}
in the previous paper~\cite{Mu19jeeq}.
That paper presents certain results
on the behavior of $\bm{s}(t)$ around $t = \zeta$.
On the other hand,
we can say little about the behavior of $F_t$ around $t = \zeta$.
In particular,
it remains to be discussed
how the ``limit'' of $F_t$ as $t \uparrow \zeta$ is constructed.
It is reasonable to believe the following:
we can define the limit hull $F_{\zeta}$
in such a way that
$\lim_{t \uparrow \zeta} y_j(t) = 0$
if and only if
$C_j \cap F_{\zeta} \neq \emptyset$.
Here, note that,
even if $(g_t)_{t \in [0, \zeta)}$ obeys \eqref{eq:app_sinKLeq},
the local growth property cannot be expected
at $t = \zeta$ anymore.
We cannot exclude the possibility
that the driving function $\xi(t)$ diverges as $t \uparrow \zeta$.
The author hopes
that the present work will help
to treat such a subtle situation.

The study on reversed Loewner chains above
plays a role in defining and analyzing extensions of SLE
to multiply connected domains.
As $\mathrm{SLE}_{\kappa}$ ($\kappa > 0$) on $\mathbb{H}$ is defined
as a reversed Loewner chain $(g_t)_{t \geq 0}$
generated by \eqref{eq:intro_Loewner}
with $\xi(t) = \sqrt{\kappa} B_t$
($B_t$ is the one-dimensional standard Brownian motion),
the \emph{stochastic Komatu--Loewner evolution} (SKLE for short) on $D$
is defined as a reversed chain $(g_t)_{0 \leq t < \zeta}$
generated by \eqref{eq:intro_1KL}
with $\xi(t)$ determined by a certain stochastic differential equation.
See Bauer and Friedrich~\cite{BF04, BF06, BF08}
and Chen and Fukushima~\cite{CF18}.
Zhan~\cite{Zh04PhD} defined \emph{harmonic random Loewner chains}
in a different way to extend SLE to finite Riemann surfaces.
Although he did not use \eqref{eq:intro_1KL} in his definition,
\eqref{eq:intro_1KL} also appeared as a byproduct of his results.
Lawler~\cite{La06} and Jahangoshahi and Lawler~\cite{JL18+}
studied further different ways to extend SLE, respectively,
without using \eqref{eq:intro_1KL}.
From this context,
the following question arises naturally:
\emph{how are these different extensions of SLE
related to each other?}
Answering this question will make the Komatu--Loewner equation applicable
to problems that have been studied by other methods.
Such a relation of our theory to SLE is yet to be investigated.

\section*{Acknowledgments}

The author wishes to express his gratitude
to Professor Zhen-Qing Chen
for his valuable comments, part of which motivates Remark~\ref{rem:Doob},
and to Professor Ikkei Hotta
for giving me the opportunity to talk in International Workshop on Complex Dynamics and Loewner Theory in 2018.
The communication with experts in this workshop led me to this study.
In addition, the author gratefully acknowledges helpful comments of Dr.\ Roland Friedrich on the first draft of this paper.
Finally, the author is greatly indebted to the anonymous referee for his careful reading and useful suggestions;
in particular, the use of Schwarz formula in the proof of Lemma~\ref{lem:boundedness} and the idea of \cite{AAS83}, both of which were taught to the author by the referee, enabled us to improve Theorems~\ref{thm:result_KL_EF} and \ref{thm:result_unbdd} a lot.
This research was supported by JSPS KAKENHI Grant Number JP19J13031.

\appendix

\section{Komatu--Loewner equation for slits}
\label{sec:KL_slit}

In this appendix,
we derive the Komatu--Loewner equation for slits~\eqref{eq:KLforSlit_vec}
from the equation~\eqref{eq:EF_KLeq},
following Bauer and Friedrich~\cite[Section~4.1]{BF08}
and Chen and Fukushima~\cite[Section~2]{CF18}.

As a preliminary, we prove a lemma,
which counts the order of zeros of conformal mappings
extended as in Section~\ref{sec:continuation}:

\begin{lemma} \label{lem:order_preim}
Suppose that $D_1 = E_1 \setminus \bigcup_{j=1}^N C_{1, j}$ and
$D_2 = E_2 \setminus \bigcup_{j=1}^N C_{2, j}$ are
parallel slit domains with $N$ slits and
that $f \colon D_1 \to D_2$ is a conformal mapping
which associates the slits $C_{1, j}$ with $C_{2, j}$, $j = 1, \ldots, N$, respectively.
Let $p_2 \in D_2^{\natural}$.
For the preimage $p_1 := (f^{\natural})^{-1}(p_2)$,
let
$\psi \colon D_1^{\natural} \supset U_{p_1} \to V_{p_1} \subset \mathbb{C}$
be a local coordinate around $p_1$.
Then the function
$h := f \circ \psi^{-1} \colon V_{p_1} \to \operatorname{pr}(D_2^{\natural})$
satisfies the following:
\begin{enumerate}
\item \label{lem:order_notEnd}
If $p_2 \notin \bigcup_{j=1}^N \{z^{\ell}_{2, j}, z^r_{2, j}\}$,
then 
$h - \operatorname{pr}(p_2)$
has a zero of the first order at $\psi(p_1)$.

\item \label{lem:order_End}
If $p_2 \in \bigcup_{j=1}^N \{z^{\ell}_{2, j}, z^r_{2, j}\}$,
then 
$h - \operatorname{pr}(p_2)$
has a zero of the second order at $\psi(p_1)$.
\end{enumerate}
\end{lemma}

\begin{proof}
We assume that 
$h - \operatorname{pr}(p_2)$
has a zero of order $m \geq 1$ at $\psi(p_1)$.
By Theorem~11 in Section~3.3, Chapter~3 of Ahlfors~\cite{Ah79},
there exist a neighborhood
$W_{\operatorname{pr}(p_2)} \subset \operatorname{pr}(D_2^{\natural})$
of $\operatorname{pr}(p_2)$
and a neighborhood
$\tilde{V}_{p_1} \subset V_{p_1}$
of $\psi(p_1)$ such that
$h(z) - w = 0$ has exactly $m$ distinct roots in $\tilde{V}_{p_1}$
for any $w \in W_{\operatorname{pr}(p_2)}\setminus \{\operatorname{pr}(p_2)\}$.

\noindent
\eqref{lem:order_notEnd}
Suppose $p_2 \notin \bigcup_{j=1}^N \{z^{\ell}_{2, j}, z^r_{2, j}\}$.
Then $h$ is univalent near $\psi(p_1)$ by definition.
Hence $m = 1$.

\noindent
\eqref{lem:order_End}
Suppose $p_2 \in \{z^{\ell}_{2, j}, z^r_{2, j}\}$ for some $j = 1, \ldots, N$.
Let $w \in C^{\circ}_{2, j} \cap W_{\operatorname{pr}(p_2)}$.
Then the equation $h(z) - w = 0$ has exactly two roots
$\tilde{z}^{+}$ and $\tilde{z}^{-}$
that satisfy
\[
f^{\natural}(\psi^{-1}(\tilde{z}^{+})) = w \in C^{+}_{2, j}
\quad \text{and} \quad
f^{\natural}(\psi^{-1}(\tilde{z}^{-})) = (w, j) \in C^{-}_{2, j}.
\]
Hence $m = 2$.
\end{proof}

Let $(\phi_{t, s})_{(s, t) \in I^2_{\leq}}$ be an evolution family over $(D_t)_{t \in I}$.
As in Section~\ref{sec:proof2},
the vectors
$\bm{s}(t) \in \mathbf{Slit}$, $t \in I$,
with $D(\bm{s}(t)) = D_t$
are determined uniquely,
provided that the order of the initial slits
$C_j(\bm{s}(0))$, $j=1, \ldots, N$,
is given.
The left and right endpoints of $C_j(t) := C_j(\bm{s}(t))$ are denoted by
$z^{\ell}_j(t) = x^{\ell}_j(t) + i y_j(t)$
and by
$z^r_j(t) = x^r_j(t) + i y_j(t)$,
respectively.
These endpoints are continuous in $t$ by Lemma~\ref{lem:EFineq}.
We put
$p^{\ell}_j(t) := (\phi_{t, 0}^{\natural})^{-1}(z^{\ell}_j(t))$
and
$p^r_j(t) := (\phi_{t, 0}^{\natural})^{-1}(z^r_j(t))$,
both of which are points on $C_j^{\natural}(0) \subset D_0^{\natural}$.

\begin{lemma} \label{lem:ac_preim}
Let $t_0 \in [0, T)$ and
$\psi \colon U_{p^{\ell}_j(t_0)} \to V_{p^{\ell}_j(t_0)}$
be a local coordinate of $p^{\ell}_j(t_0)$.
Then there exist $\delta, L > 0$ with the following properties:
$p^{\ell}_j(t) \in U_{p^{\ell}_j(t_0)}$
for every $t \in J := [(t_0 - \delta)^{+}, t_0 + \delta]$, and
$\tilde{z}^{\ell}_j(t) := \psi(p^{\ell}_j(t))$ satisfies
\begin{equation} \label{eq:ac_preim}
\lvert \tilde{z}^{\ell}_j(t) - \tilde{z}^{\ell}_j(s) \rvert
\leq L (\lambda(t) - \lambda(s)),
\quad (s, t) \in J^2_{\leq}.
\end{equation}
In addition,
the same assertion with the superscript $\ell$ replaced by $r$ holds.
\end{lemma}

\begin{proof}
We define
$h_t := \phi_{t, 0} \circ \psi^{-1}
\colon V_{p^{\ell}_j(t_0)} \to \operatorname{pr}(D_t^{\natural})$.
By Lemma~\ref{lem:order_preim}~\eqref{lem:order_End},
we have
\[
h_{t_0}^{\prime}(\tilde{z}_j(t_0)) = 0
\quad \text{and} \quad
h_{t_0}^{\prime \prime}(\tilde{z}_j(t_0)) \neq 0.
\]
In addition,
$(h_t)_{t \in I}$ satisfies $\mathrm{(Lip)}_{\lambda}$ on $V_{p^{\ell}_j(t_0)}$.
By Proposition~\ref{prop:implicit_thm},
there exist some neighborhood $J$ of $t_0$,
neighborhood
$\tilde{V} \subset V_{p^{\ell}_j(t_0)}$ of $\psi(p^{\ell}_j(t_0))$,
function $\hat{z} \colon J \to \tilde{V}$, and
constant $L > 0$ such that
\begin{equation} \label{eq:preim_implicit}
h_t^{\prime}(\hat{z}(t)) = 0 \quad \text{and} \quad
h_t^{\prime \prime}(\hat{z}(t)) \neq 0
\quad \text{for}\ t \in J
\end{equation}
and
\[
\lvert \hat{z}(t) - \hat{z}(s) \rvert \leq L (\lambda(t) - \lambda(s))
\quad \text{for}\ (s, t) \in J^2_{\leq}
\]
are satisfied.
\eqref{eq:preim_implicit} combined with Lemma~\ref{lem:order_preim}
implies that $h_t(\hat{z}(t)) \in \bigcup_{k=1}^N \{z^{\ell}_k(t), z^r_k(t)\}$.
By the continuity with respect to $t$,
we see that $h_t(\hat{z}(t))$ must coincide with $z^{\ell}_j(t)$.
In other words, $\hat{z}(t) = \tilde{z}^{\ell}_j(t)$.
The proof is now complete.
(Replacing the superscript $\ell$ by $r$ is trivial.)
\end{proof}

\begin{theorem} \label{thm:KLforSlit}
For each $j = 1, \ldots, N$,
the endpoints $z^{\ell}_j(t)$ and $z^r_j(t)$ of $C_j(t)$ enjoy
the Komatu--Loewner equation for the slits
\begin{align}
\tilde{\partial}^{\lambda}_t z^{\ell}_j(t)
&= \pi \int_{\mathbb{R}}
\Psi_{\bm{s}(t)}(z^{\ell}_j(t), \xi)
\, \nu_t(d\xi),
\label{eq:KLforLeft} \\
\tilde{\partial}^{\lambda}_t z^r_j(t)
&= \pi \int_{\mathbb{R}}
\Psi_{\bm{s}(t)}(z^r_j(t), \xi)
\, \nu_t(d\xi)
\label{eq:KLforRight}
\end{align}
for $m_{\lambda}$-a.e.\ $t \in I$.
\end{theorem}

\begin{proof}
We prove only \eqref{eq:KLforLeft}.
\eqref{eq:KLforRight} is then obtained
just by replacing the superscript $\ell$ with $r$
in the proof of \eqref{eq:KLforLeft}.
Let $N_0 \subset [0, T)$ be the exceptional set
defined by \eqref{eq:exception} with $t_0 = 0$.

We choose $t_0 \in [0, T)$ freely
and apply Lemma~\ref{lem:ac_preim} to this $t_0$.
Let
$J := [(t_0 - \delta)^{+}, t_0 + \delta]$
with $\delta$ as in Lemma~\ref{lem:ac_preim}.
By \eqref{eq:ac_preim},
there is a Lebesgue null set $\tilde{N} \subset J$
such that $\tilde{\partial}^{\lambda}_t \tilde{z}_j(t)$
exists for every $t \in J \setminus \tilde{N}$.
For $t \in J \setminus (N_0 \cup \tilde{N})$, we have
\begin{align}
\tilde{\partial}^{\lambda}_t z^{\ell}_j(t)
&= \tilde{\partial}^{\lambda}_t \left( \phi_{t, 0}(p^{\ell}_j(t)) \right)
\notag \\
&= \lim_{h \to +0} \frac{
\phi_{t+h, 0}(p^{\ell}_j(t+h)) - \phi_{t-h, 0}(p^{\ell}_j(t-h))
}
{
\lambda(t+h) - \lambda(t-h)
} \notag \\
&= \lim_{h \to +0} \frac{
\phi_{t + h, 0}(p^{\ell}_j(t + h))
- \phi_{t - h, 0}(p^{\ell}_j(t + h))
}{
\lambda(t + h) - \lambda(t - h)
} \label{eq:ChainRule} \\
& \phantom{=} + \lim_{h \to +0} \frac{
(\phi_{t - h, 0} \circ \psi^{-1})(\tilde{z}^{\ell}_j(t + h))
- (\phi_{t - h, 0} \circ \psi^{-1})(\tilde{z}^{\ell}_j(t - h))
}{
\lambda(t + h) - \lambda(t - h)
}. \notag
\end{align}
We note that
$p^{\ell}_j(\mathord{\cdot}) \colon J \to C_j^{\natural}(0)$
is continuous.
From the locally uniform convergence
in Lemma~\ref{prop:ae_diff}~\eqref{prop:AC_aeDiff},
we can see that
\begin{align*}
&\frac{
\phi_{t + h, 0}(p^{\ell}_j(t + h))
- \phi_{t - h, 0}(p^{\ell}_j(t + h))
}{
\lambda(t + h) - \lambda(t - h)
}
- (\tilde{\partial}^{\lambda}_t \phi_{t, 0})(p^{\ell}_j(t)) \\
&= \left( \frac{
\phi_{t + h, 0}(p^{\ell}_j(t + h))
- \phi_{t - h, 0}(p^{\ell}_j(t + h))
}{
\lambda(t + h) - \lambda(t - h)
}
- (\tilde{\partial}^{\lambda}_t \phi_{t, 0})(p^{\ell}_j(t+h))
\right) \\
&\phantom{=} + \left(
(\tilde{\partial}^{\lambda}_t \phi_{t, 0})(p^{\ell}_j(t+h))
- (\tilde{\partial}^{\lambda}_t \phi_{t, 0})(p^{\ell}_j(t))
\right) \\
&\to 0 \quad \text{as}\ h \to +0.
\end{align*}
Hence, the first limit in the rightmost side of \eqref{eq:ChainRule} is equal to
$(\tilde{\partial}^{\lambda}_t \phi_{t, 0})(p^{\ell}_j(t))$.
Also, we see that the second limit is equal to
$(\phi_{t, 0} \circ \psi^{-1})^{\prime}(\tilde{z}^{\ell}_j(t))
\cdot \tilde{\partial}^{\lambda}_{t} \tilde{z}^{\ell}_j(t)$.
However, since
$\phi_{t, 0} \circ \psi^{-1} - z^\ell_j(t)$
has a zero
of the second order at $\tilde{z}^{\ell}_j(t)$
by Lemma~\ref{lem:order_preim},
$(\phi_{t, 0} \circ \psi^{-1})^{\prime}(\tilde{z}^{\ell}_j(t)) = 0$.
Thus, by \eqref{eq:ChainRule} and \eqref{eq:EF_KLeq} we have
\begin{align*}
\tilde{\partial}^{\lambda}_t z^{\ell}_j(t)
&= (\tilde{\partial}^{\lambda}_{t} \phi_{t, 0})(p^{\ell}_j(t))
= \pi \int_{\mathbb{R}}
\Psi_{D_{t}}(\phi_{t, 0}(p^{\ell}_j(t)), \xi)
\, \nu_t(d\xi) \\
&= \pi \int_{\mathbb{R}}
\Psi_{D_{t}}(z^{\ell}_j(t), \xi)
\, \nu_t(d\xi).
\qedhere
\end{align*}
\end{proof}

We can rewrite \eqref{eq:KLforLeft} and \eqref{eq:KLforRight}
in the vector form
\begin{equation} \label{eq:KL_slit}
\tilde{\partial}^{\lambda}_t \bm{s}(t)
= \bm{b}(\nu_t, \bm{s}(t)).
\end{equation}
If $\lambda(t) = 2t$,
then \eqref{eq:KL_slit} coincides with \eqref{eq:KLforSlit_vec}
in Section~\ref{sec:proof3}.

\begin{remark}[SKLE and moduli diffusion]
\label{rem:moduli}

Since the slits $C_j(t)$ determine
the conformal equivalence class of
$D_t = \mathbb{H} \setminus \bigcup_{j=1}^N C_j(t)$,
Bauer and Friedrich~\cite{BF04, BF06, BF08} regarded
the system~\eqref{eq:KLforLeft} and \eqref{eq:KLforRight}
with $\nu_t = \delta_{\xi(t)}$
as a differential equation on the ``moduli space''
of $(N+1)$-connected planar domains
with one marked point $\xi(t)$ on boundary.
In the context of SKLE (see Section~\ref{sec:SLE}),
one combines these equations 
with the stochastic differential equation
\begin{equation} \label{eq:SKLE}
d\xi(t) = \alpha(\xi(t), D_t) \, dB_t + b(\xi(t), D_t) \, dt.
\end{equation}
The system of equations
\eqref{eq:KLforLeft}, \eqref{eq:KLforRight}, and \eqref{eq:SKLE}
(with $\nu_t = \delta_{\xi(t)}$)
thus determines the ``moduli diffusion''
$(\xi(t), z^{\ell}_j(t), z^r_j(t))$.
In fact,
Friedrich and Kalkkinen~\cite{FK04}
and
Kontsevich~\cite{Ko03}
studied conformally invariant probability measures
on the space of paths on Riemann surfaces,
which extends SLE,
by means of \emph{conformal field theory} and differential geometry.
Compared with their algebraic and geometric way,
the moduli diffusion $(\xi(t), z^{\ell}_j(t), z^r_j(t))$ here expresses
the random motion of moduli in an analytic, coordinate-based manner.
\end{remark}

\begin{remark}[Komatu--Loewner equation on annuli]
\label{rem:annuli}

Contreras, Diaz-Madrigal and Gumenyuk~\cite{CDMG11, CDMG13}
constructed Loewner theory on annuli.
In their theory,
the moduli, i.e., the ratios $r(t)$ of the outer and inner radii
of the underlying annuli
$\mathbb{A}_{r(t)} = \{\, z \mathrel{;} r(t) < \lvert z \rvert < 1 \,\}$
form a monotone function of $t$,
which is used as a new time-parameter.
Since $r(t)$ itself play the role of time,
Loewner theory on annuli does not involve any evolution equation for moduli.
This is a reason why
we have said that the case $N=1$ is special
in Section~\ref{sec:intro}.
See also Komatu~\cite{Ko43}, Zhan~\cite{Zh04},
and Fukushima and Kaneko~\cite{FK14}.
\end{remark}

\section{On the assumptions~(H.\ref{ass:hydro}) and (H.\ref{ass:res})}
\label{sec:at_infinity}

In this appendix,
we confirm that
the assumptions~(H.\ref{ass:hydro}) and (H.\ref{ass:res})
are preserved by taking the inverse and composite of functions.

\begin{proposition} \label{prop:inverse_norm}
Let $f \colon D \to \mathbb{C}$ be a univalent function
with {\rm (H.\ref{ass:hydro})}.
Then so is the inverse $f^{-1} \colon f(D) \to D$.
If, moreover, $f$ enjoys {\rm (H.\ref{ass:res})} with angular residue $c$,
then so does $f^{-1}$ with angular residue $-c$.
\end{proposition}

\begin{proof}
We can take $\eta, L > 0$ so that
\[
\lvert f(z) - z \rvert < 1,
\quad z \in \mathbb{H}_{\eta} \setminus \bar{B}(0, L)
\subset \mathbb{H}_{\eta+L}.
\]
Clearly $\mathbb{H}_{\eta+L+1} \subset f(\mathbb{H}_{\eta+L})$.
Similarly, let $\varepsilon > 0$
and take $L' \geq 0$ so that
$\lvert f(z) - z \rvert < \varepsilon$
holds for $z \in \mathbb{H}_{\eta} \setminus \bar{B}(0, L')$.
We have
\[
\lvert f^{-1}(w) - w \rvert = \lvert f^{-1}(w) - f(f^{-1}(w)) \rvert < \varepsilon,
\quad w \in \mathbb{H}_{\eta+L+1} \setminus \bar{B}(0, L'+1),
\]
because $f^{-1}(w) \in \mathbb{H}_{\eta+L} \setminus \bar{B}(0, L')$
for such $w$.
Thus, $f^{-1}$ enjoys (H.\ref{ass:hydro}).

Now, suppose that $f$ has the finite angular residue $c$ at infinity.
Let $\theta \in (0, \pi/2)$.
For $w \in \triangle_{\theta} \setminus \bar{B}(0, (\eta+L+2)/\sin \theta)$,
we have
$\Im w \geq 2$, $\lvert w \rvert \geq 1$, and
$\lvert f^{-1}(w) - w \rvert < 1$
because $w \in \mathbb{H}_{\eta+L+2}$.
These inequalities yield
\[
\frac{\lvert f^{-1}(w) \rvert}{\Im f^{-1}(w)}
< \frac{\lvert w \rvert + 1}{\Im w - 1}
< \frac{2\lvert w \rvert}{\Im w / 2}
< \frac{4}{\sin \theta}.
\]
Thus,
$f^{-1}(\triangle_{\theta} \setminus \bar{B}(0, (\eta+L+2)/\sin \theta))
\subset \triangle_{\theta'}$ holds
with $\theta'$ given by $4 \sin \theta' = \sin \theta$.
From the identity
\[
w(f^{-1}(w) - w) = - f^{-1}(w) (f(f^{-1}(w)) - f^{-1}(w)) - (f^{-1}(w) - w)^2,
\]
we get
\[
\lim_{\substack{w \to \infty \\ w \in \triangle_{\theta}}} w (f^{-1}(w) - w)
= - \lim_{\substack{z \to \infty \\ z \in \triangle_{\theta'}}} z (f(z) -  z)
= - (-c).
\]
Hence the angular residue of $f^{-1}$ is $-c$.
\end{proof}

The proof of the next proposition is quite similar, and we omit it.
(The same idea can be seen
in the proof of Theorem~1 of Goryainov and Ba~\cite{GB92}.)

\begin{proposition} \label{prop:composite_norm}
Let $f \colon D \to \mathbb{C}$ and $g \colon D' \to \mathbb{C}$ be
univalent functions with {\rm (H.\ref{ass:hydro})}.
Then so is the composite
$\left. g \right\rvert_{D' \cap f(D)} \circ \left. f \right\rvert_{f^{-1}(D')}$.
If, moreover, they enjoy {\rm (H.\ref{ass:res})}
with angular residues $c_f$ and $c_g$, respectively,
then so does $g \circ f$ with angular residue $c_f + c_g$.
\end{proposition}

\section{On the assumption~$\mathrm{(Lip)}_{\cdot}$}
\label{sec:parameter}

The contents of this appendix are analogous to
the classical arguments on a.e.\ differentiability
in the proof of Pommerenke~\cite[Theorem~6.2]{Po75} and
Goryainov and Ba~\cite[Theorem~3]{GB92}.
Since more general results are required in this paper,
we provide a self-contained proof
of each statement for the sake of completeness.

\subsection{Absolute continuity and almost everywhere differentiability}
\label{sec:ac}

In Section~\ref{sec:assumptions},
we have introduced the property~$\mathrm{(Lip)}_F$
for a non-decreasing continuous function $F(t)$.
Since only does the corresponding measure $m_F$ play a role
in the subsequent argument,
we change the notation slightly.
Let $I$ be an interval equipped with a non-atomic Radon measure $\mu$ and
$f_t$ be a holomorphic function on a Riemann surface $X$ for each $t \in I$.
We consider the following properties:
\begin{itemize}
\item[$\mathrm{(AC)}_{\mu}$]
For any compact subset $K$ of $X$,
there exists a measure $\nu_K$ on $I$
which is absolutely continuous with respect to $\mu$ and satisfies
\[
\sup_{p \in K} \lvert f_t(p) - f_s(p) \rvert \leq \nu_K((s, t])
\quad \text{for}\ (s, t) \in I^2_{\leq}.
\]

\item[$\mathrm{(Lip)}_{\mu}$]
For any compact subset $K$ of $X$,
there exists a constant $L_K$ such that
\[
\sup_{p \in K} \lvert f_t(p) - f_s(p) \rvert \leq L_K \mu((s, t])
\quad \text{for}\ (s, t) \in I^2_{\leq}.
\]
\end{itemize}
Obviously $\mathrm{(Lip)}_{\mu}$ implies $\mathrm{(AC)}_{\mu}$,
and if $\mathrm{(AC)}_{\mu}$ holds,
then $t \mapsto f_t$ is continuous
in $\operatorname{Hol}(X; \mathbb{C})$,
the space of holomorphic functions on $X$
equipped with the topology of locally uniform convergence.
$\mathrm{(AC)}_{\mu}$ also implies that, for each $p \in X$,
the set function $\kappa_p((s, t]) := f_t(p) - f_s(p)$
on the set of left half-open intervals
extends to a complex measure on every compact subinterval of $I$
which is absolutely continuous with respect to $\mu$.
By the generalized Lebesgue's differentiation theorem~\cite[Theorem~5.8.8]{Bo07},
the limit
\[
\tilde{\partial}^{\mu}_t f_t(p)
:= \lim_{\delta \downarrow 0} \frac{f_{t+\delta}(p) - f_{t-\delta}(p)}%
{\mu((t-\delta, t+\delta))}
\]
exists for a.e.\ $t \in I$ and
is a version of the Radon--Nikodym derivative $d\kappa_p/d\mu$.
If $\mu$ is associated with a continuous non-decreasing function $F$ on $I$
by the relation $\mu((s ,t]) = F(t) - F(s)$ (i.e., $\mu = m_F$),
then we designate the properties $\mathrm{(AC)}_{\mu}$ and $\mathrm{(Lip)}_{\mu}$
as $\mathrm{(AC)}_F$ and as $\mathrm{(Lip)}_F$, respectively, and
the derivative $\tilde{\partial}^{\mu}_t f_t(p)$ as $\tilde{\partial}^F_t f_t(p)$.
This notation is consistent with that in Section~\ref{sec:assumptions}.

In general, the $\mu$-null set
on which $\tilde{\partial}^{\mu}_t f_t(p)$ does not exist
depends on $p$.
However, $\mathrm{(AC)}_{\mu}$ enables us
to choose this exceptional set $N$ independently of $p$,
as shown in the following proposition:

\begin{proposition} \label{prop:ae_diff}
Suppose that a family $(f_t)_{t \in I}$ of holomorphic functions
on a Riemann surface $X$ satisfies $\mathrm{(AC)}_{\mu}$.
\begin{enumerate}
\item \label{prop:AC_aeDiff}
There exists a $\mu$-null set $N \subset I$ such that,
for each $t \in I \setminus N$,
the convergence
\[
\frac{f_{t+\delta}(p) - f_{t+\delta}(p)}{\mu((t-\delta, t+\delta))}
\to \tilde{\partial}^{\mu}_t f_t(p)
\quad \text{as}\ \delta \to +0
\]
occurs locally uniformly in $p \in X$,
and hence $\tilde{\partial}^{\mu}_t f_t$ is a holomorphic function on $X$.

\item \label{prop:Lip_aeDiff}
If $(f_t)_{t \in I}$ further satisfies $\mathrm{(Lip)}_{\mu}$,
then we can choose the null-set $N$ in \eqref{prop:AC_aeDiff}
by fixing a countable set $A \subset X$ with an accumulation point in $X$
and setting
\[
N:=
\bigcup_{p\in A} \{\,
	t \in I \mathrel{;} \text{$\tilde{\partial}^{\mu}_t f_t(p)$ does not exists}
	\,\}.
\]
\end{enumerate}
\end{proposition}

\begin{proof}
\eqref{prop:AC_aeDiff}
We take an exhaustion sequence $(X_n)_{n \in \mathbb{N}}$ of $X$;
that is, all $X_n$'s are relatively compact subdomains of $X$
with $\bigcup_{n=1}^{\infty} X_n = X$.
It suffices to show that, for each $n \in \mathbb{N}$, there exists a $\mu$-null set $N_n\subset I$ such that $\tilde{\partial}^{\mu}_t f_t(p)$ exists and is holomorphic on $X_n$ for each $t \in I \setminus N_n$.
Indeed, we can conclude from this auxiliary assertion that $\tilde{\partial}^{\mu}_t f_t(p)$ exists and is holomorphic on $X$ for each $t \in I \setminus N$ with $N := \bigcup_{n \in \mathbb{N}} N_n$.
Therefore, we fix $n \in \mathbb{N}$ and prove the proposition on $X_n$.

$\overline{X_n}$ is a compact subset of $X$, and hence there exists a measure $\nu_n \ll \mu$ on $I$ such that $\lvert f_t(p) - f_s(p) \rvert \leq \nu_n((s, t])$ for any $p \in X_n$ and $(s, t) \in I^2_{\leq}$.
Let $A \subset X_n$ be a countable set having an accumulation point in $X_n$.
Since $\tilde{\partial}^{\mu}_t f_t(p)$ exists at $\mu$-a.e.\ $t$ for each fixed $p \in A$, there exists a null set $N_n \subset I$ such that
\[
\tilde{\partial}^{\mu}_t f_t(p) \ (p \in A) \quad \text{and} \quad D_{\mu} \nu_n(t) :=\lim_{\delta \downarrow 0} \frac{\nu_n((t-\delta, t+\delta))}{\mu((t-\delta, t+\delta))}
\]
all exist at every $t \in I \setminus N_n$.
We fix such $t$.
By $\mathrm{(AC)}_{\mu}$ we have
\begin{equation} \label{eq:AC_DiffBound}
\frac{\lvert f_{t-\delta}(p) - f_{t+\delta}(p) \rvert}{\mu((t-\delta, t+\delta))} \leq \frac{\nu_n((t-\delta, t+\delta))}{\mu((t-\delta, t+\delta))}
\qquad \text{for all}\ p \in X_n.
\end{equation}
The left-hand side in this inequality
is bounded in $p \in X_n$ and $\delta > 0$
because the right-hand side converges to $D_{\mu} \nu_n(t)$
as $\delta \downarrow 0$.
Moreover, $(f_{t-\delta}(p) - f_{t+\delta}(p))/\mu((t-\delta, t+\delta))$
converges to $\tilde{\partial}^{\mu}_t f_t(p)$ as $\delta \downarrow 0$ for each $p \in A$.
Thus, this divided difference converges as $\delta \downarrow 0$
locally uniformly on $X_n$ by Vitali's convergence theorem,
which implies that $\tilde{\partial}^{\mu}_t f_t(p)$ exists and is holomorphic on $X_n$.

\noindent
\eqref{prop:Lip_aeDiff}
Let $A$ and $N$ be as in the statement of \eqref{prop:Lip_aeDiff}.
Then the left-hand side of \eqref{eq:AC_DiffBound} is bounded by $L_K$ on every compact subset $K$.
Hence it is locally uniformly bounded on $X$.
Vitali's theorem thus implies that $\tilde{\partial}^{\mu}_t f_t(p)$ exists for every $t \in I \setminus N$ and $p \in X$.
Note that we do not need to take an exhaustion sequence $(X_n)_n$ in this case.
\end{proof}

\begin{remark}[Another ``absolute continuity'']
In the case where $\mu$ coincides with
the Lebesgue measure $\mathbf{Leb}$ on $I$,
Bracci, Contreras and Diaz-Madrigal~\cite{BCDM12}
and Contreras, Diaz-Madrigal and Gumenyuk~\cite{CDMG10}
considered a condition
broader than $\mathrm{(AC)}_{\mathbf{Leb}}$ and $\mathrm{(Lip)}_{\mathbf{Leb}}$.
Roughly speaking,
they say that a family $(f_t)_{t \in I}$ is of order $d \in [1, \infty]$
if, for every compact subset $K$,
there exists a function $k_K \in L^d(I)$ such that
\[
\sup_{p \in K} \lvert f_t(p) - f_s(p) \rvert \leq \int_s^t k_K(u) \, du,
\quad (s, t) \in I^2_{\leq}.
\]
According to this definition,
$(f_t)_{t \in I}$ satisfies $\mathrm{(AC)}_{\mathbf{Leb}}$
if it is of order $d$ for some $d \in [1, \infty]$,
and in particular,
$\mathrm{(Lip)}_{\mathbf{Leb}}$ holds if and only if $d = \infty$.
From this viewpoint,
Lemma~\ref{lem:EFineq} shows that,
given an evolution family $(\phi_{t, s})$,
we can always assume $d = \infty$
if we replace $\mathbf{Leb}$ by the measure $m_{\lambda}$
associated with $(\phi_{t, s})$.
This fact makes our argument easier,
for example, in Lemma~\ref{lem:NullConsis} and
Proposition~\ref{prop:ae_diff}~\eqref{prop:Lip_aeDiff}.
\end{remark}

\subsection{Descent to spatial derivatives and inverse functions}
\label{sec:deriv_inverse}

In this and next subsections,
we discuss only the case in which $X$ is a planar domain $D \subset \mathbb{C}$.

\begin{proposition} \label{prop:derivative}
Let $(f_t)_{t\in I}$ be a family of holomorphic functions on a planar domain $D$.
\begin{enumerate}
\item \label{prop:derivative_conti}
If $(f_t)_{t \in I}$ is continuous in $\operatorname{Hol}(D; \mathbb{C})$,
then so is the family $(f^{(n)}_t)_{t \in I}$
of the $n$-th order $z$-derivatives
for any $n \in \mathbb{N}$.

\item \label{prop:derivative_Lip}
If $(f_t)_{t \in I}$ satisfies $\mathrm{(Lip)}_{\mu}$,
then so is $(f^{(n)}_t)_{t \in I}$ for any $n$.
\end{enumerate}
\end{proposition}

\begin{proof}
\eqref{prop:derivative_conti} is just a standard fact in complex analysis.
We prove \eqref{prop:derivative_Lip} here.

Assume that $(f_t)_{t \in I}$ satisfies $\mathrm{(Lip)}_{\mu}$.
Without loss of generality, we may and do assume that $D = \mathbb{D}$.
Let $r$ and $\delta$ be two arbitrary positive numbers such that $r + \delta < 1$.
We take the constant $L_K$ in $\mathrm{(Lip)}_{\mu}$ with $K := \partial B(0, r + \delta)$.
Using Cauchy's integral formula, we have
\begin{align*}
\sup_{\lvert z \rvert \leq r} \lvert f^{(n)}_t(z) - f^{(n)}_s(z) \rvert &\leq \frac{1}{2 \pi} \sup_{\lvert z \rvert \leq r} \int_{\lvert \zeta \rvert = r + \delta} \frac{\lvert f_t(\zeta) - f_s(\zeta) \rvert}{\lvert \zeta - z \rvert^{n+1}} \, \lvert d\zeta \rvert \\
&\leq \frac{r + \delta}{\delta^{n+1}} L_K \mu((s, t])
\end{align*}
for any $(s, t) \in I^2_{\leq}$,
which yields the property~$\mathrm{(Lip)}_{\mu}$ of $(f^{(n)}_t)_{t \in I}$.
\end{proof}

If $f_t$'s are univalent and satisfy $\mathrm{(Lip)}_{\mu}$,
then their inverse functions satisfy the same property locally in time and space,
which is a conclusion from the following Lagrange inversion formula:

\begin{lemma} \label{lem:Lagrange}
Let $f$ be a univalent function on a planar domain $D$, $w$ be a point of $f(D)$ and $C$ be a simple closed curve in $D$ surrounding $f^{-1}(w)$ such that $\operatorname{ins} C \subset D$.
Then the equality
\[
f^{-1}(w) = \frac{1}{2 \pi i} \int_C \frac{\zeta f'(\zeta)}{f(\zeta) - w} \, d\zeta
\]
holds.
\end{lemma}

\begin{proof}
The function $z f(z)/(f(z) - w)$ of $z$ has a pole of the first order
at $z = f^{-1}(w)$, and its residue is
\[
\lim_{z \to f^{-1}(w)} (z - f^{-1}(w)) \frac{z f'(z)}{f(z) - w} = f^{-1}(w).
\]
Hence the conclusion follows from the residue theorem.
\end{proof}

\begin{proposition} \label{prop:inverse}
Suppose that
a family $(f_t)_{t\in I}$ of univalent functions
is continuous in $\operatorname{Hol}(D; \mathbb{C})$.
Let $t_0 \in I$ and $U$ be a bounded domain with
$\overline{U} \subset f_{t_0}(D)$.
Then there exists a neighborhood $J$ of $t_0$ in $I$ such that
$\overline{U} \subset \bigcap_{t \in J} f_t(D)$.
For any such pair $(J, U)$,
the family $(f^{-1}_t)_{t \in J}$ of the inverse functions
is continuous in $\operatorname{Hol}(U; \mathbb{C})$.
If $(f_t)_{t\in I}$ further satisfies $\mathrm{(Lip)}_{\mu}$ on $D$,
then so does $(f^{-1}_t)_{t \in J}$ on $U$.
\end{proposition}

\begin{proof}
Owing to the compactness of $\overline{U}$,
it suffices to prove that
for any fixed $w_0 \in f_{t_0}(D)$,
the proposition holds
with $U$ replaced by a sufficiently small disk $B(w_0, r_0)$.
We assume
$f_{t_0}^{-1}(w_0) = 0 \in D$
for the simplicity of notation.

We choose such a small $r_0$ that
$f_{t_0}^{-1}(\overline{B(w_0, r_0)})
\subset B(0, r)
\subset \overline{B(0, r)}
\subset D$
holds for some $r > 0$.
Set
$\epsilon
:= d^{\mathrm{Eucl}}(f_{t_0}(\partial B(0, r)), \partial B(w_0, r_0))
> 0$.
Since $(f_t)_{t\in I}$ is continuous
in the topology of locally uniform convergence,
there exists a closed neighborhood $J = [\alpha, \beta]$ of $t_0$ such that
$\lvert f_t(z) - f_{t_0}(z) \rvert < \epsilon/2$
holds for $z \in \overline{B(0, r)}$ and $t \in J$.
This inequality implies that
$\overline{B(w_0, r_0)}
\subset \bigcap_{t \in J} f_t(B(0, r))
\subset \bigcap_{t \in J} f_t(D)$.

Next, we show that
$(f^{-1}_t)_{t \in J}$ satisfies $\mathrm{(Lip)}_{\mu}$ on $U$,
assuming that $(f_t)_{t \in I}$ satisfies $\mathrm{(Lip)}_{\mu}$ on $D$.
Since the continuity of $(f^{-1}_t)_{t \in J}$
in $\operatorname{Hol}(U; \mathbb{C})$
is proved in a similar way,
we omit it.
By Proposition~\ref{prop:derivative}~\eqref{prop:derivative_Lip},
we can take two constants $L_0$ and $L_1$ such that
$\sup_{\lvert z \rvert \leq r} \lvert f^{(n)}_t(z) - f^{(n)}_s(z)\rvert
\leq L_n \mu ((s, t])$,
$n=0, 1$,
holds for any $(s, t) \in I^2_{\leq}$.
In particular, we have
\begin{align*}
M_n
&:= \max
\{\, \lvert f^{(n)}_t(z) \rvert \mathrel{;} \lvert z \rvert = r,\ t \in J \,\}
\\
&\leq \max_{\lvert z \rvert = r}
\lvert f^{(n)}_{t_0}(z) \rvert + L_n \mu((\alpha, \beta])
< \infty,
\quad n=0, 1.
\end{align*}
Now, using Lemma~\ref{lem:Lagrange} we have
\begin{align*}
&f^{-1}_t(w) - f^{-1}_s(w) \\
&= \frac{1}{2\pi i}
\int_{\lvert \zeta \rvert=r} \left(
\frac{\zeta f_t'(\zeta)}{f_t(\zeta) - w}
- \frac{\zeta f_s'(\zeta)}{f_s(\zeta) - w}
\right)
\, d\zeta
\\
&= \frac{1}{2\pi i}
\int_{\lvert \zeta \rvert=r}
\frac{\zeta\{f_t'(\zeta)(f_s(\zeta)-w)- f_s'(\zeta)(f_t(\zeta)-w)\}}{(f_t(\zeta)-w)(f_s(\zeta)-w)}
\, d\zeta
\end{align*}
for $w \in \overline{B(w_0, r_0)}$ and $(s, t) \in J^2_{\leq}$.
Hence
\begin{align*}
&\lvert f^{-1}_t(w) - f^{-1}_s(w) \rvert
\\
&\leq \frac{2r}{\pi \epsilon^2}
\int_{\lvert \zeta \rvert=r} \left(
\lvert f_t(\zeta) \rvert \lvert f_t'(\zeta) - f'_s(\zeta) \rvert
+ \lvert f_t(\zeta) - f_s(\zeta) \rvert \lvert f'_s(\zeta) \rvert
\right)
\, \lvert d\zeta \rvert
\\
&\leq 4 r^2 \frac{(M_0+\lvert w_0 \rvert+r_0) \vee M_1}{\epsilon^2}
(L_0 + L_1) \mu((s, t]),
\end{align*}
which yields the property~$\mathrm{(Lip)}_{\mu}$
of $(f^{-1}_t)_{t \in J}$ on $B(w_0, r_0)$.
\end{proof}

\subsection{Implicit function theorem}
\label{sec:implicit_thm}

\begin{proposition}
\label{prop:implicit_thm}

Let $(f_t)_{t\in I}$ be a family of holomorphic functions
that satisfies $\mathrm{(Lip)}_{\mu}$ on a planar domain $D$.
Suppose that a point $(t_0, z_0) \in I \times D$ enjoys the conditions
\[
f_{t_0}(z_0) = 0 \quad \text{and} \quad f'_{t_0}(z_0) \neq 0.
\]
Then there exist some neighborhood $J$ of $t_0$ in $I$,
neighborhood $U$ of $z_0$ in $D$
and function $\hat{z} \colon J \to U$ such that
\begin{itemize}
\item $z = \hat{z}(t)$ is a unique zero of the holomorphic function $f_t$ in $U$,
which is of the first order, for any $t\in J$;

\item $\hat{z}(t)$ is Lipschitz continuous with respect to $\mu$ in the sense that
\[
\lvert \hat{z}(t) - \hat{z}(s) \rvert
\leq \hat{L} \mu((s, t]),
\quad (s, t) \in J^2_{\leq},
\]
holds for some constant $\hat{L}$.
In particular, the complex measure $\hat{\kappa}$ induced from $\hat{z}(t)$
on every compact subinterval of $I$
is absolutely continuous with respect to $\mu$.
\end{itemize}
\end{proposition}

\begin{proof}
As $f'_{t_0}(z_0) \neq 0$,
there exists $r_0 > 0$ such that $f_{t_0}$ is univalent on $B(z_0, r_0)$.
We take $r_1 \in (0, r_0)$ and set
$m := \min_{\lvert z - z_0 \rvert = r_1} \lvert f_{t_0}(z) \rvert > 0$.
As $f_t \to f_{t_0}$ ($t \to t_0$) uniformly on $\bar{B}(z_0, r_1)$,
there exists $\delta>0$ such that
\begin{equation} \label{eq:appendix_C2}
\max_{z \in \bar{B}(z_0, r_1)} \lvert f_t(z) - f_{t_0}(z) \rvert < \frac{m}{4}
\qquad \text{if}\ \lvert t - t_0 \rvert < \delta.
\end{equation}
In this setting, $f_t \rvert_{B(z_0, r_1)}$ takes each value $w \in B(0, m/2)$ at most once, counting multiplicities, if $\lvert t - t_0 \rvert < \delta$.
Indeed, for $w \in B(0, m/2)$, $z \in \partial B(z_0, r_1)$ and $\lvert t - t_0 \rvert < \delta$, we have
\begin{align}
\lvert (f_t(z) - w)-(f_{t_0}(z) - w) \rvert
&= \lvert f_t(z) - f_{t_0}(z) \rvert < \frac{m}{4} \notag \\
&< \frac{m}{2} \leq \lvert f_{t_0}(z) \rvert -\frac{m}{2} \notag \\
&\leq \lvert f_{t_0}(z) - w \rvert. \label{eq:appendix_C3}
\end{align}
Since $f_{t_0} \rvert_{B(z_0, r_1)}$ is univalent,
it takes each value $w \in B(0,m/2)$ at most once,
counting multiplicities, 
and so does $f_t \rvert_{B(z_0, r_1)}$ by \eqref{eq:appendix_C3} and Rouch\'e's theorem.

We take $r \in (0, r_1)$ so that $\sup_{z \in B(z_0, r)} \lvert f_{t_0}(z) \rvert < m/4$.
Then \eqref{eq:appendix_C2} implies that $f_t(B(z_0, r)) \subset B(0, m/2)$ for $\lvert t - t_0 \rvert < \delta$.
Thus, by the preceding paragraph, $f_t$ is univalent%
\footnote{To prove the univalence of $f_t$, the restriction to a disk $B(z_0, r)$ smaller than $B(z_0, r_1)$ is necessary because the inequality \eqref{eq:appendix_C3} and Rouch\'e's theorem do not imply that $f_t$ takes a value $w \notin B(0,m/2)$ at most once.}
on $B(z_0, r)$ if $\lvert t - t_0 \rvert < \delta$.
It is now clear from Proposition~\ref{prop:inverse} that
a desired triplet $(J, U, \hat{z}(t))$ is given by
$J = (t_0 - \delta, t_0 + \delta) \cap I$, $U = B(z_0, r)$
and $\hat{z}(t) = f_t^{-1}(0)$.
\end{proof}

\begin{remark}[``Continuous differentiability'']
Let us refer to one of the following two conditions,
which one can prove to be equivalent to each other, as {\rm (CD)}:
\begin{itemize}
\item For each $t \in I$,
$(f_{t + h} - f_t) / h$ converges in $\operatorname{Hol}(X; \mathbb{C})$, and
the family of the limits $\dot{f}_t$, $t \in I$, is also continuous
in $\operatorname{Hol}(X; \mathbb{C})$.

\item For each $p \in X$,
the function $t \mapsto f_t(p)$ is $C^1$, and
the family of the $t$-derivatives $\dot{f}_t$, $t \in I$, is locally bounded on $X$.
\end{itemize}
Then Propositions~\ref{prop:derivative} and \ref{prop:inverse} are valid with $\mathrm{(Lip)}_{\mu}$ replaced by {\rm (CD)}.
Proposition~\ref{prop:implicit_thm} also holds with the following replacement:
$\mathrm{(Lip)}_{\mu}$ in the assumption is replaced by {\rm (CD)},
and the Lipschitz continuity of $\hat{z}(t)$ is replaced by the continuous differentiability of $\hat{z}(t)$ in $t$.
Although these facts are not employed in this paper,
one can see that such considerations make the argument
in Section~2 of Chen and Fukushima~\cite{CF18}
slightly simpler.
\end{remark}

\section{Vague and weak topologies}
\label{sec:top_on_meas}

In this appendix, we recall the definition of vague and weak topologies
on the set $\mathcal{M}(\mathbb{R})$ of finite Borel measures on $\mathbb{R}$
and present their properties needed in this paper.

We denote by $C_c(\mathbb{R})=C_c(\mathbb{R};\mathbb{R})$ the set of real-valued continuous functions on $\mathbb{R}$ with compact support.
Let $(m_\alpha)_{\alpha\in A}$ be a net 
in $\mathcal{M}(\mathbb{R})$.
We say that $(m_\alpha)_{\alpha\in A}$ \emph{converges vaguely} 
to $m\in \mathcal{M}(\mathbb{R})$ if
\begin{equation} \label{eq:def_vague}
\lim_\alpha \int_{\mathbb{R}}f\, dm_\alpha = \int_{\mathbb{R}}f\, dm
\end{equation}
holds for every $f\in C_c(\mathbb{R})$.
We write $m_\alpha\stackrel{v}{\to}m$ to indicate this convergence.
The vague convergence defines a topology 
on $\mathcal{M}(\mathbb{R})$,
which we call the \emph{vague topology}.

As for the total masses of vaguely convergent measures,
we have the following basic property:
suppose $m_\alpha\stackrel{v}{\to}m$, and
let $f_k\in C_c(\mathbb{R})$, $k\in\mathbb{N}$, be a sequence such that $0\le f_k\uparrow 1$ pointwise as $k\to\infty$.
Then
\begin{align}
m(\mathbb{R})&=\lim_{k\to\infty}\int_{\mathbb{R}}f_k\, dm
=\lim_{k\to\infty}\lim_\alpha \int_{\mathbb{R}}f_k\, dm_\alpha \notag \\
&\le \liminf_\alpha m_\alpha(\mathbb{R}). \label{eq:vague_Fatou}
\end{align}
By this inequality, the set $\mathcal{M}_{\le 1}(\mathbb{R})$ of Borel measures on $\mathbb{R}$ with total mass less than or equal to one is closed in $\mathcal{M}(\mathbb{R})$.

We make more detailed observations in view of functional analysis.
Let $C_\infty(\mathbb{R})$ be the set of continuous functions on $\mathbb{R}$ vanishing at infinity.
Since $C_\infty(\mathbb{R})$ is the completion of $C_c(\mathbb{R})$
with respect to the supremum norm $\lVert\mathord{\cdot}\rVert_\infty$,
it is easy to see
\footnote{It is valid to replace $C_c(\mathbb{R})$ with $C_\infty(\mathbb{R})$ only when $(m_\alpha)_\alpha$ is bounded;
for example, $(1+n^2) \,\delta_{\{n\}}\stackrel{v}{\to} 0$ follows from definition, but $\int_{\mathbb{R}}(1+x^2)^{-1}(1+n^2) \,\delta_{\{n\}}(dx)=1$ for all $n\in \mathbb{N}$.
In passing, we notice that our definition of vague convergence is thus slightly different from the one based on $C_\infty(\mathbb{R})$ in Folland's book cited here \cite[p.223]{Fo99}.}
that $m_\alpha\stackrel{v}{\to}m$ 
in $\mathcal{M}_{\le 1}(\mathbb{R})$
if and only if \eqref{eq:def_vague} holds for every $f\in C_\infty(\mathbb{R})$.
Moreover, any finite signed Borel measure $m$ on $\mathbb{R}$ corresponds to
a bounded linear functional $f\mapsto\int f\, dm$ on $C_\infty(\mathbb{R})$
in one-to-one manner
by the Riesz--Markov--Kakutani representation theorem
(see Folland~\cite[Theorem~7.17]{Fo99} for instance).
Thus, the vague topology 
of $\mathcal{M}_{\le 1}(\mathbb{R})$
is the weak${}^\ast$ topology of the dual $C_\infty(\mathbb{R})^\ast$ 
restricted to $\mathcal{M}_{\le 1}(\mathbb{R})$
through this identification.

The operator norm of the above-mentioned functional $f\mapsto \int f\, dm$ equals the total variation of $m$, that is,
$\sup\{\,\lvert\int f\, dm\rvert\mathrel{;} f\in C_\infty(\mathbb{R}),\ \lVert f\rVert_\infty=1\,\}=\lvert m\rvert(\mathbb{R})$. 
Hence
the Banach--Alaoglu theorem (see, e.g.,  Rudin~\cite[Theorems~3.15 and 3.16]{Ru91}) yields the following:
\begin{proposition} \label{prop:vague_cpt}
$\mathcal{M}_{\le 1}(\mathbb{R})$ is a compact metrizable space under the vague topology.
\end{proposition}

We now introduce another topology.
We denote by $C_b(\mathbb{R})$ the set of bounded continuous functions on $\mathbb{R}$.
Let $(m_\alpha)_{\alpha\in A}$ be a net 
in $\mathcal{M}(\mathbb{R})$.
We say that $(m_\alpha)_{\alpha\in A}$ \emph{converges weakly} 
to $m\in \mathcal{M}(\mathbb{R})$
if \eqref{eq:def_vague} holds for every $f\in C_b(\mathbb{R})$.
The weak convergence defines a topology 
on $\mathcal{M}(\mathbb{R})$,
which we call the \emph{weak topology}.
Actually, this is the weak${}^\ast$ topology of the dual $C_b(\mathbb{R})^\ast$ 
restricted to $\mathcal{M}(\mathbb{R})$.

As $C_b(\mathbb{R})^\ast$ is strictly larger than the set of signed measures,
$\mathcal{M}_{\le 1}(\mathbb{R})$ is not compact in the weak topology.
For proof of the next proposition, we refer the reader to Stroock~\cite[Theorems~9.1.5 and 9.1.11]{Str11} or Kallenberg~\cite[Lemma~4.5]{Ka17}.

\begin{proposition} \label{prop:weak_Pol}
$\mathcal{M}_{\le 1}(\mathbb{R})$ is a Polish space under the weak topology.
\end{proposition}

We notice that the set $\mathcal{P}(\mathbb{R})$ of Borel probability measures on $\mathbb{R}$ is not compact in either the vague or weak topology.
The sequence of point masses $\delta_n$ at $n\in\mathbb{N}$ is a typical example: $\delta_n\stackrel{v}{\to}0\notin\mathcal{P}(\mathbb{R})$ as $n\to\infty$.
Here we give a well-known condition which is equivalent to weak relative compactness,
although it is not utilized in this paper.

\begin{definition} \label{def:tight}
A subset $\mathcal{Q}\subset\mathcal{P}(\mathbb{R})$ is said to be \emph{tight}
if, for every $\varepsilon>0$, there exists a compact set $K\subset\mathbb{R}$ such that $\sup_{m\in\mathcal{Q}}m(\mathbb{R}\setminus K)<\varepsilon$.
\end{definition}

The Prokhorov theorem, which the reader can find in standard textbooks of probability theory,
asserts that $\mathcal{Q}\subset\mathcal{P}(\mathbb{R})$ is tight if and only if $\mathcal{Q}$ is relatively compact in the weak topology of $\mathcal{P}(\mathbb{R})$.

In the above, $\mathcal{M}_{\le 1}(\mathbb{R})$ is equipped with two topologies.
We denote by $\mathcal{B}_v$ (resp.\ $\mathcal{B}_w$)
the Borel $\sigma$-algebra on $\mathcal{M}_{\le 1}(\mathbb{R})$
generated by the vague (resp.\ weak) topology.
We express these $\sigma$-algebras in terms of evaluation maps.
For a bounded Borel measurable function $f$ on $\mathbb{R}$,
we denote the evaluation map $m\mapsto\int f\, dm$
by $\pi_f\colon\mathcal{M}_{\le 1}(\mathbb{R})\to\mathbb{R}$.
For a Borel set $B\in\mathcal{B}(\mathbb{R})$, we set $\pi_B:=\pi_{\bm{1}_B}$.

\begin{proposition} \label{prop:vague_Borel}
$\mathcal{B}_v$ coincides with each of the following three $\sigma$-algebras:
\begin{align*}
\Sigma_1&:=\sigma(\,\pi_f\mathrel{;}f\in C_c(\mathbb{R})\,), \\
\Sigma_2&:=\sigma(\,\pi_f\mathrel{;}f\in C_\infty(\mathbb{R})\,), \\
\Sigma_3&:=\sigma(\,\pi_B\mathrel{;}B\in\mathcal{B}(\mathbb{R})\,).
\end{align*}
Here, the symbol $\sigma(\,\pi_f\mathrel{;}f\in C_c(\mathbb{R})\,)$ stands for
the smallest $\sigma$-algebra $\Sigma$ on $\mathcal{M}_{\le 1}(\mathbb{R})$
that makes the map $\pi_f$ $\Sigma/\mathcal{B}(\mathbb{R})$-measurable for every $f\in C_c(\mathbb{R})$.
The other two are understood in the same way.
\end{proposition}

Although this proposition can be found in Kallenberg's textbooks~\cite[Lemma~A2.3]{Ka02}, \cite[Lemma~4.7]{Ka17},
we provide a proof here so that the reader can easily observe that Proposition~\ref{prop:weak_Borel} below is proved by a similar reasoning.

\begin{proof}
By definition, the vague topology on $\mathcal{M}_{\le 1}(\mathbb{R})$ is
the coarsest topology that makes $\pi_f$ continuous for every $f\in C_c(\mathbb{R})$.
This means that the subsets $\pi_f^{-1}(U)$ for open $U\subset\mathbb{R}$ and $f\in C_c(\mathbb{R})$ form a subbase of the vague topology.
Moreover, this topology is second countable by Proposition~\ref{prop:vague_cpt}.
Hence $\mathcal{B}_v=\Sigma_1$ follows from the definition of $\mathcal{B}_v$.
Similarly we have $\mathcal{B}_v=\Sigma_2$.

We next show $\Sigma_2\subset\Sigma_3$.
Let $f\in C_\infty(\mathbb{R})$.
There exists a sequence of simple functions $f_n$, $n\in\mathbb{N}$, such that
\[
\pi_f(m)=\int f\, dm
=\lim_{n\to\infty}\int f_n\, dm=\lim_{n\to\infty}\pi_{f_n}(m),
\quad m\in\mathcal{M}_{\le 1}(\mathbb{R}).
\]
Since each $\pi_{f_n}$ is $\Sigma_3/\mathcal{B}(\mathbb{R})$-measurable,
so is $\pi_f$.
Hence $\Sigma_2\subset\Sigma_3$.

It remains to prove $\Sigma_3\subset\Sigma_1$.
Let $\mathcal{D}$ be the collection of Borel sets $B$ such that $\pi_B$ is $\Sigma_1/\mathcal{B}(\mathbb{R})$-measurable.
To obtain $\Sigma_3\subset\Sigma_1$,
it suffices to show $\mathcal{D}\supset\mathcal{B}(\mathbb{R})$.

Let $U\subset\mathbb{R}$ be an open set.
Since $U$ is the increasing limit of some compact sets $K_n$, $n\in\mathbb{N}$,
we can choose a sequence $f_n\in C_c(\mathbb{R})$ such that $f_n\uparrow \bm{1}_U$ pointwise by Urysohn's lemma.
Hence
$\pi_U=\lim_{n\to\infty}\pi_{f_n}$
by the monotone convergence theorem,
which yields $U\in\mathcal{D}$.
Now the collection of open sets is a $\pi$-system; i.e., if $U_1, U_2$ are open, then so is $U_1 \cap U_2$.
On the other hand, it is easy to see that $\mathcal{D}$ is a $\lambda$-system;
$\mathbb{R}$ is in $\mathcal{D}$, $A, B\in\mathcal{D}$ with $A\subset B$ implies $B\setminus A\in\mathcal{D}$,
and $\mathcal{D}\ni A_n\uparrow A$ implies $A\in\mathcal{D}$.
Thus $\mathcal{D}\supset\sigma(\, U \mathrel{;} U\subset \mathbb{R}\ \text{is open}\,)=\mathcal{B}(\mathbb{R})$ by the Dynkin $\pi$-$\lambda$ theorem~\cite[Theorem~1.1]{Ka02}, \cite[Lemma~1.2]{Ka17} as desired.
\end{proof}

Using Proposition~\ref{prop:weak_Pol} in place of Proposition~\ref{prop:vague_cpt},
we obtain the following:

\begin{proposition} \label{prop:weak_Borel}
The following identities hold:
\[
\mathcal{B}_w=\sigma(\,\pi_f\mathrel{;}f\in C_b(\mathbb{R})\,)
=\sigma(\,\pi_B\mathrel{;}B\in\mathcal{B}(\mathbb{R})\,).
\]
\end{proposition}

\begin{corollary} \label{cor:Borel_on_M}
The identity $\mathcal{B}_v=\mathcal{B}_w$ holds.
\end{corollary}

In view of the the last corollary, we use the symbol $\mathcal{B}(\mathcal{M}_{\le 1}(\mathbb{R}))$ to denote the Borel $\sigma$-algebra $\mathcal{B}_v=\mathcal{B}_w$.

\end{document}